\documentclass{amsart}
\usepackage{amssymb}
\newtheorem{theorem}{Theorem}[section]
\newtheorem{lemma}[theorem]{Lemma}

\newtheorem{remark}[theorem]{Remark}
\newtheorem{definition}[theorem]{Definition}

\newtheorem{corollary}[theorem]{Corollary}
\newtheorem{proposition}[theorem]{Proposition}
\newcommand{\R}{\mathbb R}
\newcommand{\N}{\mathbb N}
\newcommand{\Z}{\mathbb Z}
\newcommand{\Q}{\mathbb Q}

\def\op{\operatorname}

\def\as#1{\renewcommand\arraystretch{#1}}

\def\be{\bigskip}

\def\dg{\op{deg}}

\def\dsc{\op{disc}}
\def\diso{\lower.4ex\hbox{$\downarrow$}\raise.4ex\hbox{\mbox{\scriptsize $\wr$}}}

\def\e{\medskip}

\def\ff#1{\mathbb{F}_{#1}}

\def\fph{\mathbb{F}_{\phi}}

\def\g{\Gamma}
\def\ga{\gamma}
\def\gb#1{\overline{\gamma_{#1}(\t)}}

\def\indx{\operatorname{i}}
\def\imp{\,\Longrightarrow\,}

\def\ind{\op{ind}}
\def\iso{\,\lower .6ex\hbox{$\stackrel{\lra}{\mbox{\tiny $\sim\,$}}$}\,}

\def\j{\mathbf{j}}

\def\la{\lambda}

\def\lg{l\raise.6ex\hbox to.2em{\hss.\hss}l}
\def\lra{\longrightarrow}
\def\m{{\mathfrak m}}
\def\md#1{\ \mbox{\rm(mod }{#1})}

\def\n{\mathbf{n}}

\def\nph#1{N_{\phi}(#1)}
\def\np#1{N_{#1}^-}
\def\npp#1{N_{\phi}^-(#1)}

\def\ol{{\mathcal O}_L}
\def\om{\omega}
\def\oo{{\mathcal O}}
\def\orb{\hbox to  .3em{$\backslash$}\backslash}
\def\ord{\op{ord}}

\def\pol{{\mathcal P}{\mathcal P}}

\def\qpb{\overline{\mathbb{Q}}_p}

\def\res{\op{Res}}
\def\rd{\op{red}}
\def\rdm{\rd_{L}}

\def\rt{R_1}

\def\sii{\,\Longleftrightarrow\,}

\def\ss{{\mathcal S}}

\def\t{\theta}
\def\tb{\overline{\theta}}
\def\tt{\hat{\mathbf{t}}}
\def\ty{\mathbf{t}}
\def\Ty{\mathbf{T}}

\def\tq{\,\,|\,\,}

\def\zpx{\oo[x]}

\newcounter{cs}
\stepcounter{cs}
\newcommand{\casos}{\begin{itemize}}
\newcommand{\fcasos}{\end{itemize}\setcounter{cs}{1}}

\newfont{\tit}{cmr12 scaled \magstep3}

\begin{document}
\title{Newton polygons of higher order in algebraic number theory}
\author[Gu\`ardia]{Jordi Gu\`ardia}
\address{Departament de Matem\`atica Aplicada IV, Escola Polit\`ecnica Superior d'Enginyera de Vilanova i la Geltr\'u, Av. V\'\i ctor Balaguer s/n. E-08800 Vilanova i la Geltr\'u, Catalonia}
\email{guardia@ma4.upc.edu}

\author[Montes]{\hbox{Jes\'us Montes}}
\address{Departament de Ci\`encies Econ\`omiques i Socials,
Facultat de Ci\`encies Socials,
Universitat Abat Oliba CEU,
Bellesguard 30, E-08022 Barcelona, Catalonia, Spain}
\email{montes3@uao.es}

\author[Nart]{\hbox{Enric Nart}}
\address{Departament de Matem{\`a}tiques,
         Universitat Aut{\`o}noma de Barcelona,
         Edifici C, E-08193 Bellaterra, Barcelona, Catalonia}
\email{nart@mat.uab.cat}
\thanks{Partially supported by MTM2006-15038-C02-02 and MTM2006-11391 from the Spanish MEC}
\date{}
\keywords{Newton polygon, local field, prime ideal decomposition, discriminant, integral basis}

\subjclass[2000]{Primary 11S15; Secondary 11R04, 11R29, 11Y40}

\maketitle
\begin{abstract}
We develop a theory of arithmetic Newton polygons of higher order, that provides the factorization of a separable polynomial over a $p$-adic field, together with relevant arithmetic information about the fields generated by the irreducible factors. This carries out a program suggested by \O{}. Ore. As an application, we obtain fast algorithms to compute discriminants, prime ideal decomposition and integral bases of number fields.
\end{abstract}

\section*{Introduction}
R. Dedekind based the foundations of algebraic number theory on ideal theory, because the constructive attempts to find a rigorous general definition of the \emph{ideal numbers} introduced by E. Kummer failed. This failure is due to the existence of inessential discriminant divisors; that is, there are number fields $K$ and prime numbers $p$, such that $p$ divides the index $\indx(\t):=\left(\Z_K\colon \Z[\t]\right)$, for any integral generator $\t$ of $K$, where $\Z_K$ is the ring of integers.  Dedekind gave a criterion to detect when $p\nmid\indx(\t)$, and a procedure to construct the prime ideals of $K$ dividing $p$ in that case, in terms of the factorization of the minimal polynomial of $\t$ modulo $p$ \cite{de}. 

M. Bauer introduced an arithmetic version of Newton polygons to construct prime ideals in cases where Dedekind's criterion failed
\cite{bau}. This theory was developed and extended by \O{}. Ore in his 1923 thesis and a series of papers that followed \cite{ore1,ore2,ore3,ore4,ore5}. Let $f(x)\in\Z[x]$ be an irreducible polynomial that generates $K$. After K. Hensel's work, the prime ideals of $K$ lying above $p$ are in bijection with the irreducible factors of $f(x)$ over $\Z_p[x]$. Ore's work determines three successive factorizations of $f(x)$ in $\Z_p[x]$, known as the \emph{three classical dissections} \cite{ber}, \cite{coh}. The first dissection is determined by Hensel's lemma: $f(x)$ splits into the product of factors that are congruent to the power of an irreducible polynomial modulo $p$. The second dissection is a further splitting of each factor, according to the number of sides of certain Newton polygon. The third dissection is a further splitting of each of the late factors, according to the factorization of certain \emph{residual polynomial} attached to each side of the polygon, which is a polynomial with coefficients in a finite field.

Unfortunately, the factors of $f(x)$ obtained after these three dissections are not always irreducible.
Ore defined a  polynomial to be \emph{$p$-regular} when it satisfies a technical condition that ensures that the factorization of $f(x)$ is complete after the three dissections. Also, he proved the existence of a $p$-regular defining equation for every number field, but the proof is not constructive: it uses the Chinese remainder theorem with respect to the different prime ideals that one wants to construct. Ore himself suggested that it should be possible to introduce Newton polygons of higher order that continue the factorization process till all irreducible factors of $f(x)$ are achieved \cite[Ch.4,\S8]{ore1}, \cite[\S5]{ore5}.

Ore's program was carried out by the second author in his 1999 thesis \cite{m}, under the supervision of the third author.
For any natural number $r\ge1$, Newton polygons of order $r$ were constructed, the case $r=1$ corresponding to the Newton polygons
introduced by Ore. Also, analogous to Ore's theorems were proved for polygons of order $r$, providing two more dissections of the factors of $f(x)$, for each order $r$. The whole process is controled by an invariant defined in terms of \emph{higher order indices}, that ensures that the process ends after a finite number of steps. Once an irreducible factor of $f(x)$ is detected, the theory determines the ramification index and residual degree of the $p$-adic field generated by this factor, and a generator of the maximal ideal. These invariants are expressed in terms of combinatorial data attached to the sides of the higher order polygons and the residual polynomials of higher order attached to each side. The process yields as a by-product a computation of $\ind(f):=v_p(\indx(\t))$, where $\t$ is a root of $f(x)$. An implementation in Mathematica of this factorization algorithm was worked out by the first author \cite{gua}.   

We present these results for the first time in the form of a publication, after a thorough revision and some simplifications.
In section 1 we review Ore's results, with proofs, which otherwise can be found only in the original papers by Ore in the language of ``h\"oheren Kongruenzen". In section 2     
we develop the theory of Newton polygons of higher order, based in the concept of a \emph{type} and its \emph{representative}, which plays the analogous role in order $r$ to that played by an irreducible polynomial modulo $p$ in order one. In section 3 we prove analogous in order $r$ to Ore's Theorems of the polygon and of the residual polynomial (Theorems \ref{thpolygonr} and \ref{thresidualr}), that provide two more dissections for each order. In section 4 we introduce resultants and indices of higher order and we prove the Theorem of the index (Theorem \ref{thindex}), that relates $\ind(f)$ with the higher order indices constructed from the higher order polygons. This result guarantees that the factorization process finishes at most in  $\ind(f)$ steps.  

Although the higher order Newton polygons are apparently involved and highly technical objects, they provide fast factorization algorithms, because all computations are mainly based on two reasonably fast operations: division with remainder of monic polynomials with \emph{integer} coefficients, and factorization of polynomials over \emph{finite} fields. Thus, from a modern perspective, the main application of these results is the design of fast algorithms to compute discriminants, prime ideal decomposition and integral bases of number fields. However, we present in this paper only the theoretical background of higher order Newton polygons. We shall describe the concrete design of the algorithms and discuss the relevant computational aspects elsewhere \cite{gmna,gmnb}.    

\tableofcontents

\section{Newton polygons of the first order}\label{secNP}
\subsection{Abstract polygons}\label{abstract}
Let $\la\in \Q^-$ be a negative rational number, expressed in lower terms as $\la=-h/e$, with $h,e$ positive coprime integers. We denote by $\ss(\la)$ the set of segments of the Euclidian plane with slope $\la$ and end points having nonnegative integer coordinates. The points of $(\Z_{\ge0})^2$ are also considered to be segments in $\ss(\la)$, whose initial and final points coincide. The elements of $\ss(\la)$ will be called \emph{sides of slope $\la$}. For any side $S\in\ss(\la)$, we define its \emph{length}, $\ell:=\ell(S)$, and \emph{height}, $H:=H(S)$, to be the length of the respective  projections of $S$ to the horizontal and vertical axis. We define the \emph{degree} of $S$ to be 
$$
d:=d(S):=\ell(S)/e=H(S)/h.
$$
Note that any side $S$ of positive length is divided into $d$ segments by the
points of integer coordinates that lie on  $S$. 
A side $S\in\ss(\la)$ is determined by the initial point $(s,u)$ and the degree $d$. The final point is $(s+\ell,u-H)=(s+de,u-dh)$. For instance, the next figure represents a side of slope $-1/2$, initial point $(s,u)$, and degree three.

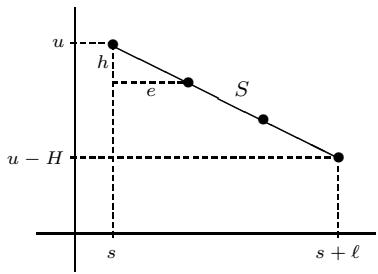
\begin{center}
\setlength{\unitlength}{5.mm}
\begin{picture}(8,6)
\put(.85,4.85){$\bullet$}\put(2.85,3.85){$\bullet$}\put(4.85,2.85){$\bullet$}\put(6.85,1.85){$\bullet$}
\put(-1,0){\line(1,0){9}}\put(0,-1){\line(0,1){7}}
\put(1,5){\line(2,-1){6}}\put(1.02,5){\line(2,-1){6}}
\put(.85,-.65){\begin{scriptsize}$s$\end{scriptsize}}
\put(-.6,4.95){\begin{scriptsize}$u$\end{scriptsize}}
\put(-1.8,1.85){\begin{scriptsize}$u-H$\end{scriptsize}}
\multiput(-.1,5.05)(.25,0){5}{\hbox to 2pt{\hrulefill }}
\multiput(-.1,2)(.25,0){29}{\hbox to 2pt{\hrulefill }}
\multiput(1,4)(.25,0){9}{\hbox to 2pt{\hrulefill }}
\multiput(7,-.1)(0,.25){9}{\vrule height2pt}
\put(6.4,-.65){\begin{scriptsize}$s+\ell$\end{scriptsize}}
\multiput(1,-.1)(0,.25){21}{\vrule height2pt}
\put(4.25,3.65){\begin{footnotesize}$S$\end{footnotesize}}
\put(.6,4.4){\begin{scriptsize}$h$\end{scriptsize}}
\put(1.9,3.65){\begin{scriptsize}$e$\end{scriptsize}}
\end{picture}
\end{center}\be\be

The set $\ss(\la)$ has the structure of an abelian semigroup with the following addition rule: given $S,\,T\in \ss(\la)$, the sum $S+T$ is the side of degree $d(S)+d(T)$ of $\ss(\la)$, whose initial point is the sum of the initial points of $S$ and $T$. Thus, the addition is geometrically represented by the process of joining the two segments and choosing an apropriate initial point. The addition of a segment $S$ with a point $P$ is represented by the translation $P+S$ of $S$ by the vector represented by $P$. The neutral element is the point $(0,0)$. The invariants $\ell(S)$, $H(S)$, $d(S)$ determine semigroup homomorphisms
$$
\ell,\,H,\,d\,\colon \ss(\la)\lra \Z_{\ge0}.
$$

For technical reasons we consider also a set of \emph{sides of slope $-\infty$}, which is formally defined as $\ss(-\infty):=\Z_{>0}\times (\Z_{\ge0})^2$. If $S=(\ell,(s,u))$ is a side of slope minus infinity, we define $\ell(S):=\ell$, $H(S):=\infty$, $d(S):=1$. Also, we take by convention $h=\infty$, $e=\ell$.  This set has an obvious structure of an abelian monoid, and the length  determines a monoid homomorphism, $\ell\colon \ss(-\infty)\lra \Z_{>0}$.
There is a geometric representation of such an $S$ as a side whose end points are $(s-\ell,\infty)$ and $(s,u)$. \be

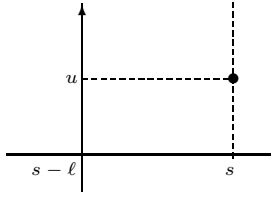
\begin{center}
\setlength{\unitlength}{5.mm}
\begin{picture}(6,4)
\put(3.85,1.85){$\bullet$}\put(-2,0){\line(1,0){7}}
\put(0,-1){\vector(0,1){5}}
\put(3.8,-.55){\begin{scriptsize}$s$\end{scriptsize}}
\put(-1.35,-.55){\begin{scriptsize}$s-\ell$\end{scriptsize}}
\put(-.45,1.9){\begin{scriptsize}$u$\end{scriptsize}}
\multiput(4,-.1)(0,.25){17}{\vrule height2pt}
\multiput(0,2)(.25,0){16}{\hbox to 2pt{\hrulefill }}
\end{picture}
\end{center}\be\be

The \emph{set of sides of negative slope} is defined as the formal disjoint union
$$
\ss:=\ss(-\infty)\coprod\left(\bigcup_{\la\in\Q^-}\ss(\la)\right).
$$
Note that the points of $(\Z_{\ge0})^2$ belong to $\ss(\la)$ for all finite $\la$, so that it is not possible (even in a formal sense) to attach a slope to them. 

We have a natural geometric representation of a side. Let us introduce a geometric representation of a formal sum of sides as an open convex polygon of the plane. Let $N=S_1+\cdots+S_t$ be a formal sum of sides of negative slope. Let $S_{\infty}=(\ell_{\infty},P_{\infty})$ be the sum of all sides of slope $-\infty$ among the $S_i$, and let $P_0$ be the sum of all initial points of the $S_i$ that don't belong to $\ss(-\infty)$ (in case of an empty sum we consider respectively $P_{\infty}=(0,0)$ and/or $P_0=(0,0)$). Let $P=P_{\infty}+P_0$. Then, $N$ is represented as the polygon that starts at $P$ and is obtained by joining all sides of positive length and finite slope, ordered
by increasing slopes. If $i_1$ is the abscissa of $P$, we have to think that the polygon starts at the abscisa $i_0=i_1-\ell_{\infty}$, that formally indicates the starting point (at infinity) of a side of slope $-\infty$. The typical shape of this polygon is

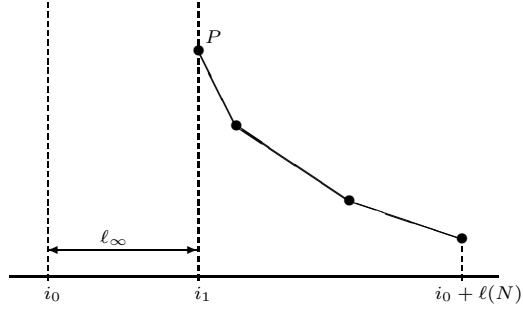
\begin{center}
\setlength{\unitlength}{5.mm}
\begin{picture}(12,8)
\put(10.85,.85){$\bullet$}\put(7.85,1.85){$\bullet$}
\put(4.85,3.85){$\bullet$}\put(3.85,5.85){$\bullet$}
\put(-1,0){\line(1,0){13}}
\put(8,2){\line(-3,2){3}}\put(5,4){\line(-1,2){1}}\put(8,2.03){\line(-3,2){3}}
\put(5,4.03){\line(-1,2){1}}\put(11,1){\line(-3,1){2}}\put(11.02,1){\line(-3,1){3}}
\multiput(0,-.1)(0,.25){30}{\vrule height2pt}
\multiput(4,-.1)(0,.25){30}{\vrule height2pt}
\multiput(11,-.1)(0,.25){5}{\vrule height2pt}
\put(2,.7){\vector(1,0){2}}\put(2,.7){\vector(-1,0){2}}
\put(1.4,.9){\begin{scriptsize}$\ell_{\infty}$\end{scriptsize}}
\put(4.2,6.2){\begin{scriptsize}$P$\end{scriptsize}}
\put(-.1,-.6){\begin{scriptsize}$i_0$\end{scriptsize}}
\put(3.9,-.6){\begin{scriptsize}$i_1$\end{scriptsize}}
\put(10.3,-.6){\begin{scriptsize}$i_0+\ell(N)$\end{scriptsize}}
\end{picture}
\end{center}\be\be

\begin{definition}
The semigroup $\pol$ of principal polygons is defined to be the set of all these geometric configurations.
\end{definition}

By definition, every principal polygon represents a formal sum, $N=S_1+\cdots+S_t$, of sides $S_i\in\ss$. This expression is unique in any of the two following situations
\begin{enumerate}
 \item $N=S$, with $S\in(\Z_{\ge0})^2$,
\item $N=S_1+\cdots+S_t$, with all $S_i$ of positive length and pairwise different slopes.
\end{enumerate}
It is clear that any $N\in\pol$ can be expressed in one (and only one) of these canonical forms. 
Usually, when we speak of the \emph{sides} of a principal polygon, we mean the sides of this canonical expression. If we need to emphasize this we shall use the term \emph{canonical sides} of $N$. The finite end points of the canonical sides are called the \emph{vertices} of the polygon.

The addition of polygons is defined in terms of the expression as a formal sum of sides (not necessarily the canonical ones). That is, if 
$N=S_1+\cdots+S_r$ and $N'=S'_1+\cdots+ S'_s$, then $N+N'$ is the geometric representation of $S_1+\cdots+S_r+S'_1+\cdots +S'_s$. The reader may check easily that this is well-defined and $\pol$ has a structure of semigroup with neutral element $(0,0)$.

Also, it is clear that this addition is compatible with the sum operations that we had on all $\ss(\la)$. Note that the addition of $N\in\pol$ with (the polygon represented by) a point $P\in(\Z_{\ge0})^2$ is the translation $P+N$. The fact of adding to $N$ (the polygon represented by) a side of slope $-\infty$ is reflected by a horizontal shift of the finite part of $N$, without changing the starting abscissa $i_0$ of $N$.

\begin{definition}
We define the length of a principal polygon $N=S_1+\cdots+S_r$ to be $\ell(N):=\ell(S_1)+\cdots+\ell(S_r)$.
The length determines a semigroup homomorphism, $\ell\colon \pol\lra \Z_{\ge 0}$.
\end{definition}

Let $N\in\pol$. Let $i_0$ be the abscissa where the polygon starts and $i_1$ the abscissa of the point $P$ where the finite part of $N$ starts. For any integer abscissa $i_0 \le i\le i_0+\ell(N)$ we denote by
$$y_i=y_i(N)=\left\{\begin{array}{ll}
 \infty,&\mbox{ if }i<i_1,\\
\mbox{the ordinate of the point of $N$ of abscissa $i$,}&\mbox{ if }i\ge i_1.
\end{array}\right.
$$  For $i\ge i_1$ these rational numbers form an strictly decreasing sequence.

\begin{definition}\label{above}
Let $P=(i,y)$ be a ``point" of the plane, with  integer abscissa $i_0 \le i\le i_0+\ell(N)$, and ordinate $y\in\R\cup\{\infty\}$. We say
that $P$ \emph{lies on} $N$ if $y=y_i$, and in this case we write $P\in N$.
 We say that $P$ \emph{lies above} $N$ if $y\ge y_i$. We say that $P$
\emph{lies strictly above} $N$ if $y>y_i$.
\end{definition}

For any $i_1<i\le i_0+\ell(N)$, let $\mu_i$ be the slope of the segment joining $(i-1,y_{i-1})$ and $(i,y_i)$. The sequence
$\mu_{i_1+1}\le\cdots\le\mu_{i_0+\ell(N)}$ is an increasing sequence of negative rational numbers. We call these elements the \emph{unit slopes} of $N$. Consider the multisets of unit slopes:
$$
U_{i_1}(N):=\emptyset;\quad U_i(N):=\{\mu_{i_1+1},\dots,\mu_i\}, \quad\forall i_1<i\le i_0+\ell(N). 
$$
Clearly, $y_i(N)=y_{i_1}(N)+\sum_{\mu\in U_i(N)}\mu$.

Let $N'$ be another principal polygon with starting abscissa $j_0$ and starting abscissa for the finite part $j_1$. Consider  analogous multisets $U_j(N')$, for all $j_1\le j\le j_0+\ell(N')$. By the definition of the addition law of principal polygons, the multiset $U_k(N+N')$ contains the smallest $k-i_1-j_1$ unit slopes of the multiset $U_{i_0+\ell(N)}(N)\cup U_{j_0+\ell(N')}(N')$ that contains all unit slopes of both polygons. Thus, 
$$
y_i(N)+y_j(N')\ge y_{i+j}(N+N'),
$$
and equality holds if and only if $U_i(N)\cup U_j(N')=U_{i+j}(N+N')$.

\begin{lemma}\label{sum}
Let $N,N'\in\pol$. Let $P=(i,u)$ be a point lying above the finite part of $N$ and
$P'=(j,u')$ a point lying above the finite part of $N'$. Then $P+P'$ lies
above the finite part of $N+N'$ and $$P+P'\in N+N'\ \sii\ P\in N,\ P'\in N',\mbox{
and }\,U_i(N)\cup U_j(N')=U_{i+j}(N+N').$$
\end{lemma}

\begin{proof}
Clearly, $u+u'\ge y_i(N)+y_j(N')\ge y_{i+j}(N+N')$ and $P+P'\in
N+N'$ if and only if both inequalities are equalities.
\end{proof}

\begin{definition}\label{sla}
Let $\la\in\Q^-$ and $N\in\pol$. Consider a line of slope $\la$ far below $N$ and let it shift upwards till it touches $N$ for the first time. Denote by $L_{\la}(N)$ this line of slope $\la$ having first contact with $N$. We define the \emph{$\la$-component} of $N$ to be $S_{\la}(N):=N\cap L_{\la}(N)$. We obtain in this way a map:
$$
S_{\la}\colon \pol\lra \ss(\la).
$$
\end{definition}
If $N$ has a canonical side $S$ of positive length and finite slope $\la$, we have $S_{\la}(N)=S$, otherwise the $\la$-component $S_{\la}(N)$ reduces to a point.

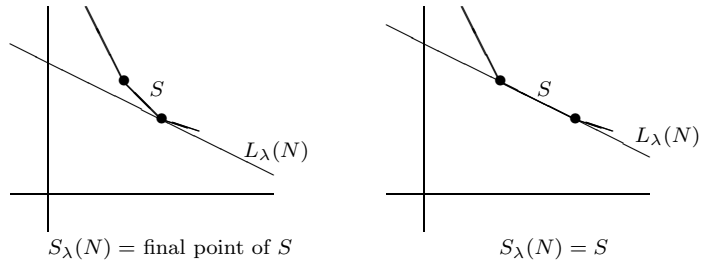
\begin{center}
\setlength{\unitlength}{5.mm}
\begin{picture}(15,5.5)
\put(2.85,1.85){$\bullet$}\put(1.85,2.85){$\bullet$}
\put(-1,0){\line(1,0){7}}\put(0,-1){\line(0,1){6}}
\put(3,2){\line(-1,1){1}}\put(3.02,2){\line(-1,1){1}}
\put(3,2){\line(3,-1){1}}\put(3.02,2){\line(3,-1){1}}
\put(2,3){\line(-1,2){1}}\put(2.02,3){\line(-1,2){1}}
\put(6,.5){\line(-2,1){7}}\put(5.2,1){\begin{footnotesize}$L_{\la}(N)$\end{footnotesize}}
\put(2.7,2.6){\begin{footnotesize}$S$\end{footnotesize}}
\put(0,-1.6){\begin{footnotesize}$S_{\la}(N)=$ final point of
$S$\end{footnotesize}}
\put(13.85,1.85){$\bullet$}\put(11.85,2.85){$\bullet$}
\put(9,0){\line(1,0){7}}\put(10,-1){\line(0,1){6}}
\put(14,2){\line(-2,1){2}}\put(14.02,2){\line(-2,1){2}}
\put(14,2){\line(3,-1){1}}\put(14.02,2){\line(3,-1){1}}
\put(12,3){\line(-1,2){1}}\put(12.02,3){\line(-1,2){1}}
\put(16,1){\line(-2,1){7}}\put(15.6,1.4){\begin{footnotesize}$L_{\la}(N)$\end{footnotesize}}
\put(13,2.6){\begin{footnotesize}$S$\end{footnotesize}}
\put(12,-1.6){\begin{footnotesize}$S_{\la}(N)=S$\end{footnotesize}}
\end{picture}
\end{center}\be\be

Lemma \ref{sum} shows that $S_{\la}$ is a semigroup homomorphism:
\begin{equation}\label{sumsla}
S_{\la}(N+N')=S_{\la}(N)+S_{\la}(N'),
\end{equation}
for all $N,\,N'\in\pol$ and all $\la\in\Q^-$.

\subsection{$\phi$-Newton polygon of a polynomial}\label{phiN}
Let $p$ be a prime number and let $\qpb$ be a fixed algebraic closure of the field $\Q_p$ of the $p$-adic numbers. For any finite extension, $\Q_p\subseteq L\subseteq \qpb$,  of $\Q_p$ we denote by $v_L$ the $p$-adic valuation, $v_L\colon \qpb\lra \Q\cup\{\infty\}$, normalized by $v_L(L^*)=\Z$. Throughout the paper $\ol$ will denote the ring of integers of $L$, $\m_L$ its maximal ideal, and $\ff{L}$ the residue field. The canonical reduction map $\rdm\colon \ol\lra\ff{L}$ will be usually indicated by a bar: $\overline{\alpha}:=\rdm(\alpha)$.

We fix a finite extension $K$ of $\Q_p$ as a base field, and we denote $v:=v_K$, $\oo:=\oo_K$, $\m:=\m_K$, $\ff{}:=\ff{K}$, $q:=|\ff{}|$. We fix also a prime element $\pi\in\oo$.

We extend the valuation $v$ to polynomials with coefficients in $\oo$ in a natural way:
$$
v\colon \zpx\lra \Z_{\ge0}\cup\{\infty\},\quad v(b_0+\cdots+b_rx^r):=\min\{v(b_j),\,0\le j\le r\}.
$$

Let $\phi(x)\in\zpx$ be a monic polynomial of degree $m$ whose reduction mo\-dulo $\m$ is irreducible. We denote by $\fph$ the finite field $\zpx/(\pi,\phi(x))$, and by $\rd\colon \zpx\lra \fph$
the canonical homomorphism. We denote also by a bar the reduction of polynomials modulo $\m$, $\bar{\ }\colon\zpx\lra\ff{}[x]$.

Any $f(x)\in\zpx$ admits a unique $\phi$-adic development:
$$ 
f(x)=a_0(x)+a_1(x)\phi(x)+\cdots+a_n(x)\phi(x)^n,
$$with $a_i(x)\in\zpx$, $\dg a_i(x)<m$. For any coefficient $a_i(x)$ we denote $u_i:=v(a_i(x))\in\Z_{\ge0}\cup\{\infty\}$ and we attach to $a_i(x)$ the point $P_i=(i,u_i)$, which is a point of the plane if $u_i$ is finite, and it is thought to be the point at infinity of the vertical line with abscissa $i$, if $u_i=\infty$.

\begin{definition}
The $\phi$-Newton polygon of a nonzero polynomial $f(x)\in\zpx$ is the lower convex envelope of the set of points $P_i=(i,u_i)$, $u_i<\infty$, in the euclidian plane.
We denote this polygon by $\nph{f}$.
\end{definition}

The length of this polygon is by definition the abscissa of the last vertex. We denote it by  $\ell(\nph{f}):=n=\lfloor \dg(f)/m\rfloor$. Note that $\dg f(x)=m n+\dg a_n(x)$. The typical shape of this polygon is the following

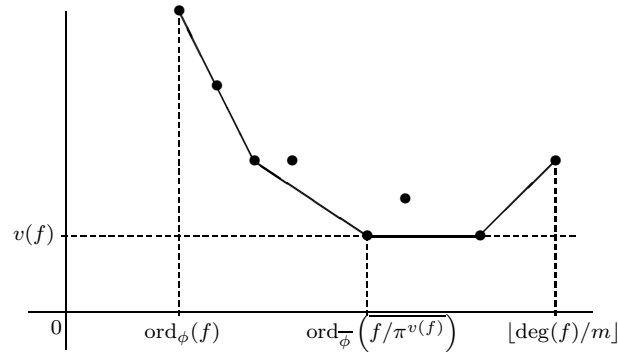
\begin{center}
\setlength{\unitlength}{5.mm}
\begin{picture}(14,9)
\put(10.85,1.85){$\bullet$}\put(8.85,2.85){$\bullet$}\put(7.85,1.85){$\bullet$}\put(5.85,3.85){$\bullet$}
\put(4.85,3.85){$\bullet$}\put(3.85,5.85){$\bullet$}\put(2.85,7.85){$\bullet$}\put(-1,0){\line(1,0){15}}
\put(0,-1){\line(0,1){9}}
\put(8,2){\line(-3,2){3}}\put(5,4){\line(-1,2){2}}\put(8,2.03){\line(-3,2){3}}
\put(5,4.03){\line(-1,2){2}}\put(11,2){\line(-1,0){3}}\put(11,2.02){\line(-1,0){3}}
\put(12.85,3.85){$\bullet$}\put(11,2){\line(1,1){2}}\put(11,2.02){\line(1,1){2}}
\multiput(13,-.1)(0,.25){16}{\vrule height2pt}
\multiput(3,-.1)(0,.25){32}{\vrule height2pt}
\multiput(8,-.1)(0,.25){9}{\vrule height2pt}
\put(11.7,-.7){\begin{footnotesize}$\lfloor \dg(f)/m\rfloor$\end{footnotesize}}
\put(6.4,-.7){\begin{footnotesize}$\ord_{\overline{\phi}}\left(\overline{f/\pi^{v(f)}}\right)$\end{footnotesize}}
\put(2.1,-.7){\begin{footnotesize}$\ord_{\phi}(f)$\end{footnotesize}}
\multiput(-.1,2)(.25,0){55}{\hbox to 2pt{\hrulefill }}
\put(-1.4,1.9){\begin{footnotesize}$v(f)$\end{footnotesize}}
\put(-.4,-.6){\begin{footnotesize}$0$\end{footnotesize}}
\end{picture}
\end{center}\be\be

\begin{remark}\label{phiell}
The $\phi$-Newton polygon of $f(x)$ consists of a single side of slope $-\infty$ if and only if $f(x)=a(x)\phi(x)^n$, with $\dg(a)<m$. 
\end{remark}

\begin{definition}\label{shape1}
The principal $\phi$-polygon of $f(x)$ is the element $\npp{f}\in\pol$ determined by the sides of negative slope  of $\nph{f}$, including the side of slope $-\infty$ represented by the length $\ord_{\phi}(f)$. It always starts at the abscissa $i_0=0$ and has length $\ord_{\overline{\phi}}\left(\overline{f/\pi^{v(f)}}\right)$. 

For any $\la\in\Q^-$ we shall denote by $S_{\la}(f):=S_{\la}(\npp{f})$ the $\la$-component of this polygon (cf. Definition \ref{sla}) 
\end{definition}

From now on, we denote $N=\npp{f}$ for simplicity. The principal polygon $N$ and the set of points $P_i=(i,u_i)$ that lie on $N$, contain the arithmetic information we are interested in. 
Note that, by construction, the points $P_i$ lie all above $N$.

We attach to any abscissa $\ord_{\phi}(f)\le i\le \ell(N)$ the following \emph{residual coefficient} $c_i\in\fph$:
$$\as{1.6}
c_i=\left\{\begin{array}{ll}
0,&\mbox{ if $(i,u_i)$ lies strictly above }N,\\\rd\left(\dfrac{a_i(x)}{\pi^{u_i}}\right),&\mbox{ if $(i,u_i)$ lies on }N.
\end{array}
\right.
$$ Note that $c_i$ is always nonzero in the latter case, because $\dg a_i(x)<m$. 

Let $\la=-h/e$ be a negative rational number, with $h,e$ positive coprime integers. Let $S=S_{\la}(f)$ be the $\la$-component of $N$, $(s,u)$ the initial point of $S$, and $d:=d(S)$ the degree of $S$. The points $(i,u_i)$ that lie on $S$ contain important arithmetic information that is
kept in the form of two polynomials that are built with the
coefficients of the $\phi$-adic development of $f(x)$ to whom these points are attached.

\begin{definition}
We define the \emph{virtual factor} of $f(x)$ attached to $\la$ (or to $S$) to be the polynomial
$$
f^S(x):=\pi^{-u}\phi(x)^{-s}f^0(x)\in K[x],\ \mbox{ where }\ f^0(x):=\sum_{(i,u_i)\in S}a_{i}(x)\phi(x)^i.
$$
We define the \emph{residual polynomial} attached to $\la$ (or to $S$) to be the polynomial:
$$
R_{\la}(f)(y):=c_s+c_{s+e}\,y+\cdots+c_{s+(d-1)e}\,y^{d-1}+c_{s+de}\,y^d\in\fph[y].
$$
\end{definition}

Note that only the points $(i,u_i)$ that lie on $S$ yield a nonzero coefficient of $R_{\la}(f)(y)$. In particular, $c_s$ and $c_{s+de}$ are always nonzero, so that $R_{\la}(f)(y)$ has degree $d$ and it is never  divisible by $y$.

If $\pi'=\rho\pi$ is another prime element of $\oo$, and $c=\bar{\rho}\in\ff{}^*$, the residual coefficients of $N_{\phi}^-(f)$ with respect to $\pi'$ satisfy $c'_i=c_ic^{-u_i}$, so that the corresponding residual polynomial $R'_{\la}(f)(y)$ is equal to $c^{-u}R_{\la}(f)(c^hy)$.

We can define in a completely analogous way the residual polynomial of $f(x)$ with respect to a side $T$, which is not necessarily a $\la$-component of $\npp{f}$. 

\begin{definition}
Let $T\in\ss(\la)$ be an arbitrary side of slope $\la$, with abscissas $s_0\le s_1$ for the end points, and let $d'=d(T)$. We say that the polynomial $f(x)$ \emph{lies above} $T$ if all points of $\npp{f}$ with integer abscissa $s_0\le i\le s_1$ lie above $T$; in this case we define
$$
R_{\la}(f,T)(y):=\tilde{c}_{s_0}+\tilde{c}_{s_0+e}\,y+\cdots+\tilde{c}_{s_0+(d'-1)e}\,y^{d'-1}+\tilde{c}_{s_0+d'e}\,y^{d'}\in\fph[y],
$$
where $\tilde{c}_i=c_i$ if $(i,u_i)$ lies on $T$ and $\tilde{c}_i=0$ otherwise. 
\end{definition}
 Thus, if all points of $S_{\la}(f)$ lie strictly above $T$ we have $R_{\la}(f,T)(y)=0$.
Note that $\dg R_{\la}(f,T)(y)\le d'$ and equality holds if and only if the final point of $T$ belongs to $S_{\la}(f)$. Usually, $T$ will be an enlargement of $S_{\la}(f)$ and then,
\begin{equation}\label{enlarge1}
T\supseteq S_{\la}(f)\ \imp\ R_{\la}(f,T)(y)=y^{(s-s_0)/e}R_{\la}(f)(y), 
\end{equation}
where $s$ is the abscissa of the initial point of $S_{\la}(f)$.

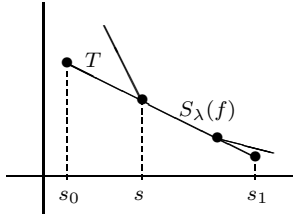
\begin{center}
\setlength{\unitlength}{5.mm}
\begin{picture}(6,5)
\put(4.85,.35){$\bullet$}\put(3.85,.85){$\bullet$}\put(1.85,1.85){$\bullet$}\put(-.15,2.85){$\bullet$}
\put(-1.6,0){\line(1,0){7.6}}\put(-.6,-1){\line(0,1){5.6}}
\put(5,.5){\line(-2,1){5}}\put(5.02,.5){\line(-2,1){5}}
\put(4,1){\line(4,-1){1.5}}\put(4.02,1){\line(4,-1){1.5}}
\put(2,2){\line(-1,2){1}}\put(2.02,2){\line(-1,2){1}}
\put(3,1.6){\begin{footnotesize}$S_{\la}(f)$\end{footnotesize}}
\put(.5,2.9){\begin{footnotesize}$T$\end{footnotesize}}
\multiput(0,-.1)(0,.25){13}{\vrule height2pt}
\multiput(2,-.1)(0,.25){9}{\vrule height2pt}
\multiput(5,-.1)(0,.25){3}{\vrule height2pt}
\put(-.2,-.65){\begin{footnotesize}$s_0$\end{footnotesize}}
\put(1.8,-.65){\begin{footnotesize}$s$\end{footnotesize}}
\put(4.8,-.65){\begin{footnotesize}$s_1$\end{footnotesize}}
\end{picture}
\end{center}\be\be\be

The motivation for this more general definition lies in the bad behaviour of the residual polynomial $R_{\la}(f)(y)$ with respect to sums. Nevertheless, if $T$ is a fixed side and $f(x)$, $g(x)$ lie both above $T$, it is clear that $f(x)+g(x)$ lies above $T$ and 
\begin{equation}\label{sumR}
R_{\la}(f+g,T)(y)=R_{\la}(f,T)(y)+R_{\la}(g,T)(y).
\end{equation}

\subsection{Admissible $\phi$-developments and Theorem of the product}\label{thproduct}
Let
\begin{equation}\label{phdev}
f(x)=\sum_{i\ge 0}a'_i(x)\phi(x)^i,\quad a'_i(x)\in\zpx,
\end{equation}
be a $\phi$-development of $f(x)$, not necessarily the $\phi$-adic
one. Take $u'_i=v(a'_i(x))$, and let $N'$ be the
principal polygon of the set of points $(i,u'_i)$. Let $i_1$ be the first abscissa with $a'_{i_1}(x)\ne0$. To any $i_1\le i\le \ell(N')$ we attach a residual coefficient as before:
 $$\as{1.6}
c'_i=\left\{\begin{array}{ll}
0,&\mbox{ if $(i,u'_i)$ lies strictly above }N',\\\rd\left(\dfrac{a'_i(x)}{\pi^{u'_i}}\right),&\mbox{ if $(i,u'_i)$ lies on }N'
\end{array}
\right.
$$ 
For the points $(i,u'_i)$ lying on $N'$ we can have now $c'_i=0$; for instance, in the case $a'_0(x)=f(x)$, the Newton polygon has only one point $(0,v(f))$ and $c'_0=0$ if $f(x)/\pi^{v(f)}$ is divisible by $\phi(x)$ modulo $\m$.

Finally, for any negative rational number $\la=-h/e$ as above, we can define the \emph{residual polynomial} attached to the $\la$-component $S'=S_{\la}(N')$ to be 
$$
R'_{\la}(f)(y):=c'_{s'}+c'_{s'+e}\,y+\cdots+c'_{s'+(d'-1)e}\,y^{d'-1}+c'_{s'+d'e}\,y^{d'}\in\fph[y],
$$where  $d'=d(S')$ and $s'$ is the abscissa of the initial point of $S'$.

\begin{definition}
We say that the $\phi$-development (\ref{phdev}) is admissible if
for each abs\-cissa $i$ of a vertex of $N'$
we have $c'_i\ne0$.
\end{definition}

\begin{lemma}\label{admissible}
If a $\phi$-development is admissible, then
$N'=\npp{f}$ and $c'_i=c_i$ for all abscissas $i$ of the finite part of $N'$. In particular, for any negative rational number $\la$ we have $R'_{\la}(f)(y)=R_{\la}(f)(y)$.
\end{lemma}

\begin{proof}
Consider the $\phi$-adic developments of $f(x)$ and each $a'_i(x)$:
$$
f(x)=\sum_{0\le i}a_i(x)\phi(x)^i,\qquad a'_i(x)=\sum_{0\le k}b_{i,k}(x)\phi(x)^k.
$$
By the uniqueness of the $\phi$-adic development we have
\begin{equation}\label{unique}
a_i(x)=\sum_{0\le k\le i}b_{i-k,k}(x).
\end{equation}
Clearly, $w_{i,k}:=v(b_{i,k})\ge u'_i$, for all $0\le k$, $0\le i\le \ell(N')$. In particular, all points $(i,u_i)$ lie above $N'$; in fact, for some $0\le k_0\le i$, we have
\begin{equation}\label{ui}
u_i=v(a_i)\ge \min_{0\le k\le i}\{w_{i-k,k}\}=w_{i-k_0,k_0}\ge u'_{i-k_0}\ge y_{i-k_0}(N')\ge  y_i(N').
\end{equation}
From now on, $i$ will be an integer abscissa of the finite part of $N'$. Clearly, 
\begin{equation}\label{clear}
 w_{i-k,k}\ge u'_{i-k}\ge y_{i-k}(N')>y_i(N'),
\end{equation}
for any $0<k\le i$. Also, for the abscissas $i$ with $u'_i=y_i(N')$ we have 
\begin{equation}\label{cprima}
c'_i=\rd(a'_i(x)/\pi^{u'_i})=\rd(b_{i,0}(x)/\pi^{u'_i}).
\end{equation}

Now, if $(i,u'_i)$ is a vertex of $N'$ we have $c'_i\ne0$ by hypothesis, and from (\ref{cprima}) 
we get $y_i(N')=u'_i=w_{i,0}$. By (\ref{clear}) and (\ref{unique}) we have $u_i=w_{i,0}=u'_i$. This shows that $N'=\npp{f}$. Let us denote this common polygon by $N$.

Finally, let us prove the equality of all residual coefficients. If $c_i\ne0$, then $u_i=y_i(N)$, and from (\ref{ui}) we get $k_0=0$ and $u_i=w_{i,0}=u'_i$. By (\ref{clear}), (\ref{unique}) and  (\ref{cprima}), we get $c_i=\rd(a_i(x)/\pi^{u_i})=\rd(b_{i,0}(x)/\pi^{u_i})=c'_i$.
If $c_i=0$, then $u_i>y_i(N)$, and from (\ref{unique}) and (\ref{clear}) we get $w_{i,0}>y_i(N)$ too. By (\ref{cprima}) we get  $c'_i=0$.
\end{proof}

The construction of the principal part of the $\phi$-Newton polygon of a polynomial can be
interpreted as a mapping
$$
N_{\phi}^-\colon \zpx\setminus\{0\}\lra \pol,\qquad f(x)\mapsto \npp{f}.
$$
Also, for any negative rational number $\la$, the construction of the residual polynomial attached to $\la$  can be
interpreted as a mapping
$$
R_{\la}\colon \zpx\setminus\{0\}\lra \fph[y]\setminus\{0\},\qquad f(x)\mapsto R_{\la}(f)(y).
$$  
The Theorem of the product says that both mappings are semigroup
homomorphisms.

\begin{theorem}[Theorem of the product] For any $f(x),g(x)\in\zpx\setminus\{0\}$ and any $\la\in\Q^-$ we have
 $$\npp{fg}=\npp{f}+\npp{g},\qquad R_{\la}(fg)(y)=R_{\la}(f)(y)R_{\la}(g)(y).$$
\end{theorem}

\begin{proof}
Consider the respective $\phi$-adic developments
$$
f(x)=\sum_{0\le i}a_i(x)\phi(x)^i,\quad
g(x)=\sum_{0\le j}b_j(x)\phi(x)^j,
$$and denote $u_i=v(a_i(x))$, $v_j=v(b_j(x))$, $N_f=\npp{f}$, $N_g=\npp{g}$. Then,

\begin{equation}\label{product}
f(x)g(x)=\sum_{0\le k}A_k(x)\phi(x)^k,\qquad
A_k(x)=\sum_{i+j=k}a_i(x)b_j(x).
\end{equation}
Denote by $N'$ the principal part of the Newton polygon of $fg$, determined by this $\phi$-development.

We shall show that $N'=N_f+N_g$, that this $\phi$-development is admissible, and that $R'_{\la}(fg)=R_{\la}(f)R_{\la}(g)$ for all $\la$. The theorem will be then a
consequence of Lemma \ref{admissible}.

Let $w_k:=v(A_k(x))$ for all $0\le k$. Lemma
\ref{sum} shows that the point $(i,u_i)+(j,v_j)$ lies above $N_f+N_g$
for any $i,j\ge0$. Since $w_k\ge
\min\{u_i+v_j,i+j=k\}$, the points $(k,w_k)$ lie all above $N_f+N_g$.
On the other hand, let $P_k=(k,y_k(N_f+N_g))$ be a vertex of $N_f+N_g$; that is, $P_k$ is the end point
of $S_1+\cdots+S_r+T_1+\cdots+T_s$, for certain sides $S_i$ of $N_f$ and $T_j$ of $N_g$, ordered by increasing slopes among all sides of $N_f$ and $N_g$. By Lemma \ref{sum}, for all
pairs $(i,j)$ such that $i+j=k$, the point $(i,u_i)+(j,v_j)$ lies
strictly above $N_f+N_g$ except for the pair
$i_0=\ell(S_1+\cdots+S_r)$, $j_0=\ell(T_1+\cdots+T_s)$ that
satisfies $(i_0,u_{i_0})+(j_0,v_{j_0})=P_k$. Thus,
$(k,w_k)=P_k$ and
$$\rd\left(\frac{A_k(x)}{\pi^{y_k(N_f+N_g)}}\right)=\rd\left(\frac{a_{i_0}(x)b_{j_0}(x)}{\pi^{y_k(N_f+N_g)}}
\right)=\rd\left(\frac{a_{i_0}(x)}{\pi^{y_{i_0}(N_f)}}\right)\rd\left(\frac{b_{j_0}(x)}{\pi^{y_{j_0}(N_g)}}\right)\ne0.$$
This shows that $N'=N_f+N_g$ and that the $\phi$-development (\ref{product}) is
admisible.

Finally, by (\ref{sumsla}), the $\la$-components $S'=S_{\la}(N')$, $S_f=S_{\la}(N_f)$, $S_g=S_{\la}(N_g)$ are related by: $S'=S_f+S_g$. Let $(k,y_k(N'))$ be a point of integer coordinates lying on $S'$ (not necessarily a vertex). Denote by $I$ the set of the pairs $(i,j)$ such that $(i,u_i)$ lies on $S_f$, $(j,v_j)$ lies on $S_g$, and $i+j=k$. Take $P(x)=\sum_{(i,j)\in I}a_i(x)b_j(x)$. By Lemma \ref{sum}, for all other pairs $(i,j)$ with $i+j=k$, the point $(i,u_i)+(j,v_j)$ lies strictly above $N'$. Therefore, $c'_k(fg)=\rd(P(x)/\pi^{y_k(N')})=\sum_{(i,j)\in I}c_i(f)c_j(g)$. This shows that $R'_{\la}(fg)(y)=R_{\la}(f)(y)R_{\la}(g)(y)$.
\end{proof}

\noindent{\bf Notation. }{\it
Let $\mathcal{F}$ be a field and $\varphi(y),\,\psi(y)\in \mathcal{F}[y]$ two polynomials. We write $\varphi(y)\sim \psi(y)$ to indicate that  there exists a constant $c\in\mathcal{F}^*$ such that $\varphi(y)=c\psi(y)$}.\medskip

\begin{corollary}\label{slopezero}
Let $f(x)\in\zpx$ be a monic polynomial. Let $\phi_1(x),\dots,\phi_r(x)$ be monic polyomials in $\zpx$ such that their reductions modulo $\m$ are pairwise different irreducible polynomials of $\ff{}[x]$ and
$$
f(x)\equiv \phi_1(x)^{n_1}\cdots \phi_r(x)^{n_r}\md{\m}.
$$  
Let $f(x)=F_1(x)\cdots F_r(x)$ be the factorization into a product of monic polynomials of $\zpx$ satisfying $F_i(X)\equiv \phi_i(x)^{n_i}\md{\m}$, provided by Hensel's lemma. Then, 
$$
N_{\phi_i}(F_i)=N_{\phi_i}^-(F_i)=N_{\phi_i}^-(f),\qquad R_{\la}(F_i)(y)\sim R_{\la}(f)(y),
$$for all $1\le i\le r$ and all $\la\in\Q^-$.
\end{corollary}

\begin{proof}For any $1\le i\le r$, let $G_i(x)=\prod_{j\ne i}F_j(x)$. Since $\phi_i(x)$ does not divide $G_i(x)$ modulo $\m$, the principal $\phi_i$-Newton polygon of $G_i(x)$ reduces to the point $(0,0)$. By the Theorem of the product, $N_{\phi_i}^-(f)=N_{\phi_i}^-(F_i)+N_{\phi_i}^-(G_i)=N_{\phi_i}^-(F_i)$. On the other hand, 
$N_{\phi_i}(F_i)=N_{\phi_i}^-(F_i)$ because both polygons have length $n_i$.
Now, for any $\la\in\Q^-$, $S_{\la}(G_i)$ is a point and $R_{\la}(G_i)(y)$ is a nonzero constant. By the 
 Theorem of the product, $R_{\la}(f)(y)=R_{\la}(F_i)(y)R_{\la}(G_i)(y)\sim R_{\la}(F_i)(y)$.
\end{proof}

\subsection{Theorems of the polygon and of the residual polynomial}
Let $f(x)\in\zpx$ be a monic polynomial divisible by $\phi(x)$ modulo $\m$. By Hensel's lemma, 
$f(x)=f_{\phi}(x)G(x)$ in $\zpx$, with monic polynomials $f_{\phi}(x)$, $G(x)$ such that $\rd(G(x))\ne0$ and
$f_{\phi}(x)\equiv \phi(x)^n\md{\m}$. The aim of this section is to obtain a further factori\-zation of $f_{\phi}(x)$ and certain arithmetic data about the factors. Thanks to Corollary \ref{slopezero}, we shall be able to read this information directly on $f(x)$; more precisely, on $\npp{f}=\nph{f_{\phi}}$ and $R_{\la}(f)(y)\sim R_{\la}(f_{\phi})(y)$.

\begin{theorem}[Theorem of the polygon] Let $f(x)\in\zpx$ be a monic polynomial divisible by $\phi(x)$ modulo $\m$. Suppose that $\npp{f}=S_1+\cdots+S_g$ has $g$ sides with pairwise different slopes
$\la_1,\dots,\la_g$. Then, $f_{\phi}(x)$ admits a factorization in $\zpx$ into a product of $g$ monic polynomials
$$
f_{\phi}(x)=F_1(x)\cdot\cdots\cdot F_g(x),
$$
such that, for all $1\le i\le g$,
\begin{enumerate}
 \item $\nph{F_i}=S'_i$ is one-sided, and $S'_i$ is equal to $S_i$ up to a translation,
\item If $S_i$ has finite slope $\la_i$, then $R_{\la_i}(F_i)(y)\sim R_{\la_i}(f)(y)$,
\item For any root $\t\in\qpb$ of $F_i(x)$, we have $v(\phi(\t))=|\la_i|$. 
\end{enumerate}
\end{theorem}

\begin{proof}
By  the Theorem of the product and  Corollary \ref{slopezero}, it is sufficient to show that if $F(x):=f_{\phi}(x)$ is irreducible, then  $\nph{F}=S$ is one-sided and the roots $\t\in\qpb$ have all $v(\phi(\t))$ equal to minus the slope of $S$.

In fact, for all the roots $\t\in\qpb$ of $F(x)$, the rational number $v(\phi(\t))$ takes  the same value because the $p$-adic valuation is invariant under the Galois action. Since $F(x)$ is congruent to a power of $\phi(x)$ modulo $\m$ we have $\la:=-v(\phi(\t))<0$. We have $\la=-\infty$ if and only if $F(x)=\phi(x)$, and in this case the theorem is clear. Suppose $\la$ is finite.

Let $x^k+b_{k-1}x^{k-1}+\cdots+b_0\in\zpx$ be the minimal polynomial
of $\phi(\t)$ and let $Q(x)=\phi(x)^k+b_{k-1}\phi(x)^{k-1}+\cdots+b_0$.
We have $v(b_0)=k|\la|$ and $v(b_i)\ge (k-i)|\la|$ for all $i$; this implies that
$\nph{Q}$ is one-sided with slope $\la$. Since $Q(\t)=0$, our
polynomial $F(x)$ is an irreducible factor of $Q(x)$ and by the
Theorem of the product $\nph{F}$ is also one-sided with slope
$\la$.
\end{proof}

We note that the factor corresponding to a side $S_i$ of slope $-\infty$ is necessarily $F_i(x)=\phi(x)^{\ord_{\phi}(f)}$ (cf. Remark \ref{phiell}). 

Let $\la=-h/e$, with $h,\,e$ coprime positive integers, be a negative rational number such that $S:=S_{\la}(f)$ has positive length. Let $f_{\phi,\la}(x)$ be the factor of $f(x)$, corres\-ponding to the pair $\phi,\la$ by the Theorem of the polygon. Choose a root $\t\in\qpb$ of $f_{\phi,\la}(x)$ and let $L=K(\t)$. Since $v(\phi(\t))>0$, we can consider an embedding 
\begin{equation}\label{embedding}
 \zpx/(\pi,\phi(x))=\fph\hookrightarrow \ff{L},\quad \rd(x)\mapsto \tb.
\end{equation}
Thus, a polynomial $P(x)\in\zpx$ satisfies $v(P(\t))>0$ if and only if $P(x)$ is divisible by $\phi(x)$ modulo $\m$. 
This embedding depends on the choice of $\t$ (and not only on $L$). After this identification of $\fph$ with a subfield of $\ff{L}$ we can think that all residual polynomials have coefficients in $\ff{L}$. The Theorem of the polygon yields certain arithmetic information on the field $L$.

\begin{corollary}\label{ram}
With the above notations, the residual degree $f(L/K)$ is divisible by $\,m=\deg \phi(x)$, and the ramification index $e(L/K)$ is divisible by $\,e$.
Moreover, the number of irreducible factors of $f_{\phi,\la}(x)$ is at most $d(S)$; in particular, if $d(S)=1$ the polynomial $f_{\phi,\la}(x)$ is irreducible in $\zpx$, and $f(L/K)=m$, $e(L/K)=e$.
\end{corollary}

\begin{proof}
The statement on the residual degree is a consequence of the embedding (\ref{embedding}). By the theorem of the polygon,
$v_L(\phi(\t))=e(L/K)h/e$. Since this is an integer and $h,e$ are coprime, ne\-cessarily $e$ divides $e(L/K)$. The upper bound for the number of irreducible factors is a consequence of the Theorem of the product. Finally, if $d(S)=1$, we have $me=\dg(f_{\phi,\la}(x))=f(L/K)e(L/K)$, and necessarily $f(L/K)=m$ and $e(L/K)=e$. 
\end{proof}

Let $\ga(x):=\phi(x)^e/\pi^h\in K[x]$. Note that $v(\ga(\t))=0$; in particular, $\ga(\t)\in \ol$.

\begin{proposition}[Computation of $v(P(\t))$ with the polygon]\label{vqt}
We keep the above notations for $f(x),\,\la=-h/e,\,\t,\,L$, $\ga$, and the embedding $\fph\subseteq \ff{L}$ of (\ref{embedding}). Let $P(x)\in\zpx$ be a nonzero polynomial, $S=S_{\la}(P)$, $L_{\la}$ the line of slope $\la$ that contains $S$, and $H$ the ordinate of the intersection of this line with the vertical axis. Then,  
\begin{enumerate}
 \item $v(P^S(\t))\ge0,\quad \overline{P^S(\t)}=R_{\la}(P)(\gb{})$,
\item $v(P(\t)-P^0(\t))>H$.
\item $v(P(\t))\ge H$, and equality holds if and only if $R_{\la}(P)(\gb{})\ne0$,
\item $R_{\la}(f)(\gb{})=0$.
\item  If $R_{\la}(f)(y)\sim \psi(y)^a$ for an irreducible polynomial $\psi(y)\in\fph[y]$, then $v(P(\t))=H$  if and only if $\psi(y)\nmid R_{\la}(P)(y)$ in $\fph[y]$.
\end{enumerate}
\end{proposition}

\begin{proof}
Let $P(x)=\sum_{0\le i}b_i(x)\phi(x)^i$ be the $\phi$-adic development of $P(x)$, and denote $u_i=v(b_i)$, $N=\npp{P}$. Recall that $P^S(x)=\phi(x)^{-s}\pi^{-u}P^0(x)$, where $(s,u)$ are the coordinates of the initial point of $S$, and $P^0(x)=\sum_{(i,u_i)\in S}b_i(x)\phi(x)^i$. Hence, for $d=d(S)$ we have 
\begin{multline*}
P^S(x)=\pi^{-u}\left(b_s(x)+b_{s+e}(x)\phi(x)^e+\cdots+b_{s+de}\phi(x)^{de}\right)\\
=\frac{b_s(x)}{\pi^u}+\frac{b_{s+e}(x)}{\pi^{u-h}}\ga(x)+\cdots+\frac{b_{s+de}(x)}{\pi^{u-hd}}\ga(x)^d.
\end{multline*}
Since $v(b_{s+ie})\ge y_{s+ie}(N)=u-hi$ for all $1\le i\le d$, the two statements of item 1 are clear.

All points of $N$ lie above the line $L_{\la}$; hence, for any integer abscissa $i$ 
$$
v(b_i(\t)\phi(\t)^i)=u_i+i\frac he\ge y_i(N)+i\frac he\ge H,
$$and equality holds if and only  if $(i,u_i)\in S$. 
This proves item 2. Also, this shows that $v(P(\t))\ge H$. Since $v(\phi(\t)^s\pi^u)=u+sh/e=H$, we have  
$$v(P(\t))=H\,\sii\,v(P^0(\t))=H\,\sii\,v(P^S(\t))=0\,\sii\,R_{\la}(P)(\gb{})\ne0,$$
the last equivalence by item 1. This proves item 3.

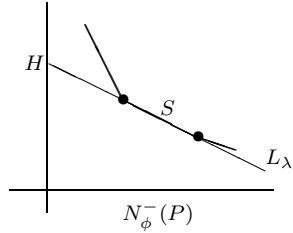
\begin{center}
\setlength{\unitlength}{5.mm}
\begin{picture}(6,5)
\put(3.85,.85){$\bullet$}\put(1.85,1.85){$\bullet$}
\put(-1,-.4){\line(1,0){7}}\put(0,-1){\line(0,1){5.6}}
\put(4,1){\line(-2,1){2}}\put(4.02,1){\line(-2,1){2}}
\put(4,1){\line(3,-1){1}}\put(4.02,1){\line(3,-1){1}}
\put(2,2){\line(-1,2){1}}\put(2.02,2){\line(-1,2){1}}
\put(5.8,0.1){\line(-2,1){5.75}}\put(5.8,.3){\begin{footnotesize}$L_{\la}$\end{footnotesize}}
\put(3,1.6){\begin{footnotesize}$S$\end{footnotesize}}
\put(2,-1.2){\begin{footnotesize}$\npp{P}$\end{footnotesize}}
\put(-.6,2.8){\begin{footnotesize}$H$\end{footnotesize}}
\end{picture}
\end{center}\be\be\be

Since $f(\t)=0$, item 4 is a consequence of item 3 applied to $P(x)=f(x)$.
Finally, if $R_{\la}(f)(y)\sim\psi(y)^a$, then $\psi(y)$ is the minimal polynomial of $\gb{}$ over $\fph$, by item 4. Hence, $R_{\la}(P)(\gb{})\ne0$ is equivalent to $\psi(y)\nmid R_{\la}(P)(y)$ in $\fph[y]$.
\end{proof}

We discuss now how Newton polygons and residual polynomials are affected by an extension of the base field by an unramified extension. 

\begin{lemma}\label{extension}
We keep the above notations for $f(x),\,\la=-h/e,\,\t,\,L$. Let $K'\subseteq L$ be the unramified extension of $K$ of degree $m$, and identify $\fph=\ff{K'}$ through the embedding (\ref{embedding}). Let $G(x)\in \oo_{K'}[x]$ be the minimal polynomial of $\t$ over $K'$. Let $\phi'(x)=x-\eta$, where $\eta\in K'$ is the unique root of $\phi(x)$ such that $G(x)$ is divisible by $x-\eta$ modulo $\m_{K'}$. Then, for any nonzero polynomial $P(x)\in\zpx$:
$$N_{\phi'}^-(P)=N_{\phi}^-(P), \quad R'_{\la}(P)(y)=\epsilon^sR_{\la}(P)(\epsilon^ey),$$
where $R'$ denotes the residual polynomial with respect to $\phi'(x)$ over $K'$, $\epsilon\in\ff{K'}^*$ does not depend on $P(x)$, and $s$ is the initial abscissa of $S_{\la}(P)$.  
\end{lemma}

\begin{proof}
Consider the $\phi$-adic development of $P(x)$:
\begin{multline*}
P(x)=\phi(x)^n+a_{n-1}(x)\phi(x)^{n-1}+\cdots+a_0(x)=\\=\rho(x)^{n}\phi'(x)^n+a_{n-1}(x)\rho(x)^{n-1}\phi'(x)^{n-1}+\cdots+a_0(x),
\end{multline*}
where $\rho(x)= \phi(x)/\phi'(x) \in\oo_{K'}[x]$. Since $\phi(x)$ is irreducible modulo $\m$, it is a separable polynomial modulo $\m_{K'}$, so that $\rho(x)$ is not divisible by $\phi'(x)$ modulo $\m_{K'}$, and
$v(\rho(\t))=0$. Therefore, the above $\phi'(x)$-development of
$P(x)$ is admissible and $N_{\phi'}^-(P)=\npp{P}$ by Lemma \ref{admissible}. Moreover the residual coefficients of the two polygons are related by $c'_i=c_i\epsilon^i$, where $\epsilon=\overline{\rho(\t)}\in\ff{K'}^*$. This proves that
 $R'_{\la}(P)(y)=\epsilon^s R_{\la}(P)(\epsilon^ey)$.
\end{proof}

\begin{theorem}[Theorem of the residual polynomial]\label{thresidual} Let $f(x)\in\zpx$ be a monic polynomial which is divisible by $\phi(x)$ modulo $\m$. Let $S$ be a side of $N_{\phi}^-(f)$, of finite slope  $\la$, and let 
$$
R_{\la}(f)(y)\sim\psi_1(y)^{a_1}\cdots \psi_t(y)^{a_t}
$$be the factorization of the residual polynomial of $f(x)$ into the product of powers of pairwise different monic irreducible polynomials in $\fph[y]$. Then, the factor $f_{\phi,\la}(x)$ of $f(x)$, attached to $\phi,\la$ by the Theorem of the polygon, admits a factorization
$$
f_{\phi,\la}(x)=G_1(x)\cdots G_t(x)
$$into a product of $t$ monic polynomials in $\zpx$, such that all $N_{\phi}(G_i)$ are one-sided with slope $\la$, and
 $R_{\la}(G_i)(y)\sim \psi_i(y)^{a_i}$ in $\fph[y]$, for all $1\le i\le t$.
\end{theorem}

\begin{proof}By the Theorem of the product, we need only to prove that if $F(x):=f_{\phi,\la}(x)$ is  irreducible, then $R_{\la}(F)(y)$ is the power of an irreducible polynomial of $\fph[y]$. Let $\t,\,L,\,K',\,G(x)$ be as in Lemma \ref{extension}, so that $F(x)=\prod_{\sigma\in\op{Gal}(K'/K)}G^{\sigma}(x)$. Under the embedding $\fph\lra\ff{L}$, the field $\fph$ is identified to $\ff{K'}$. By Lemma \ref{extension}, there is a polynomial of degree one, $\phi'(x)\in\oo_{K'}[x]$, such that $R'_{\la}(F)(y)\sim R_{\la}(F)(cy)$, for some nonzero constant $c\in\ff{K'}$. For any $\sigma\ne1$, the polynomial $G^{\sigma}(x)$ is not divisible by $\phi'(x)$ modulo $\m_{K'}$; thus, $N_{\phi'}(G^{\sigma})$ is reduced to a point, and $R'_{\la}(G^{\sigma})(y)$ is a constant. Therefore, by the Theorem of the product, $R'_{\la}(G)(y)\sim R'_{\la}(F)(y)\sim R_{\la}(F)(cy)$, so that $R_{\la}(F)(y)$ is the power of an irreducible polynomial of $\fph[y]$ if and only if 
$R'_{\la}(G)(y)$ has the same property over $\ff{K'}$. In conclusion, by extending the base field, we can suppose that $\dg \phi=m=1$.

Consider now the minimal polynomial
$P(x)=x^k+b_{k-1}x^{k-1}+\cdots+b_0\in K[x]$ of $\ga(\t)=\phi(\t)^e/\pi^h$ over $K$. Since $v(\ga(\t))=0$, we have $v(b_0)=0$. Thus, the polynomial
$$
Q(x)=\phi(x)^{ek}+\pi^hb_{k-1}\phi(x)^{e(k-1)}+\cdots+\pi^{kh}b_0,
$$
has one-sided $N^-_{\phi}(Q)$ of slope $\la$, and $R_{\la}(Q)(y)$ is the reduction of
$P(y)$ modulo $\m$, which is the power of an irreducible
polynomial because $P(x)$ is irreducible in $K[x]$. Since $Q(\t)=0$, $F(x)$ divides $Q(x)$, and it has the same property by the Theorem of the product.
\end{proof}

\begin{corollary}
With the above notations, let $\t\in\qpb$ be a root of $G_i(x)$, and $L=K(\t)$. Then, $f(L/K)$ is divisible by $m\dg\psi_i$. Moreover, the number of irreducible factors of $G_i(x)$ is at most $a_i$; in particular, if $a_i=1$, then $G_i(x)$ is irreducible in $\zpx$, and $f(L/K)=m\dg\psi_i$, $e(L/K)=e$.
\end{corollary}

\begin{proof}
The statement about $f(L/K)$ is a consequence of the  
embedding of the finite field $\fph[y]/(\psi_i(y))$ into $\ff{L}$ determined by $\rd(x)\mapsto \tb$, $y\mapsto \gb{}$. This embedding is well-defined by item 4 of Proposition \ref{vqt}. The other statements are a consequence of the Theorem of the product and Corollary \ref{ram}. 
\end{proof}

\subsection{Types of order one}\label{orderone}
Starting with a monic and separable polynomial $f(x)\in\zpx$, the Newton polygon techniques provide partial information on the factorization of $f(x)$ in $\zpx$, obtained after three \emph{dissections} \cite{ber}. In the first dissection we obtain as many factors of $f(x)$ as pairwise different irreducible factors modulo $\m$ (by Hensel's lemma). In the second dissection, each of these factors splits into the product of as many factors as sides of certain Newton polygon of $f(x)$ (by the Theorem of the polygon). In the third dissection, the factor that corresponds to a side of finite slope splits into the product of as many factors as irreducible factors of the residual polynomial of $f(x)$ attached to this side (by the Theorem of the residual polynomial).

The final list of factors of $f(x)$ obtained by this procedure can be parameterized by certain data, which we call \emph{types of order zero and of order one}. 

\begin{definition}
A type of order zero is a monic irreducible polynomial $\ty=\psi_0(y)\in\ff{}[y]$. We attach to any type of order zero the map
$$
\om^\ty\colon \zpx\setminus\{0\}\lra \Z_{\ge 0},\quad P(x)\mapsto \ord_{\psi_0}(\overline{P(x)/\pi^{v(P)}}).
$$ 

Let $f(x)\in\zpx$ be a monic and separable polynomial. We say that the type $\ty$ is $f$-complete if $\om^\ty(f)=1$. In this case, we denote by $f_\ty(x)\in\zpx$ the monic irreducible factor of $f(x)$ determined by $f_\ty(y)\equiv \psi_0(y) \md{\m}$.
\end{definition}

\begin{definition}
A type of order one is a triple $\ty=(\phi(x);\la,\psi(y))$, where 
\begin{enumerate}
 \item $\phi(x)\in\Z[x]$ is a monic polynomial which is irreducible modulo $\m$.
\item  $\la=-h/e\in\Q^-$, with $h,e$ positive coprime integers.
\item  $\psi(y)\in \fph[y]$ is a monic irreducible polynomial, $\psi(y)\ne y$.
\end{enumerate}
By truncation of $\ty$ we obtain the type of order zero $\ty_0:=\phi(y)\md{\m}$.
\end{definition}

We denote by $\ty_0(f)$ the set of all monic irreducible factors of $f(x)$ modulo $\m$. 
We denote by $\ty_1(f)$ the set of all types of order one obtained from $f(x)$ along the process of applying the three classical dissections: for any non-$f$-complete $\psi_0(y)\in\ty_0(f)$, we take a monic lift $\phi(x)$ to $\oo[x]$; then we consider all finite slopes $\la$ of the sides of positive length of $N_{\phi}^-(f)$, and finally, for each of them we take the different monic irreducible factors $\psi(y)$ of the residual polynomial $R_{\la}(f)(y)\in\fph[y]$. 
These types are not intrinsical objects of $f(x)$. There is a non-canonical choice of the lifts $\phi(x)\in\oo[x]$, and the data $\la$, $\psi(y)$ depend on this choice. 

We denote by $\Ty_1(f)$ the union of $\ty_1(f)$ and the set of all $f$-complete types of order zero.
By the previous results we have a factorization in $\zpx$
$$
f(x)=f_{\infty}(x)\prod_{\ty\in\Ty_1(f)}f_{\ty}(x),
$$
where $f_{\infty}(x)$ is the product of the different $\phi(x)$ that divide $f(x)$ in $\zpx$, and, if $\ty$ has order one,  $f_{\ty}(x)$ is the unique monic divisor of $f(x)$ in $\zpx$ satisfying the following properties:
$$
\as{1.4}
\begin{array}{l}
f_{\ty}(x)\equiv \phi(x)^{ea\deg\psi} \md{\m}, \ \mbox{ where }a=\ord_{\psi}(R_{\la}(f)),\\
\nph{f_{\ty}} \mbox{ is one-sided with slope }\la,\\
R_{\la}(f_{\ty})(y)\sim\psi(y)^a \mbox{ in }\fph[y].
\end{array}
$$
The factor $f_{\infty}(x)$ is already expressed as a product of irreducible polynomials in $\zpx$. Also, if $a=1$, the Theorem of the residual polynomial shows that $f_\ty(x)$ is irreducible too. Thus, the remaining task is to obtain the further factorization of $f_{\ty}(x)$, for the types $\ty\in\ty_1(f)$ with $a>1$. 
The factors of $f_{\ty}(x)$ will bear a reminiscence of $\ty$ as a birth mark (cf. Lemma \ref{factortype}). 

Once a type of order one $\ty=(\phi(x);\la,\psi(y))$ is fixed, we change the notation of several objects that depend on the data of the type. We omit the data from the notation but we include the subscript ``$1$" to emphasize that they are objects of the first order. From now on, for any nonzero polynomial $P(x)\in\zpx$, any principal polygon $N\in\pol$, and any $T\in\ss(\la)$, we shall denote
$$\as{1.2}
\begin{array}{ll}
v_1(P):=v(P),\quad&\om_1(P):=\om^{\overline{\phi}}(P)=\ord_{\overline{\phi}}(\overline{P(x)/\pi^{v(P)}}),\\
N_1(P):=\nph{P},\quad &N_1^-(P):=\npp{P},\\
S_1(N):=S_{\la}(N),\quad &S_1(P):=S_{\la}(P)=S_{\la}(N_1^-(P)),\\
R_1(P)(y):=R_{\la}(P)(y),\quad &R_1(P,T)(y):=R_{\la}(P,T)(y). 
\end{array}
$$
Note that $\om_1(P),N_1(P)$ depend only on $\phi(x)$, wheras $S_1(P)$, $R_1(P)$, $R_1(P,T)$ depend on the pair $\phi(x),\la$.

The aim of the next two sections is to introduce Newton polygons of higher order and prove similar theorems, yielding information on a further factorization of each $f_{\ty}(x)$. As before, we shall obtain arithmetic information about the factors of $f_{\ty}(x)$ just by a direct manipulation of $f(x)$, without actually computing a $p$-adic approximation to these factors. This fact is crucial to ensure that the whole process has a low complexity. However, once an irreducible factor of $f(x)$ is ``detected",  
the theory provides a reasonable approximation of this factor as a by-product (cf. Proposition \ref{appfactor}). 

\section{Newton polygons of higher order}\label{secNPr}
Throughout this section, $r$ is an integer, $r\ge 2$. We shall construct Newton polygons of order $r$ and prove their basic properties and the Theorem of the product in order $r$, under the assumption that analogous results have been already obtained in orders $1,\dots,r-1$. We also assume that the theorems of the polygon and of the residual polynomial have been proved in orders $1,\dots,r-1$ (cf. section \ref{secOre}). For $r=1$ all these results have been proved in section \ref{secNP}.

\subsection{Types of order $r-1$}
A \emph{type of order $r-1$} is a sequence of data
$$
\ty=(\phi_1(x);\la_1,\phi_2(x);\cdots;\la_{r-2},\phi_{r-1}(x);\la_{r-1},\psi_{r-1}(y)),
$$
where $\phi_i(x)$ are monic polynomials in $\oo[x]$, $\la_i$ are negative rational numbers and $\psi_{r-1}(y)$ is a polynomial over certain finite field (to be specified below), that satisfy the following recursive properties:
\begin{enumerate}
\item $\phi_1(x)$ is irreducible modulo $\m$. We denote by $\psi_0(y)\in\ff{}[y]$ the polynomial obtained by reduction of $\phi_1(y)$ modulo $\m$. We define $\ff1:=\ff{}[y]/(\psi_0(y))$. 
\item For all $1\le i<r-1$, the Newton polygon of $i$-th order, $N_i(\phi_{i+1})$, is one-sided, with positive length and slope $\la_i$. 
\item For all $1\le i<r-1$, the residual polynomial of $i$-th order, $R_i(\phi_{i+1})(y)$, is an irreducible polynomial in $\ff{i}[y]$. We denote by $\psi_i(y)\in \ff{i}[y]$ the monic polynomial determined by $R_i(\phi_{i+1})(y)\sim\psi_i(y)$. We define $\ff{i+1}:=\ff{i}[y]/(\psi_i(y))$.
\item $\psi_{r-1}(y)\in\ff{r-1}[y]$ is a monic irreducible polynomial, $\psi_{r-1}(y)\ne y$.  We define $\ff{r}:=\ff{r-1}[y]/(\psi_{r-1}(y))$.
\end{enumerate}

The type determines a tower $\ff{}=:\ff{0}\subseteq\ff{1}\subseteq\cdots \subseteq\ff{r}$ of finite fields. The field $\ff{i}$ should not be confused with the finite field with $i$ elements. 

By the Theorem of the product in orders $1,\dots,r-1$, the polynomials $\phi_i(x)$ are all irreducible over $\zpx$. 

Let us be more precise about the meaning of $N_i(-)$, $R_i(-)$, used in items 2,3.\medskip

\noindent{\bf Notation. }{\it We denote $\ty_0=\psi_0(y)$. For all $1\le i<r$, we obtain by truncation of $\ty$ a type of order $i$, and a \emph{reduced type} of order $i$, defined respectively as:  
$$\as{1.4}
\begin{array}{l}
\ty_i:=(\phi_1(x);\la_1,\phi_2(x);\cdots;\la_{i-1},\phi_i(x);\la_i,\psi_i(y)),\\
\ty_i^0:=(\phi_1(x);\la_1,\phi_2(x);\cdots;\la_{i-1},\phi_i(x);\la_i).
\end{array}
$$
Also, we define the \emph{extension} of the type $\ty_{i-1}$ to be $$\tilde{\ty}_{i-1}:=(\phi_1(x);\la_1,\phi_2(x);\cdots;\la_{i-1},\phi_i(x)).$$
We have semigroup homomorphisms:
$$\as{1.4}
\begin{array}{l}
N_i^-\colon \zpx\!\setminus\!\{0\}\to \pol,\quad S_i\colon \zpx\!\setminus\!\{0\}\to \ss(\la_i),\quad R_i\colon \zpx\!\setminus\!\{0\}\to \ff{i}[y]. 
\end{array}
$$

For any nonzero polynomial $P(x)\in\zpx$, $N_i(P)$ is the $i$-th order Newton polygon with respect to the extended type $\tilde{\ty}_{i-1}$,  $S_i(P)$ is the $\la_i$-component of $N_i^-(P)$, and $R_i(P)(y)\in\ff{i}[y]$ is the residual polynomial of $i$-th order with respect to $\la_i$. The polynomial $R_i(P)(y)$ has degree $d(S_i(P))$. Both $S_i$ and $R_i$ depend only on the reduced type $\ty^0_i$. 
Finally, we denote by $s_i(P)$ the initial abscissa of $S_i(P)$}.\medskip

Other data attached to the type $\ty$ deserve an specific notation. For all $1\le i<r$:
\begin{itemize}
\item $\la_i=-h_i/e_i$, with $e_i,\,h_i$ positive coprime integers,
\item $f_i:=\dg \psi_i(y)$, $f_0:=\dg \psi_0(y)=\dg \phi_1(x)$,  
\item $m_i:=\dg \phi_i(x)$, and $m_r:=m_{r-1}e_{r-1}f_{r-1}$. Note that $m_{i+1}=m_ie_if_i=m_1e_1f_1\cdots e_if_i$,
\item $\ell_i,\,\ell'_i\in\Z$ are fixed integers such that $\ \ell_i h_i-\ell'_ie_i=1$, 
\item $z_i:=y \md{\psi_i(y)}\in\ff{i+1}^*$, $z_0:=y \md{\psi_0(y)}\in\ff{1}^*$. Note that $\ff{i+1}=\ff{i}(z_i)$,
\end{itemize}
Also, for all $0\le i<r$ we have semigroup homomorphisms $$
\om_{i+1}\colon \zpx\setminus\{0\} \lra\Z_{\ge 0},\qquad P(x)\mapsto  \ord_{\psi_i}(R_i(P)),$$
where, by convention: $R_0(P)=\overline{P(y)/\pi^{v(P)}}\in\ff{}[y]$. By Lemma \ref{shape} in order $r-1$ (see Definition \ref{shape1} and Remark \ref{phiell} for order one):
\begin{equation}\label{om=length}
\ell(N_i(P))=\lfloor \deg P/m_i\rfloor,\quad  \ell(N_i^-(P))=\om_i(P), \quad 1\le i< r,
\end{equation}
and $N_i^-(P)$ has a side of slope $-\infty$ if and only if $P(x)$ is divisible by $\phi_i(x)$ in $\zpx$. 

\begin{definition}\label{type}
We say that a monic polynomial $P(x)\in\zpx$ has type $\ty$ when 
\begin{enumerate}
 \item $P(x)\equiv \phi_1(x)^{a_0} \md{\m}$, for some positive integer $a_0$,
\item  For all $1\le i<r$, the Newton polygon $N_i(P)$ is one-sided, of slope $\la_i$, and
 $R_i(P)(y)\sim\psi_i(y)^{a_i}$ in $\ff{i}[y]$, for some positive integer $a_i$. 
\end{enumerate}
\end{definition}

\begin{lemma}\label{typedegree}
Let $P(x)\in\zpx$ be a nonzero polynomial. Then,
\begin{enumerate}
 \item $\om_1(P)\ge e_1f_1\om_2(P)\ge\cdots\ge e_1f_1\cdots e_{r-1}f_{r-1}\om_r(P)$.
\item $\deg P<m_r\imp\om_r(P)=0$.
\item If $P(x)$ has type $\ty$ then all inequalities in item 1 are equalities, and $$\dg P(x)=m_r\om_r(P)=m_{r-1}\om_{r-1}(P)=\cdots =m_1\om_1(P).$$
\end{enumerate}
\end{lemma}

\begin{proof}Item 1 is a consequence of (\ref{om=length}); in fact, for all $1\le i<r$:
\begin{equation}\label{chain}
e_if_i\om_{i+1}(P)\le e_i\dg R_i(P)=e_id(S_i(P))=\ell(S_i(P))\le \ell(N_i^-(P))=\om_i(P).
\end{equation}

Item 2 is a consequence of (\ref{om=length}) and item 1:
$$m_r>\deg P\ge m_{r-1}\om_{r-1}(P)\ge m_r\om_r(P).$$
  
Finally, if $P(x)$ has type $\ty$ the two inequalities of (\ref{chain}) are equalities, so that $m_i\om_i(P)=m_{i+1}\om_{i+1}(P)$; on the other hand, $\dg P=m_1a_0=m_1\om_1(P)$.
\end{proof}

\begin{definition}\label{ppt}
Let $P(x)\in\zpx$ be a monic polynomial with $\om_r(P)>0$. We denote by $P_{\ty}(x)$ the monic factor of $P(x)$ of greatest degree that has type $\ty$. By the Theorems of the polygon and of the residual polynomial in orders $1,\dots,r-1$, this factor exists and it satisfies
\begin{equation}\label{sameomega}
 \om_r(P_\ty)=\om_r(P),\quad \dg P_{\ty}=m_r\om_r(P).
\end{equation}
\end{definition}

\begin{lemma}\label{factortype}
Let $P(x),\,Q(x)\in\zpx$ be monic polynomials of positive degree. 
\begin{enumerate}
\item If $P(x)$ is irreducible in $\zpx$, then it is of type $\ty$ if and only if $\om_r(P)>0$.
\item $P(x)$ is of type $\ty$  if and only if $\dg P=m_r\om_r(P)>0$. 
\item $P(x)Q(x)$ has type $\ty$ if and only if $P(x)$ and $Q(x)$ have both type $\ty$.
\end{enumerate}
\end{lemma}

\begin{proof}The polynomial $P(x)$ is of type $\ty$ if and only if $\om_r(P)>0$ and $P_\ty(x)=P(x)$; thus,
items 1 and 2 are an immediate consequence of (\ref{sameomega}). Item 3 follows from the Theorem of the product in orders $1,\dots,r-1$.
\end{proof}

We fix a type $\ty$ of order $r-1$ for  the rest of section \ref{secNPr}.

\subsection{The $p$-adic valuation of $r$-th order}  
In this paragraph we shall attach to $\ty$ a discrete valuation $v_r\colon K(x)^*\lra \Z$, that restricted to $K$ extends $v$ with index $e_1\cdots e_{r-1}$. We need only to define $v_r$ on $\zpx$.
Consider the mapping
$$
H_{r-1}\colon \ss(\la_{r-1})\lra \Z_{\ge0},
$$
that assigns to each side $S\in \ss(\la_{r-1})$ the ordinate of the point of intersection of the vertical axis with the line $L_{\la_{r-1}}$ of slope $\la_{r-1}$ that contains $S$. If $(i,u)$ is any point of integer coordinates lying on $S$, then $H_{r-1}(S)=u+|\la_{r-1}|i$; thus, $H_{r-1}$ is a semigroup homomorphism.

\begin{definition}
For any nonzero polynomial $P(x)\in\zpx$, we define
$$
v_r(P):=e_{r-1}H_{r-1}(S_{r-1}(P)).
$$
Note that $v_r$ depends only on the reduced type $\ty^0$.
\end{definition}\medskip

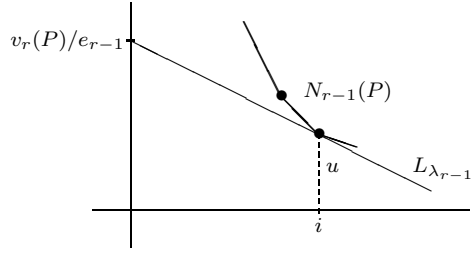
\begin{center}
\setlength{\unitlength}{5.mm}
\begin{picture}(10,5.4)
\put(5.85,1.85){$\bullet$}\put(4.85,2.85){$\bullet$}
\put(0,0){\line(1,0){10}}
\put(6,2){\line(-1,1){1}}\put(6.02,2){\line(-1,1){1}}
\put(6,2){\line(3,-1){1}}\put(6.02,2){\line(3,-1){1}}
\put(5,3){\line(-1,2){1}}\put(5.02,3){\line(-1,2){1}}
\put(9,.5){\line(-2,1){8}}
\put(.9,4.5){\line(1,0){.2}}
\put(1,-1){\line(0,1){6.5}}
\put(5.6,3){\begin{footnotesize}$N_{r-1}(P)$\end{footnotesize}}
\put(8.5,1){\begin{footnotesize}$L_{\la_{r-1}}$\end{footnotesize}}
\put(-2.2,4.35){\begin{footnotesize}$v_r(P)/e_{r-1}$\end{footnotesize}}
\multiput(6,-.1)(0,.25){9}{\vrule height2pt}
\put(5.9,-.6){\begin{footnotesize}$i$\end{footnotesize}}
\put(6.2,1){\begin{footnotesize}$u$\end{footnotesize}}
\end{picture}
\end{center}\be\be
\begin{proposition}
The natural extension of $v_r$ to $K(x)^*$ is a discrete valuation, whose restriction to $K^*$ extends $v$ with index $e_1\cdots e_{r-1}$.  
\end{proposition}

\begin{proof}
The mapping $v_r$ restricted to $\zpx\setminus\{0\}$ is a semigroup homomorphism, because it is the composition of three semigroup homomorphisms; in particular, $v_r\colon K(x)^*\lra \Z$ is a group homomorphism.

Let $P(x),\,Q(x)\in\zpx$ be two nonzero polynomials and denote $N_P=\np{r-1}(P)$, $N_Q=\np{r-1}(Q)$, $L_P=L_{\la_{r-1}}(N_P)$, $L_Q=L_{\la_{r-1}}(N_Q)$ (cf. Definition \ref{sla}). All points of $N_P$ lie above the line $L_P$ and all points of $N_Q$ lie above the line $L_Q$. If $v_r(P)\le v_r(Q)$, all points of both polygons lie above the line $L_P$. Thus, all points of $\np{r-1}(P+Q)$ lie above this line too, and this shows that $v_r(P+Q)\ge v_r(P)$.

Finally, for any $a\in\oo$, we have $v_r(a)=e_{r-1}v_{r-1}(a)$ by definition, since the $(r-1)$-th order Newton polygon of $a$ is the single point $(0,v_{r-1}(a))$.  
\end{proof}

This valuation was introduced by S. MacLane without using Newton polygons \cite{mcla}, \cite{mclb}. In \cite[Ch.2,\S2]{m}, J. Montes computed explicit generators of the residue field of $v_r$ as a transcendental extension of a finite field. These results lead to a more conceptual and elegant definition of residual polynomials in higher order, as the reductions modulo $v_r$ of the virtual factors. However, we shall not follow this approach, in order not to burden the paper with more technicalities.
     
The next proposition gathers the basic properties of this discrete valuation.

\begin{proposition}\label{propertiesv} Let $P(x)\in\zpx$ be a nonzero polynomial.
\begin{enumerate}
 \item $v_r(P)\ge e_{r-1}v_{r-1}(P)$ and equality holds if and only if $\om_{r-1}(P)=0$.
\item $v_r(P)=0$ if and only if $v_2(P)=0$ if and only if $\rd(P)\ne0$.
\item If $P(x)=\sum_{0\le i}a_i(x)\phi_{r-1}(x)^i$ is the $\phi_{r-1}$-adic development of $P(x)$, then
$$
v_r(P)=\min_{0\le i}\{v_r(a_i(x)\phi_{r-1}(x)^i)\}=e_{r-1}\min_{0\le i}\{v_{r-1}(a_i)+i(v_{r-1}(\phi_{r-1})+|\la_{r-1}|)\}.
$$ 
\item $v_r(\phi_{r-1})=e_{r-1}v_{r-1}(\phi_{r-1})+h_{r-1}$.
\end{enumerate}
\end{proposition}
  
\begin{proof}We denote $N=N_{r-1}^-(P)$ throughout the proof. By (1) of Lemma \ref{shape} in order $r-1$,
all points of $N$ lie above the horizontal line with ordinate $v_{r-1}(P)$. Hence, $v_r(P)\ge e_{r-1}v_{r-1}(P)$. Equality holds if and only if the first point of $N$ is $(0,v_{r-1}(P))$; this is equivalent to $\om_{r-1}(P)=0$, by (\ref{om=length}). This proves item 1.

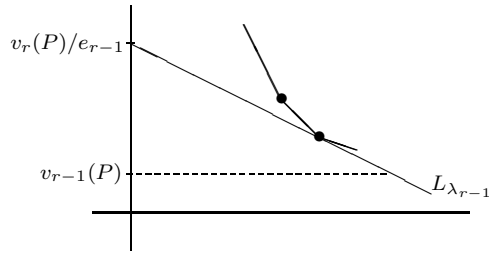
\begin{center}
\setlength{\unitlength}{5.mm}
\begin{picture}(10,5.4)
\put(5.85,1.85){$\bullet$}\put(4.85,2.85){$\bullet$}
\put(0,0){\line(1,0){10}}
\put(6,2){\line(-1,1){1}}\put(6.02,2){\line(-1,1){1}}
\put(6,2){\line(3,-1){1}}\put(6.02,2){\line(3,-1){1}}
\put(5,3){\line(-1,2){1}}\put(5.02,3){\line(-1,2){1}}
\put(9,.5){\line(-2,1){8}}
\put(.9,4.5){\line(1,0){.2}}
\put(1,-1){\line(0,1){6.5}}
\put(9,.6){\begin{footnotesize}$L_{\la_{r-1}}$\end{footnotesize}}
\put(-2.2,4.35){\begin{footnotesize}$v_r(P)/e_{r-1}$\end{footnotesize}}
\multiput(.9,1)(.25,0){28}{\hbox to 2pt{\hrulefill }}
\put(-1.4,.9){\begin{footnotesize}$v_{r-1}(P)$\end{footnotesize}}
\end{picture}
\end{center}\be\medskip

By a recurrent aplication of item 1, $v_r(P)=0$ is equivalent to $v_1(P)=0$ and $\om_1(P)=\cdots=\om_{r-1}(P)=0$. By Lemma \ref{typedegree} this is equivalent to  $v_1(P)=0$ and $\om_1(P)=0$, which is equivalent to $v_2(P)=0$, and also to $P(x)\not\in(\pi,\phi_1(x))$. This proves item 2.

\begin{center}
\setlength{\unitlength}{5.mm}
\begin{picture}(10,8)
\put(0,0){\line(1,0){10}}
\put(6,2){\line(-1,1){1}}\put(6.02,2){\line(-1,1){1}}
\put(6,2){\line(3,-1){1}}\put(6.02,2){\line(3,-1){1}}
\put(5,3){\line(-1,2){.8}}\put(5.02,3){\line(-1,2){.8}}
\put(9,.5){\line(-2,1){8}}
\put(5.85,1.85){$\bullet$}\put(4.85,2.85){$\bullet$}
\put(9,2.7){\line(-2,1){8}}
\put(1,-1){\line(0,1){8.5}}
\put(9.3,.4){\begin{footnotesize}$L_{\la_{r-1}}(N)$\end{footnotesize}}
\put(9.3,2.4){\begin{footnotesize}$L$\end{footnotesize}}
\multiput(8,0)(0,.25){13}{\vrule height2pt}
\put(7.85,3.05){$\bullet$}\put(7.95,-.6){\begin{footnotesize}$i$\end{footnotesize}}
\put(8,3.45){\begin{footnotesize}$(i,u_i)$\end{footnotesize}}
\put(-5.4,6.5){\begin{footnotesize}$v_r(a_i(x)\phi_{r-1}(x)^i)/e_{r-1}$\end{footnotesize}}
\put(-2.3,4.37){\begin{footnotesize}$v_r(P)/e_{r-1}$\end{footnotesize}}
\put(0.9,6.7){\line(1,0){.2}}
\put(.9,4.44){\line(1,0){.2}}
\end{picture}
\end{center}\be\be

By definition, $v_r(a_i(x)\phi_{r-1}(x)^i)$ is $e_{r-1}$ times the ordinate at the origin of the line $L$ that has slope $\la_{r-1}$ and passes through $(i,v_{r-1}(a_i(x)\phi_{r-1}(x)^i))$. Since all points of $N$ lie above the line 
$L_{\la_{r-1}}(N)$, the line $L$ lies above $L_{\la_{r-1}}(N)$ too, and $v_r(a_i(x)\phi_{r-1}(x)^i)\ge v_r(P)$. Since $v_r$ is a valuation, this proves item 3, and item 4 is a particular case.
\end{proof}

In a natural way, $\om_r$ induces a group homomorphism from $K(x)^*$ to $\Z$, but it is not a discrete valuation of this field. For instance, for $K=\Q_p$, $\pi=p$, $\ty=(x;-1,y+1)$ and $P(x)=x+p$, $Q(x)=x+p+p^2$, we have
$$
\as{1.4}
\begin{array}{lll}
\rt(P)=y+1, \quad&  \rt(Q)=y+1, \quad& \rt(P-Q)=1,\\
\om_2(P)=1,&\om_2(Q)=1,&\om_2(P-Q)=0.
\end{array}
$$
However, we shall say that $\om_r$ is a \emph{pseudo-valuation with respect to $v_r$}; this is justified by the following properties of $\om_r$.

\begin{proposition}\label{pseudo} 
Let $P(x),\,Q(x)\in\zpx$ be two nonzero polynomials such that $v_r(P)=v_r(Q)$. Then,
\begin{enumerate}
\item  $v_r(P-Q)>v_r(P)$ if and only if $S_{r-1}(P)=S_{r-1}(Q)$ and $R_{r-1}(P)=R_{r-1}(Q)$. In particular, $\om_r(P)=\om_r(Q)$ in this case.   
 \item If $\om_r(P)\ne\om_r(Q)$, then $\om_r(P-Q)=\min\{\om_r(P),\om_r(Q)\}$. 
\end{enumerate}
\end{proposition}

\begin{proof}
Let us denote  $N=N_{r-1}^-(P)$,  $N'=N_{r-1}^-(Q)$. Since $v_r(P)=v_r(Q)$, the pa\-rallel lines $L_{\la_{r-1}}(N)$, $L_{\la_{r-1}}(N')$ coincide and we can consider the shortest segment $T$ of $L_{\la_{r-1}}(N)$ that contains $S_{r-1}(P)$ and $S_{r-1}(Q)$. 

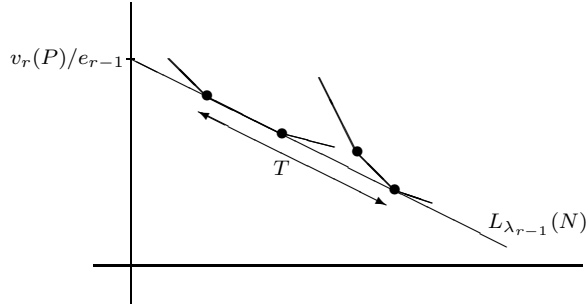
\begin{center}
\setlength{\unitlength}{5.mm}
\begin{picture}(12,7.2)
\put(5.85,1.85){$\bullet$}\put(4.85,2.85){$\bullet$}
\put(-2,0){\line(1,0){13}}
\put(6,2){\line(-1,1){1}}\put(6.02,2){\line(-1,1){1}}
\put(6,2){\line(3,-1){1}}\put(6.02,2){\line(3,-1){1}}
\put(5,3){\line(-1,2){1}}\put(5.02,3){\line(-1,2){1}}
\put(9,.5){\line(-2,1){10}}
\put(-1.1,5.5){\line(1,0){.2}}
\put(-1,-1){\line(0,1){8}}
\put(8.5,1){\begin{footnotesize}$L_{\la_{r-1}}(N)$\end{footnotesize}}
\put(-4.2,5.35){\begin{footnotesize}$v_r(P)/e_{r-1}$\end{footnotesize}}
\put(2.85,3.35){$\bullet$}\put(.85,4.35){$\bullet$}
\put(3,3.5){\line(-2,1){2}}\put(3.02,3.5){\line(-2,1){2}}
\put(3,3.5){\line(4,-1){1.4}}\put(3.02,3.5){\line(4,-1){1.4}}
\put(1,4.5){\line(-1,1){1}}\put(1.02,4.5){\line(-1,1){1}}
\put(3,3){\vector(-2,1){2.2}}\put(3,3){\vector(2,-1){2.8}}
\put(2.8,2.4){\begin{footnotesize}$T$\end{footnotesize}}
\end{picture}
\end{center}\be\medskip

By Lemma \ref{sumRr} in order $r-1$ (cf. (\ref{sumR}) in order one):
\begin{equation}\label{RPQ}
R_{r-1}(P-Q,T)=R_{r-1}(P,T)-R_{r-1}(Q,T).
\end{equation}

By (\ref{RvssubSr}) in order $r-1$ (cf. (\ref{enlarge1}) in order one), the double condition $S_{r-1}(P)=S_{r-1}(Q)$, 
$R_{r-1}(P)=R_{r-1}(Q)$, is equivalent to $R_{r-1}(P,T)=R_{r-1}(Q,T)$; that is, to $R_{r-1}(P-Q,T)=0$. This is equivalent to $N_{r-1}^-(P-Q)$ lying strictly above $L_{\la_{r-1}}(N)$, which is equivalent in turn to $v_r(P-Q)>v_r(P)$. This proves item 1.

By (\ref{RvssubSr}) and (\ref{enlarge1}), again, the equality (\ref{RPQ}) translates into 
$$
y^aR_{r-1}(P-Q)(y)=y^bR_{r-1}(P)(y)-y^cR_{r-1}(Q)(y),
$$
for certain nonnegative integers $a,b,c$. 
Since the residual polynomials are never divisible by $y$, and $\psi_{r-1}(y)\ne y$, from $\ord_{\psi_{r-1}}(R_{r-1}(P))<\ord_{\psi_{r-1}}(R_{r-1}(Q))$ we deduce $\ord_{\psi_{r-1}}(R_{r-1}(P-Q))=\ord_{\psi_{r-1}}(R_{r-1}(P))$.   This proves item 2.
\end{proof}

We can reinterpret the computation of $v(P(\t))$ given in item 5 of Proposition \ref{vqtr} in order $r-1$ (cf. Proposition \ref{vqt} for $r=2$), in terms of the pair $v_r,\,\om_r$.

\begin{proposition}\label{vpt}
Let $\t\in\qpb$ be a root of a polynomial in $\zpx$ of type $\ty$. Then, for any nonzero polynomial $P(x)\in\zpx$, 
$$v(P(\t))\ge v_r(P(x))/e_1\cdots e_{r-1},$$
and equality holds if and only if $\om_r(P)=0$. \hfill{$\Box$}
\end{proposition}

\subsection{Construction of a representative of $\ty$}
By Lemma \ref{typedegree}, a nonconstant polynomial of type $\ty$ has degree at least $m_r$. In this section  
we shall show how to construct in a effective (and recursive) way a polynomial $\phi_r(x)$ of type $\ty$ and minimal degree $m_r$. 

We first show how to construct a polynomial with  prescribed residual polynomial.
\begin{proposition}\label{construct}
Let $V$ be an integer, $V\ge e_{r-1}f_{r-1}v_r(\phi_{r-1})$.
Let $\varphi(y)\in\ff{r-1}[y]$ be a nonzero polynomial of degree less than $f_{r-1}$, and let $\nu=\ord_{y}(\varphi)$. 
Then, we can construct in an effective way a polynomial $P(x)\in\zpx$ satisfying the following properties
$$\dg P(x)<m_r,\qquad v_r(P)=V, \qquad y^{\nu}R_{r-1}(P)(y)=\varphi(y) .
$$ 
\end{proposition}

\begin{proof}
Let $L$ be the line of slope $\la_{r-1}$ with ordinate $V/e_{r-1}$ at the origin. By item 4 of Proposition \ref{propertiesv}, 
$V/e_{r-1}\ge f_{r-1}v_r(\phi_{r-1})\ge f_{r-1}h_{r-1}$; thus, the line $L$ cuts the horizontal axis at the abscissa $V/h_{r-1}\ge e_{r-1}f_{r-1}$. Let $T$ be the greatest side contained in $L$, whose end points have nonnegative integer coordinates. Let $(s,u)$ be the initial point of $T$ and denote $u_j:=u-jh_{r-1}$, for all $0\le j<f_{r-1}$, so that $(s+je_{r-1},u_j)$ lies on $L$. Clearly, $s<e_{r-1}$ and, for all $j$,
\begin{equation}\label{ef}
 j<f_{r-1},\,s<e_{r-1}\,\imp\,s+je_{r-1}<e_{r-1}f_{r-1}.
\end{equation}
Hence, $(s+je_{r-1},u_j)$ lies on $T$. 

Let $\varphi(y)=\sum_{0\le j<f_{r-1}}c_jy^j$, with $c_j\in\ff{r-1}$.
Select polynomials $c_j(y)\in\ff{r-2}[y]$ of degree less than $f_{r-2}$, such that $c_j$ is the class of $c_j(y)$ modulo $\psi_{r-2}(y)$, or equivalently, $c_j(z_{r-2})=c_j$. Let us construct $P(x)\in\zpx$ such that
$$
\dg P(x)<m_r,\quad v_r(P)=V,\quad \nu=(s_{r-1}(P)-s)/e_{r-1},\quad y^{\nu}R_{r-1}(P)(y)=\varphi(y) .
$$
We proceed by induction  on $r\ge 2$. For $r=2$ the polynomials $c_j(y)$ belong to $\ff{}[y]$; we abuse of language and denote by $c_j(x)\in\oo[x]$ the polynomials obtained by choosing arbitrary lifts to $\oo$ of the nonzero coefficients of $c_j(y)$. The polynomial
 $P(x)=\sum_{0\le j<f_{r-1}} \pi^{u-jh_1}c_j(x)\phi_1(x)^{s+je_1}$ satisfies the required properties. In fact, by (\ref{ef}),
$$\dg (c_j(x)\phi_1(x)^{s+je_1})<m_1+(e_1f_1-1)m_1=m_2,$$ for all $j$.
For the coefficients $c_j=0$ we take $c_j(x)=0$. For the coefficients $c_j\ne0$, we have $c_j(y)\ne 0$ and $v(c_j(x))=0$; hence, $v(\pi^{u-jh_1}c_j(x))=u-jh_1=u_j$. Thus, the coefficient $\pi^{u-jh_1}c_j(x)$ determines a point of $N_1^-(P)$ lying on $T$, and $v_2(P)=V$. Finally, it is clear by construction that $\nu=(s_1(P)-s)/e_1$ and $y^{\nu}R_1(P)(y)=R_1(P,T)(y)=\varphi(y)$.

Let now $r\ge 3$, and suppose that the proposition has been proved for orders $2,\dots,r-1$. For any $0\le j<f_{r-1}$, denote $V_j:=u_j-(s+je_{r-1})v_{r-1}(\phi_{r-1})$. Since $u=(V-sh_{r-1})/e_{r-1}$, we get
\begin{align*}
V_j=&\,\frac 1{e_{r-1}}\left(V-(s+je_{r-1})(e_{r-1}v_{r-1}(\phi_{r-1})+h_{r-1}\right)=(\mbox{by item 4 of Prop. \ref{propertiesv}})\\
=&\,\frac 1{e_{r-1}}\left(V-(s+je_{r-1})v_r(\phi_{r-1})\right)\ge\  (\mbox{by }(\ref{ef})) \\
\ge&\,\frac 1{e_{r-1}}\left(V-(e_{r-1}f_{r-1}-1)v_r(\phi_{r-1})\right)\ge\  (\mbox{by hypothesis})\\
\ge&\,\frac 1{e_{r-1}}v_r(\phi_{r-1})=v_{r-1}(\phi_{r-1})+\dfrac{h_{r-1}}{e_{r-1}}>v_{r-1}(\phi_{r-1})=e_{r-2}f_{r-2}v_{r-1}(\phi_{r-2}),  \end{align*}
the last equality by (\ref{vrphir}) below, in order $r-1$.

Let $L_j$ be the line of slope $\la_{r-2}$ with ordinate at the origin $V_j/e_{r-2}$. Let $T(j)$ be the greatest side contained in $L_j$, whose end points have nonnegative integer coordinates. Let $\mathfrak{s}_j$ be the initial abscissa of $T(j)$. Consider the unique polynomial $\varphi_j(y)\in\ff{r-2}[y]$, of degree less than $f_{r-2}$, such that 
$$\varphi_j(y)\equiv y^{(\ell_{r-2}u_j-\mathfrak{s}_j)/e_{r-2}}c_j(y) \md{\psi_{r-2}(y)},$$
and let $\nu_j=\ord_y(\varphi_j)$. By induction hypothesis, we are able to construct a polynomial $P_j(x)$ of degree less than $m_{r-1}$, with $v_{r-1}(P_j)=V_j$, $\nu_j=(s_{r-2}(P_j)-\mathfrak{s}_j)/e_{r-2}$, and $y^{\nu_j}R_{r-2}(P_j)(y)=\varphi_j(y)$ in $\ff{r-2}[y]$.\be

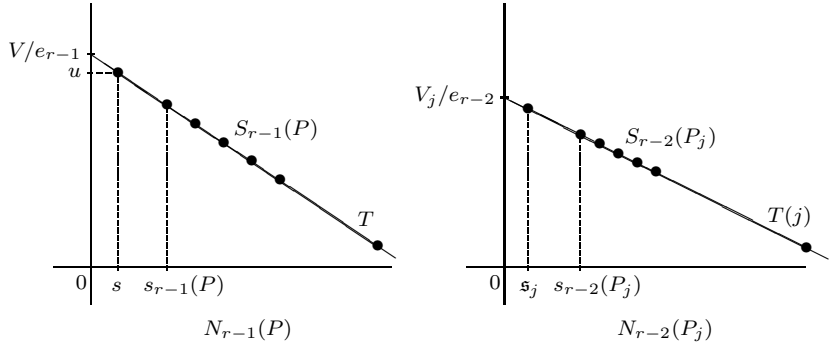
\begin{center}
\setlength{\unitlength}{5.mm}
\begin{picture}(20,7)
\put(7.45,.4){$\bullet$}\put(.55,5){$\bullet$}
\put(1.85,4.15){$\bullet$}\put(2.6,3.65){$\bullet$}\put(3.35,3.15){$\bullet$}\put(4.1,2.65){$\bullet$}
\put(4.85,2.15){$\bullet$}
\put(-1,0){\line(1,0){9}}
\put(0,-1){\line(0,1){8}}
\put(8.1,.25){\line(-3,2){8.1}}\put(7.5,.62){\line(-3,2){6.7}}
\multiput(2,-.1)(0,.20){22}{\vrule height2pt}
\multiput(.7,-.1)(0,.20){27}{\vrule height2pt}
\multiput(-.1,5.15)(.25,0){4}{\hbox to 2pt{\hrulefill }}
\put(-.1,5.65){\line(1,0){.2}}
\put(7.1,1.1){\begin{footnotesize}$T$\end{footnotesize}}
\put(3.8,3.5){\begin{footnotesize}$S_{r-1}(P)$\end{footnotesize}}
\put(-2.2,5.6){\begin{footnotesize}$V/e_{r-1}$\end{footnotesize}}
\put(-.6,5.05){\begin{footnotesize}$u$\end{footnotesize}}
\put(.55,-.65){\begin{footnotesize}$s$\end{footnotesize}}
\put(1.4,-.65){\begin{footnotesize}$s_{r-1}(P)$\end{footnotesize}}
\put(-.4,-.6){\begin{footnotesize}$0$\end{footnotesize}}
\put(3,-1.8){\begin{footnotesize}$N_{r-1}(P)$\end{footnotesize}}
\put(10,0){\line(1,0){9}}
\put(19.6,0.2){\line(-2,1){8.6}}
\put(10.9,4.5){\line(1,0){.2}}
\put(8.6,4.4){\begin{footnotesize}$V_j/e_{r-2}$\end{footnotesize}}
\put(11,-1){\line(0,1){8}}
\put(14.85,2.35){$\bullet$}\put(14.35,2.6){$\bullet$}\put(13.85,2.85){$\bullet$}\put(13.35,3.1){$\bullet$}\put(12.85,3.35){$\bullet$}
\put(18.85,.35){$\bullet$}\put(11.45,4.05){$\bullet$}
\put(19,.53){\line(-2,1){7.3}}
\put(18,1.2){\begin{footnotesize}$T(j)$\end{footnotesize}}
\put(14.2,3.3){\begin{footnotesize}$S_{r-2}(P_j)$\end{footnotesize}}
\put(14,-1.8){\begin{footnotesize}$N_{r-2}(P_j)$\end{footnotesize}}
\multiput(13,-.1)(0,.20){19}{\vrule height2pt}
\multiput(11.6,-.1)(0,.20){22}{\vrule height2pt}
\put(11.4,-.65){\begin{footnotesize}$\mathfrak{s}_j$\end{footnotesize}}
\put(12.3,-.65){\begin{footnotesize}$s_{r-2}(P_j)$\end{footnotesize}}
\put(10.6,-.6){\begin{footnotesize}$0$\end{footnotesize}}
\end{picture}
\end{center}\be\be

The polynomial $P(x)$ we are looking for is:
$$
P(x)=\sum_{0\le j<f_{r-1}}P_j(x)\phi_{r-1}(x)^{s+je_{r-1}}\in\zpx.
$$
In fact, by (\ref{ef}), $\dg(P_j(x)\phi_{r-1}(x)^{s+je_{r-1}})<m_{r-1}+(e_{r-1}f_{r-1}-1)m_1=m_r$, for all $j$. If $P_j(x)\ne 0$, then
$v_{r-1}(P_j(x)\phi_{r-1}(x)^{s+je_{r-1}})=V_j+(s+je_{r-1})v_{r-1}(\phi_{r-1})=u_j$, so that all of these coefficients determine points of $N_{r-1}^-(P)$ lying on $T$; this shows that $v_r(P)=V$. For $c_j=0$ we take $P_j(x)=0$; hence, $\nu=(s_{r-1}(P)-s)/e_{r-1}$, and by (\ref{RvssubSr}) in order $r-1$:  
$$y^{\nu}R_{r-1}(P)(y)=R_{r-1}(P,T)(y)=\sum_{P_j(x)\ne0}(z_{r-2})^{t(j)}R_{r-2}(P_j)(z_{r-2})y^j,$$ where $t(j):=t_{r-2}(s+je_{r-1})=(s_{r-2}(P_j)-\ell_{r-2}u_j)/e_{r-2}$ (cf. Definition \ref{t(i)}). Finally,
\begin{align*}
(z_{r-2})^{t(j)}R_{r-2}(P_j)(z_{r-2})=&\, (z_{r-2})^{t(j)-\nu_j}\varphi_j(z_{r-2})\\=&\, (z_{r-2})^{t(j)-\nu_j+\frac{\ell_{r-2}u_j-\mathfrak{s}_j}{e_{r-2}}}c_j(z_{r-2})=c_j,
\end{align*}
so that $y^{\nu}R_{r-1}(P)(y)=\varphi(y)$.
\end{proof}

\begin{theorem}\label{phir}
We can effectively construct a monic polynomial $\phi_r(x)$ of type $\ty$ such that $R_{r-1}(\phi_r)(y)\sim\psi_{r-1}(y)$. This polynomial is irreducible over $\zpx$ and it satisfies 
\begin{equation}\label{vrphir}
 \dg \phi_r=m_r,\quad \om_r(\phi_r)=1,\quad  v_r(\phi_r)=e_{r-1}f_{r-1}v_r(\phi_{r-1}).
\end{equation}
\end{theorem}

\begin{proof}
By Lemma \ref{shape} in order $r-1$ (cf. Remark \ref{phiell} if $r=2$), the polygon $N_{r-1}(\phi_{r-1})$ is one-sided with slope $-\infty$ and final point $(1,v_{r-1}(\phi_{r-1}))$. Therefore, $S_{r-1}(\phi_{r-1})$ reduces to a point and $R_{r-1}(\phi_{r-1})(y)=c_1$, where $c_1$ is the residual coefficient of this polygon (cf. Definition \ref{rescoeff}):
$$
c_1=\left\{\begin{array}{ll}
1,&\mbox{ if }r=2,\\(z_{r-2})^{-\ell_{r-2}v_{r-1}(\phi_{r-1})/e_{r-2}},&\mbox{ if }r>2.
\end{array}
\right.
$$

Denote $c:=c_1^{e_{r-1}f_{r-1}}$. The polynomial $\varphi(y):=c(\psi_{r-1}(y)-y^{f_{r-1}})$ has degree less than $f_{r-1}$, and $\nu=\ord_y(\varphi)=0$. Let $P(x)$ be the polynomial attached by Proposition \ref{construct} to $\varphi(y)$ and $V=e_{r-1}f_{r-1}v_r(\phi_{r-1})$.
Since $\dg(P(x))<m_r$, the polynomial $\phi_r(x):=\phi_{r-1}(x)^{e_{r-1}f_{r-1}}+P(x)$ is monic and it has degree $m_r$.
Let $T$ be the auxiliary side used in the construction of $P(x)$; we saw along the proof of Proposition \ref{construct} that $R_{r-1}(P)(y)=\varphi(y)=R_{r-1}(P,T)(y)$. By (\ref{RvssubSr}) (and (\ref{enlarge1}) if $r=2$), $S_{r-1}(P)$ has the same initial point than $T$ and $R_{r-1}(\phi_r,T)(y)=R_{r-1}(\phi_r)(y)$. Also, 
$$R_{r-1}(\phi_{r-1}^{e_{r-1}f_{r-1}},T)(y)=y^{f_{r-1}}R_{r-1}(\phi_{r-1}^{e_{r-1}f_{r-1}})(y)=cy^{f_{r-1}}.
$$ Finally, by Lemma \ref{sumRr} in order $r-1$ (cf. (\ref{sumR}) in order one):
\begin{align*}
R_{r-1}(\phi_r,T)(y)=&\,R_{r-1}(\phi_{r-1}^{e_{r-1}f_{r-1}},T)(y)+R_{r-1}(P,T)(y)\\=&\,cy^{f_{r-1}}+\varphi(y)=c\psi_{r-1}(y),
\end{align*}
so that $R_{r-1}(\phi_r)(y)\sim \psi_{r-1}(y)$ and $\om_r(\phi_r)=1$. The polynomial $\phi_r(x)$ is irreducible over $\zpx$ by the Theorem of the product in order $r-1$. Finally, it has $v_r(\phi_r)=V$ because all points of $N_{r-1}(\phi_r)$ lie on $T$.
\end{proof}

\begin{definition}
A \emph{representative} of the type $\ty$ is a monic polynomial $\phi_r(x)\in\zpx$ of type $\ty$ such that $R_{r-1}(\phi_r)(y)\sim\psi_{r-1}(y)$. This object plays the analogous role in order $r-1$ to that of an irreducible polynomial modulo $\m$ in order one.
\end{definition}

From now on, we fix a representative $\phi_r(x)$ of $\ty$, without necessarily assuming that it has been constructed by the method of Propositon \ref{construct}. We denote by  $\tilde{\ty}=(\phi_1(x);\la_1,\phi_2(x);\cdots;\la_{r-2},\phi_{r-1}(x);\la_{r-1},\phi_r(x))$, the extension of $\ty$, which is half-way in the process of extending $\ty$ to different types of order $r$.
 
\subsection{Certain rational functions}
We introduce in a recursive way several rational functions in $K(x)$. We let $h_r,e_r$ be arbitrary coprime positive integers, and we fix $\ell_r,\ell'_r\in\Z$ such that $\ell_rh_r-\ell'_re_r=1$.
\begin{definition}\label{ratfracs}
We define $\,\pi_0(x)=1$, $\pi_1(x)=\pi$, and, for all $1\le i\le r$, 
$$
\Phi_i(x)=\dfrac{\phi_i(x)}{\pi_{i-1}(x)^{f_{i-1}v_i(\phi_{i-1})}},\qquad 
\ga_i(x)=\dfrac{\Phi_i(x)^{e_i}}{\pi_i(x)^{h_i}},\qquad 
\pi_{i+1}(x)=\dfrac{\Phi_i(x)^{\ell_i}}{\pi_i(x)^{\ell'_i}}.
$$
\end{definition}

Each of these rational functions can be written as
$\pi^{n_0}\phi_1(x)^{n_1}\cdots\phi_r(x)^{n_r}$, for adequate integers $n_i\in\Z$. Also, 
\begin{equation}\label{phii}
\Phi_i(x)=\cdots\phi_i(x),\quad \ga_i(x)=\cdots \phi_i(x)^{e_i},\quad\pi_{i+1}(x)=\cdots \phi_i(x)^{\ell_i},
\end{equation}
where the dots indicate a product of integral powers of $\pi$ and $\phi_j(x)$, with $1\le j<i$.  
We want to compute the value of $v_r$ on all these functions.

\begin{lemma}\label{omji}For all $1\le i<j\le r$, we have $\om_j(\phi_i)=0$.
\end{lemma}

\begin{proof}
Since $N_i(\phi_i)$ is a side of slope $-\infty$, we have $\om_{i+1}(\phi_i)=0$ because $S_i(\phi_i)$ reduces to a point. By Lemma \ref{typedegree}, $\om_j(\phi_i)=0$ for all $i<j\le r$.
\end{proof}

\begin{proposition}\label{vrphii}
For all $1\le i<r$ we have
\begin{enumerate}
 \item $v_r(\phi_i)=\sum_{j=1}^i\left(e_{j+1}\cdots e_{r-1}\right)\left(e_jf_j\cdots e_{i-1}f_{i-1}\right)h_j$,
\item  $v_r(\Phi_i)=e_{i+1}\cdots e_{r-1}h_i$,
\item $v_r(\pi_{i+1})=e_{i+1}\cdots e_{r-1}$,
\item $v_r(\ga_i)=0$.
\item $\om_{r}(\phi_i)=\om_{r}(\Phi_i)=\om_{r}(\ga_i)=\om_{r}(\pi_{i+1})=0$.
\end{enumerate}
Moreover, $v_r(\phi_r)=\sum_{j=1}^{r-1}\left(e_{j+1}\cdots e_{r-1}\right)\left(e_jf_j\cdots e_{r-1}f_{r-1}\right)h_j$ and $v_r(\Phi_r)=0$.
\end{proposition}

\begin{proof}
We proceed by induction on $r$. For $r=2$ all formulas are easily deduced from $v_2(\phi_1)=h_1$, that was proved in Proposition \ref{propertiesv}. Suppose $r\ge 3$ and all statements true for $r-1$. 

Let us start with item 1. By Proposition \ref{propertiesv} and (\ref{vrphir}),  $$v_r(\phi_{r-1})=h_{r-1}+e_{r-1}v_{r-1}(\phi_{r-1}), \qquad v_{r-1}(\phi_{r-1})=e_{r-2}f_{r-2}v_{r-1}(\phi_{r-2}).$$ Hence, the formula for $i=r-1$ follows from the induction hypothesis. Suppose from now on $i<r-1$. By Lemma \ref{omji}, $\phi_i(x)=\phi_i(x)$ is an admissible $\phi_{r-1}$-adic development of $\phi_i(x)$, and by Lemma \ref{admissible2} in order $r-1$ (Lemma \ref{admissible} in order one) we get $N_{r-1}^-(\phi_i)=(0,v_{r-1}(\phi_i))$, so that 
$v_r(\phi_i)=e_{r-1} v_{r-1}(\phi_i)$ and the formula follows by induction. 

Let us prove now simultaneously items 2 and 3 by induction on $i$. For $i=1$ we have  by item 1, 
$$\as{1.2}
\begin{array}{l}
v_r(\Phi_1)=v_r(\phi_1)=e_2\cdots e_{r-1}h_1,\\
v_r(\pi_2)=\ell_1v_r(\Phi_1)-\ell'_1v_r(\pi)=(\ell_1h_1-\ell'_1e_1)e_2\cdots e_{r-1}=e_2\cdots e_{r-1}.
\end{array}
$$
Suppose now $i>1$ and the formulas hold for $1,\dots,i-1$. 
$$\as{1.2}
\begin{array}{l}
v_r(\Phi_i)=v_r(\phi_i)-f_{i-1}v_i(\phi_{i-1})e_{i-1}\cdots e_{r-1}=e_{i+1}\cdots e_{r-1}h_i,\\
v_r(\pi_{i+1})=\ell_i v_r(\Phi_i)-\ell'_iv_r(\pi_i)=(\ell_ih_i-\ell'_ie_i)e_{i+1}\cdots e_{r-1}=e_{i+1}\cdots e_{r-1}.
\end{array}
$$     

Item 4 is easily deduced from the previous formulas, and item 5 is an immediate consequence of (\ref{phii}) and Lemma \ref{omji}. The last statements follow from (\ref{vrphir}) and the previous formulas.
\end{proof}

\begin{lemma}\label{gammas0}
For $\n=(n_0,\dots,n_{r-1})\in\Z^r$, consider the rational function $\Phi(\n)=\pi^{n_0}\phi_1(x)^{n_1}\cdots\phi_{r-1}(x)^{n_{r-1}}\in K(x)$. Then, if $v_r(\Phi(\n))=0$, there exists a unique sequence $i_1,\dots,i_{r-1}$ of integers such that $\Phi(\n)=\ga_1(x)^{i_1}\cdots \ga_{r-1}(x)^{i_{r-1}}$. Moreover, $i_s$ depends only on $n_s,\dots,n_{r-1}$, for all $1\le s<r$.  
\end{lemma}

\begin{proof}
Since the polynomials $\phi_s(x)$ are irreducible and pairwise different, we have $\Phi(\n)=\Phi(\n')$ if and only if $\n=\n'$.
By (\ref{phii}), any product $\ga_1(x)^{i_1}\cdots \ga_{r-1}(x)^{i_{r-1}}$ can be expressed as $\Phi(\j)$, for a suitable $\j=(j_0,\dots,j_{r-2},e_{r-1}i_{r-1})$. Thus, if $\ga_1(x)^{i_1}\cdots \ga_{r-1}(x)^{i_{r-1}}=1$ we have necessarily $i_{r-1}=0$, and recursively, $i_1=\dots=i_{r-2}=0$. This proves the unicity of the expression of any $\Phi(\n)$ as a product of powers of gammas.

Let us prove the existence of such an expression by induction on $r\ge 1$. For $r=1$, let $\n=(n_0)$; the condition $v_r(\pi^{n_0})=0$ implies
$n_0=0$ and $\Phi(\n)=1$. Suppose $r\ge 2$ and the lemma proven for all $\n'\in\Z^{r-1}$. By item 1 of the last proposition, $v_r(\Phi(\n))\equiv n_{r-1}h_{r-1}\md{e_{r-1}}$; hence, if $v_r(\Phi(\n))=0$ we have necessarily $n_{r-1}=e_{r-1}i_{r-1}$ for some integer $i_{r-1}$ that depends only on $n_{r-1}$. By (\ref{phii}), $\ga_{r-1}(x)^{i_{r-1}}=\Phi(\j)$, for some $\j=(j_0,\dots,j_{r-2},e_{r-1}i_{r-1})$; hence, $\Phi(\n)\ga_{r-1}(x)^{-i_{r-1}}=\Phi(\n')$, with $\n'=(n'_0,\dots,n'_{r-2},0)$, and each $n'_s$ depends only on $n_s$ and $n_{r-1}$. By item 4 of the last proposition, we have still $v_r(\Phi(\n'))=0$, and by induction hypothesis we get the desired expression of $\Phi(\n)$ as a product of powers of gammas.
\end{proof}

\subsection{Newton polygon and residual polynomials of $r$-th order}\label{NPRr}
Let $f(x)\in\zpx$ be a nonzero polynomial, and consider its unique $\phi_r$-adic development
\begin{equation}\label{phiradic}
f(x)=\sum_{0\le i\le \lfloor \dg(f)/m_r\rfloor} a_i(x)\phi_r(x)^i,\quad \dg a_i(x)<m_r.
\end{equation}

We define the Newton polygon $N_r(f)$ of $f(x)$, with respect to the extension $\tilde{\ty}$ of $\ty$, to be the lower convex envelope of the set of points $(i,u_i)$, where 
$$
u_i:=v_r(a_i(x)\phi_r(x)^i)=v_r(a_i(x))+iv_r(\phi_r(x)).
$$ 
Note that we consider the $v_r$-value of the whole monomial $a_i(x)\phi_r(x)^i$. Actually, we did the same for the Newton polygons of first order, but in that case $v_1(a_i(x)\phi_1(x)^i)=v_1(a_i(x))$, because $v_1(\phi_1(x))=0$.

The principal part $N_r^-(f)$ is the principal polygon formed by all sides of negative slope, including the side of slope $-\infty$ if $f(x)$ is divisible by $\phi_r(x)$ in $\zpx$. The typical shape of the polygon is the following

\begin{center}
\setlength{\unitlength}{5.mm}
\begin{picture}(14,9)
\put(10.85,1.85){$\bullet$}\put(8.85,2.85){$\bullet$}\put(7.85,1.85){$\bullet$}\put(5.85,3.85){$\bullet$}
\put(4.85,3.85){$\bullet$}\put(3.85,5.85){$\bullet$}\put(2.85,7.85){$\bullet$}\put(-1,0){\line(1,0){15}}
\put(0,-1){\line(0,1){9}}
\put(8,2){\line(-3,2){3}}\put(5,4){\line(-1,2){2}}\put(8,2.03){\line(-3,2){3}}
\put(5,4.03){\line(-1,2){2}}\put(11,2){\line(-1,0){3}}\put(11,2.02){\line(-1,0){3}}
\put(12.85,3.85){$\bullet$}\put(11,2){\line(1,1){2}}\put(11,2.02){\line(1,1){2}}
\multiput(13,-.1)(0,.25){16}{\vrule height2pt}
\multiput(3,-.1)(0,.25){32}{\vrule height2pt}
\multiput(8,-.1)(0,.25){9}{\vrule height2pt}
\put(11.7,-.7){\begin{footnotesize}$\lfloor \dg(f)/m_r\rfloor$\end{footnotesize}}
\put(7.4,-.7){\begin{footnotesize}$\om_r(f)$\end{footnotesize}}
\put(2.1,-.7){\begin{footnotesize}$\ord_{\phi_r}(f)$\end{footnotesize}}
\multiput(-.1,2)(.25,0){55}{\hbox to 2pt{\hrulefill }}
\put(-1.6,1.8){\begin{footnotesize}$v_r(f)$\end{footnotesize}}
\put(-.4,-.6){\begin{footnotesize}$0$\end{footnotesize}}
\put(6,-2){$N_r(f)$}
\end{picture}
\end{center}\be\be\be
 
\begin{lemma}\label{shape}
 \begin{enumerate}
  \item $\min_{0\le i\le n}\{u_i\}=v_r(f)$, where $n:=\ell(N_r(f))=\lfloor \deg f/m_r\rfloor$.
\item The length of $N_r^-(f)$ is $\om_r(f)$.
\item The side of slope $-\infty$ of $N_r^-(f)$ has length $\ord_{\phi_r}(f)$. 
 \end{enumerate}
\end{lemma}

\begin{proof}
The third item is obvious. Let us prove items 1, 2. Let $u:=\min_{0\le i\le n}\{u_i\}$, and consider the polynomial
$$
g(x):=\sum_{u_i=u}a_i(x)\phi_r(x)^i.
$$
All monomials of $g(x)$ have the same $v_r$-value and a different $\om_r$-value:
$$
\om_r(a_i(x)\phi_r(x)^i)=\om_r(a_i(x))+\om_r(\phi_r(x)^i)=i,
$$
because $\om_r(a_i)=0$ by Lemma \ref{typedegree}.
By item 2 of Proposition \ref{pseudo}, $v_r(g)=u$ and $\om_r(g)=i_0$, the least abscissa with $u_{i_0}=u$. Since, $v_r(f-g)>u$, we have $v_r(f)=v_r(g)=u$, and this proves item 1. 
On the other hand, item 1 of Propositon \ref{pseudo} shows that $\om_r(f)=\om_r(g)=i_0$, and this proves item 2. 
\end{proof}

The following observation is a consequence of Lemmas \ref{typedegree} and \ref{shape}.
\begin{corollary}\label{rminus}
If $f(x)$ has type $\ty$ then $N_r(f)=N_r^-(f)$. \hfill{$\Box$}
\end{corollary}

From now on let $N=N_r^-(f)$. As we did in order one, we attach to any integer abscissa $i$ of the finite part of $N$ a \emph{residual coefficient} $c_i\in\ff{r}$. The natural idea is to consider 
$c_i=R_{r-1}(a_i)(z_{r-1})$ for the points lying on $N$. However, this does not lead to the right concept of residual polynomial attached to a side; it is necessary to twist these coefficients by certain powers of $z_{r-1}$. 

\begin{definition}\label{t(i)}
For any nonzero $a(x)\in\zpx$ and any integer $i\ge 0$, we denote 
$$t_{r-1}(a)_i:=\dfrac{s_{r-1}(a)-\ell_{r-1} v_r(a\phi_r^i)}{e_{r-1}}.
$$ 

For any nonzero $f(x)\in\zpx$ with $\phi_r$-adic development (\ref{phiradic}), we denote
$$t_{r-1}(i):=t_{r-1}(i,f):=t_{r-1}(a_i)_i=\dfrac{s_{r-1}(a_i)-\ell_{r-1} u_i}{e_{r-1}}.
$$ 
\end{definition}

The number $t_{r-1}(a)_i$ is always an integer. In fact, if $u_{r-1}(a)$ denotes the ordinate of the initial point of $S_{r-1}(a)$, then
\begin{align*}
v_r(a\phi_r^i)=v_r(a)+iv_r(\phi_r)&\equiv v_r(a)=h_{r-1}s_{r-1}(a)+e_{r-1}u_{r-1}(a)\md{e_{r-1}}\\&\equiv h_{r-1}s_{r-1}(a)\md{e_{r-1}},
\end{align*}
the first congruence by (\ref{vrphir}). Hence, $\ell_{r-1}v_r(a\phi_r^i)\equiv s_{r-1}(a)\md{e_{r-1}}$.

\begin{definition}\label{rescoeff}For any integer abscissa $\ord_{\phi_r}(f)\le i\le \om_r(f)$, 
the \emph{residual coefficient} $c_i$ of $N:=N_r^-(f)$ is defined to be:
$$\as{1.4}
c_i:=c_i(f):=\left\{\begin{array}{ll}
0,&\mbox{ if $(i,u_i)$ lies strictly above }N,\\z_{r-1}^{t_{r-1}(i)}R_{r-1}(a_i)(z_{r-1})\in\ff{r},&\mbox{ if $(i,u_i)$ lies on }N
\end{array}
\right.
$$ 
\end{definition}
Note that $c_i\ne0$ if $(i,u_i)$ lies on $N$ because $\om_r(a_i)=0$ and $\psi_{r-1}(y)$ is the minimal polynomial of $z_{r-1}$ over $\ff{r-1}$.
 
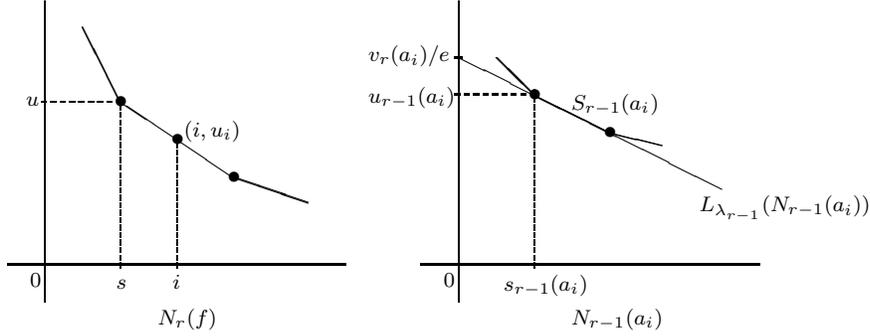
\begin{center}
\setlength{\unitlength}{5.mm}
\begin{picture}(20,8)
\put(1.85,4.15){$\bullet$}\put(3.35,3.15){$\bullet$}\put(4.85,2.15){$\bullet$}
\put(2,4.3){\line(-1,2){1}}\put(2,4.32){\line(-1,2){1}}
\put(5,2.3){\line(3,-1){2}}\put(5,2.32){\line(3,-1){2}}
\put(-1,0){\line(1,0){9}}
\put(0,-1){\line(0,1){8}}
\put(5,2.32){\line(-3,2){3}}
\multiput(2,-.1)(0,.20){22}{\vrule height2pt}
\multiput(3.5,-.1)(0,.20){17}{\vrule height2pt}
\multiput(-.1,4.3)(.25,0){8}{\hbox to 2pt{\hrulefill }}
\put(3.7,3.4){\begin{footnotesize}$(i,u_i)$\end{footnotesize}}
\put(-.5,4.2){\begin{footnotesize}$u$\end{footnotesize}}
\put(3.4,-.65){\begin{footnotesize}$i$\end{footnotesize}}
\put(1.9,-.65){\begin{footnotesize}$s$\end{footnotesize}}
\put(-.4,-.6){\begin{footnotesize}$0$\end{footnotesize}}
\put(3,-1.6){\begin{footnotesize}$N_r(f)$\end{footnotesize}}
\put(10,0){\line(1,0){9}}
\put(18,2){\line(-2,1){7}}
\put(10.9,5.5){\line(1,0){.2}}
\put(8.6,5.4){\begin{footnotesize}$v_r(a_i)/e$\end{footnotesize}}
\put(11,-1){\line(0,1){8}}
\put(14.85,3.35){$\bullet$}\put(12.85,4.35){$\bullet$}
\put(15,3.5){\line(-2,1){2}}\put(15.02,3.5){\line(-2,1){2}}
\put(15,3.5){\line(4,-1){1.4}}\put(15.02,3.5){\line(4,-1){1.4}}
\put(13,4.5){\line(-1,1){1}}\put(13.02,4.5){\line(-1,1){1}}
\put(14,-1.6){\begin{footnotesize}$N_{r-1}(a_i)$\end{footnotesize}}
\multiput(13,-.1)(0,.20){24}{\vrule height2pt}
\multiput(10.9,4.5)(.25,0){8}{\hbox to 2pt{\hrulefill }}
\put(12.2,-.7){\begin{footnotesize}$s_{r-1}(a_i)$\end{footnotesize}}
\put(8.6,4.3){\begin{footnotesize}$u_{r-1}(a_i)$\end{footnotesize}}
\put(17.4,1.4){\begin{footnotesize}$L_{\la_{r-1}}(N_{r-1}(a_i))$\end{footnotesize}}
\put(14,4.1){\begin{footnotesize}$S_{r-1}(a_i)$\end{footnotesize}}
\put(10.6,-.6){\begin{footnotesize}$0$\end{footnotesize}}
\end{picture}
\end{center}\be\be\be

\begin{definition}\label{defresidual}
Let $\la_r=-h_r/e_r$ be a negative rational number, with $h_r,e_r$ positive coprime integers. Let $S=S_{\la_r}(N)$ be the $\la_r$-component of $N$, $d=d(S)$ the degree, and $(s,u)$ the initial point of $S$. 

We define the \emph{virtual factor} of $f(x)$ attached to $S$ (or to $\la_r$) to be the rational function
$$
f^S(x):=\Phi_r(x)^{-s}\pi_r(x)^{-u}f^0(x)\in K(x),\quad f^0(x):=\sum_{(i,u_i)\in S}a_i(x)\phi_r(x)^i,
$$where $\Phi_r(x)$, $\pi_r(x)$ are the rational functions introduced in Definition \ref{ratfracs}.

We define the \emph{residual polynomial} attached to $S$ (or to $\la_r$) to be the polynomial: 
$$
R_{\la_r}(f)(y):=c_s+c_{s+e_r}\,y+\cdots+c_{s+(d-1)e_r}\,y^{d-1}+c_{s+de_r}\,y^d\in\ff{r}[y].
$$
\end{definition}

Only the points $(i,u_i)$ that lie on $S$ yield a non-zero coefficient of $R_{\la_r}(f)(y)$. In particular, $c_s$ and $c_{s+de}$ are always nonzero, so that $R_{\la_r}(f)(y)$ has degree $d$ and it is never  divisible by $y$. We emphasize that $R_{\la_r}(f)(y)$ does not depend only on $\la_r$; as all other objects in Sect.\ref{secNPr}, it depends on $\ty$ too.   

We define in an analogous way the residual polynomial of $f(x)$ with respect to a side $T$ that is not necessarily a $\la_r$-component of $N$. Let $T\in\ss(\la_r)$ be an arbitrary side of slope $\la_r$, with abscissas $s_0\le s_1$ for the end points. Let $d'=d(T)$. We say that $f(x)$ \emph{lies above $T$ in order $r$} if all points of $N$ with abscissa $s_0\le i\le s_1$ lie above $T$. In this case we define
$$
R_{\la_r}(f,T)(y):=\tilde{c}_{s_0}+\tilde{c}_{s_0+e_r}\,y+\cdots+\tilde{c}_{s_0+(d'-1)e_r}\,y^{d'-1}+
\tilde{c}_{s_0+d'e_r}\,y^{d'}\in\ff{r}[y],
$$
where $\tilde{c}_i:=\tilde{c}_i(f):=c_i$ if $(i,u_i)$ lies on $T$ and $\tilde{c}_i=0$ otherwise. 

Note that $\dg R_{\la_r}(f,T)(y)\le d'$ and equality holds if and only if the final point of $T$ belongs to $S_{\la_r}(f)$. Usually, $T$ will be an enlargement of $S_{\la_r}(f)$ and then,
\begin{equation}\label{RvssubSr}
T\supseteq S_{\la_r}(f)\ \imp\ R_{\la_r}(f,T)(y)=y^{(s-s_0)/e_r}R_{\la_r}(f)(y), 
\end{equation}
where $s$ is the abscissa of the initial point of $S_{\la_r}(f)$.

For technical reasons, we express $c_i$ in terms of a residual polynomial attached to certain auxiliary side.

\begin{lemma}\label{enlarge}
Let $N\in\pol$ be a principal polygon. Let $(i,y_i(N))$  be a point lying on $N$ and with integer abscissa $i$.
Let $V=y_i(N)-iv_r(\phi_r)$, and let $L_{\la_{r-1}}$ be the line of slope $\la_{r-1}$ that cuts the vertical axis at the point with ordinate $V/e_{r-1}$.  Denote by $T(i)$ the greatest side contained in $L_{\la_{r-1}}$, whose end points have nonnegative integer coordinates, and let
 $\mathfrak{s}_i$ be the abscissa of the initial point of $T(i)$.

Let  $a(x)\in\zpx$ be a nonzero polynomial such that $u_i:=v_r(a\phi_r^i)\ge y_i(N)$. Then,
$$\as{1.4}
y^{(\mathfrak{s}_i-\ell_{r-1} u_i)/e_{r-1}}R_{r-1}(a,T(i))(y)=
\left\{\begin{array}{ll}
0,&\mbox{ if }u_i>y_i(N),\\y^{t_{r-1}(a)_i}R_{r-1}(a)(y),&\mbox{ if }u_i=y_i(N).
\end{array}
\right.
$$
In particular, $c_i=z_{r-1}^{(\mathfrak{s}_i-\ell_{r-1} u_i)/e_{r-1}}R_{r-1}(a_i,T(i))(z_{r-1})$.
\end{lemma}

\begin{proof}
If $v_r(a\phi_r^i)=y_i(N)$, we have $v_r(a)=V$ and $S_{r-1}(a)\subseteq T(i)$. Then, the lemma follows from (\ref{RvssubSr}) in order $r-1$. If $v_r(a\phi_r^i)>y_i(N)$ then $S_{r-1}(a)$ lies strictly above $T(i)$ and $R_{r-1}(a_i,T(i))(y)=0$. 
\end{proof}

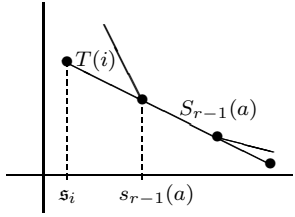
\begin{center}
\setlength{\unitlength}{5.mm}
\begin{picture}(6,5)
\put(5.25,.15){$\bullet$}\put(3.85,.85){$\bullet$}\put(1.85,1.85){$\bullet$}\put(-.15,2.85){$\bullet$}
\put(-1.6,0){\line(1,0){7.6}}\put(-.6,-1){\line(0,1){5.6}}
\put(5.4,.3){\line(-2,1){5.5}}\put(5.42,.3){\line(-2,1){5.5}}
\put(4,1){\line(4,-1){1.5}}\put(4.02,1){\line(4,-1){1.5}}
\put(2,2){\line(-1,2){1}}\put(2.02,2){\line(-1,2){1}}
\put(3,1.6){\begin{footnotesize}$S_{r-1}(a)$\end{footnotesize}}
\put(.25,2.9){\begin{footnotesize}$T(i)$\end{footnotesize}}
\multiput(0,-.1)(0,.25){13}{\vrule height2pt}
\multiput(2,-.1)(0,.25){9}{\vrule height2pt}
\put(-.2,-.65){\begin{footnotesize}$\mathfrak{s}_i$\end{footnotesize}}
\put(1.4,-.65){\begin{footnotesize}$s_{r-1}(a)$\end{footnotesize}}
\end{picture}
\end{center}\be

\begin{lemma}\label{sumRr}
Let $T\in\ss(\la_r)$ be a side of slope $\la_r$ and let $f(x),g(x)\in\zpx$. If $f(x)$ and $g(x)$ lie above $T$ in order $r$, then $(f+g)(x)$ lies above $T$ in order $r$ and 
$$
R_{\la_r}(f+g,T)=R_{\la_r}(f,T)+R_{\la_r}(g,T).
$$
\end{lemma}

\begin{proof}
Let $s_0\le s_1$ be the abscissas of the end points of $T$. We want to check that, for all integers $s_0\le i\le s_1$, 
\begin{equation}\label{sumci}
 \tilde{c}_i(f+g)=\tilde{c}_i(f)+\tilde{c}_i(g).
\end{equation}
Let $a_i(x)$, $b_i(x)$, be the respective $i$-th coefficients of the $\phi_r$-adic development of $f(x)$, $g(x)$; then, $a_i(x)+b_i(x)$ is the $i$-th coefficient of the $\phi_r$-adic development of $f(x)+g(x)$. 
By Lemma \ref{enlarge} applied to the point $(i,y_i(T))$ of $T$, $$\tilde{c}_i(f)=z_{r-1}^{(\mathfrak{s}_i-\ell_{r-1} u_i)/e_{r-1}}R_{r-1}(a_i,T(i))(z_{r-1}).$$
 Analogous equalities hold for $g(x)$ and $(f+g)(x)$, and (\ref{sumci})
follows from Lemma \ref{sumRr} itself, in order $r-1$ (cf. (\ref{sumR}) for $r=2$).
\end{proof}

\subsection{Admissible $\phi_r$-developments and Theorem of the product in order $r$}
Let
\begin{equation}\label{phdev2}
f(x)=\sum_{i\ge 0}a'_i(x)\phi_r(x)^i,\quad a'_i(x)\in\zpx,
\end{equation}
be a $\phi_r$-development of $f(x)$, not necessarily the $\phi_r$-adic
one. Let $N'$ be the
principal polygon of the set of points $(i,u'_i)$, with $u'_i=v_r(a'_i(x)\phi_r(x)^i)$. Let $i_1$ be the first abscissa with $a'_{i_1}(x)\ne0$. As we did in order one, to each integer abscissa $i_1\le i\le\ell(N')$ we attach a residual coefficient
$$\as{1.6}
c'_i=\left\{\begin{array}{ll}
0,&\mbox{ if $(i,u'_i)$ lies strictly above }N',\\z_{r-1}^{t'_{r-1}(i)}R_{r-1}(a'_i)(z_{r-1})\in\ff{r},&\mbox{ if $(i,u'_i)$ lies on }N'
\end{array}
\right.
$$ 
where $t'_{r-1}(i):=t_{r-1}(a'_i)_i$. 
For the points $(i,u'_i)$ lying on $N'$ we may have now $c'_i=0$; for instance in the case $a'_0(x)=f(x)$ the Newton polygon has only one point $(0,v_r(f))$ and $c'_0=0$ if $\om_r(f)>0$. 

Finally, for any negative rational number $\la_r=-h_r/e_r$, with $h_r,e_r$ positive coprime integers, we define the 
residual polynomial attached to the $\la_r$-component $S'=S_{\la_r}(N')$ to be
$$R'_{\la_r}(f)(y):=c'_{s'}+c'_{s'+e_r}\,y+\cdots+c'_{s'+(d'-1)e_r}\,y^{d'-1}+c'_{s'+d'e_r}\,y^{d'}\in\ff{r}[y],$$
where $d'=d(S')$ and $s'$ is the initial abscissa of $S'$.

\begin{definition}
We say that the $\phi_r$-development (\ref{phdev2}) is admissible if
for each abscissa $i$ of a vertex of $N'$
we have $c'_i\ne0$, or equivalently, $\om_r(a'_i)=0$.
\end{definition}

\begin{lemma}\label{admissible2}
If a $\phi_r$-development is admissible then
$N'=N_r^-(f)$ and $c'_i=c_i$ for all abscissas $i$ of the finite part of $N'$. In particular, for any negative rational number $\la_r$ we have $R'_{\la_r}(f)(y)=R_{\la_r}(f)(y)$.
\end{lemma}

\begin{proof}
Consider the $\phi_r$-adic developments of $f(x)$ and each $a'_i(x)$:
$$
f(x)=\sum_{0\le i}a_i(x)\phi_r(x)^i,\qquad a'_i(x)=\sum_{0\le k}b_{i,k}(x)\phi_r(x)^k.
$$
By the uniqueness of the $\phi_r$-adic development we have
\begin{equation}\label{bi2}
a_i(x)=\sum_{0\le k\le i}b_{i-k,k}(x).
\end{equation}
Let us denote $w_{i,k}:=v_r(b_{i,k})$, $w:=v_r(\phi_r)$. By item 1 of Lemma \ref{shape},
$u'_i=v_r(a'_i)+iw=\min_{0\le k}\{w_{i,k}+(k+i)w\}$. Hence, for all $0\le k$ and all $0\le i\le \ell(N')$:
\begin{equation}\label{uprima}
w_{i,k}+(k+i)w\ge u'_i\ge y_i(N').
\end{equation}
Therefore, by (\ref{bi2}) and (\ref{uprima}), all points $(i,u_i)$ lie above $N'$; in fact
\begin{multline}\label{uuprima}
u_i=v_r(a_i)+iw\ge \min_{0\le k\le i}\{w_{i-k,k}+iw\}=w_{i-k_0,k_0}+iw\\
\ge u'_{i-k_0}\ge y_{i-k_0}(N')\ge  y_i(N'),
\end{multline}
for some $0\le k_0\le i$. On the other hand, for any abscissa $i$ of the finite part of $N'$ and for any $0<k\le i$ we have by (\ref{uprima})
\begin{equation}\label{zeroterm}
w_{i-k,k}\ge u'_{i-k}-iw \ge y_{i-k}(N')-iw>y_i(N')-iw. 
\end{equation}
The following claim ends the proof of the lemma:\e 

\noindent{\bf Claim. }Let $i$ be an abscissa of the finite part of $N'$ such that $(i,u'_i)\in N'$. Then, $u_i=u'_i$ if and only if $c'_i\ne0$; and in this case $c'_i=c_i$.

In fact, suppose $c'_i\ne0$, or equivalently, $\om_r(a'_i)=0$. We decompose
$$
a'_i(x)=b_{i,0}(x)+B(x),\quad B(x)=\sum_{0< k}b_{i,k}(x)\phi_r(x)^k.
$$
Note that $\om_r(B)>0$, because $\phi_r(x)|B(x)$. By (\ref{uprima}), $v_r(b_{i,0})=w_{i,0}\ge u'_i-iw=v_r(a'_i)$. Since $\om_r(a'_i)=0$ and $\om_r(B)>0$, item 1 of Proposition \ref{pseudo} shows that $v_r(b_{i,0})=\min\{v_r(a'_i),v_r(B)\}$; hence, $v_r(b_{i,0})=v_r(a'_i)$. By (\ref{bi2}) and (\ref{zeroterm}) we have $u_i-iw=v_r(a_i)=w_{i,0}=u'_i-iw$, so that $u_i=u'_i$.
Let $T(i)$ be the side attached to the point $(i,u'_i)\in N'$ in Lemma \ref{enlarge}. Since $R_{r-1}(B)(z_{r-1})=0$, (\ref{RvssubSr})
shows that $R_{r-1}(B,T(i))(z_{r-1})=0$. By Lemma \ref{sumRr}, $R_{r-1}(a'_i,T(i))(z_{r-1})=R_{r-1}(b_{i,0},T(i))(z_{r-1})$, and Lemma \ref{enlarge} shows that
\begin{align*}
c'_i=&\,(z_{r-1})^{(\mathfrak{s}_i-\ell_{r-1} u_i)/e_{r-1}}R_{r-1}(a'_i,T(i))(z_{r-1})\\
=&\,(z_{r-1})^{(\mathfrak{s}_i-\ell_{r-1} u_i)/e_{r-1}}R_{r-1}(b_{i,0},T(i))(z_{r-1})
\\=&\,(z_{r-1})^{(s_{r-1}(b_{i,0})-\ell_{r-1} u_i)/e_{r-1}}R_{r-1}(b_{i,0})(z_{r-1})\\
=&\,(z_{r-1})^{(s_{r-1}(a_i)-\ell_{r-1} u_i)/e_{r-1}}R_{r-1}(a_i)(z_{r-1})=c_i,
\end{align*}
the last but one equality because $S_{r-1}(a_i)=S_{r-1}(b_{i,0})$, $R_{r-1}(a_i)=R_{r-1}(b_{i,0})$, by (\ref{zeroterm}) and Proposition \ref{pseudo}.

Conversely, if $u_i=u'_i=y_i(N')$ we have necessarily $k_0=0$ in (\ref{uuprima}) and all inequalities of (\ref{uuprima}) are equalities. Hence, $w_{i,0}+iw=u'_i$, or equivalently, $v_r(a'_i)=v_r(b_{i,0})$. Since $\om_r(b_{i,0})=0$ and $\om_r(B)>0$, Proposition \ref{pseudo} shows that $\om_r(a'_i)=0$. This ends the proof of the claim.
\end{proof}

\begin{theorem}[Theorem of the product in order $r$]For any nonzero $f(x),g(x)\in\zpx$ and any negative rational number $\la_r$ we have
 $$N_r^-(fg)=N_r^-(f)+N_r^-(g),\qquad R_{\la_r}(fg)(y)=R_{\la_r}(f)(y)R_{\la_r}(g)(y).$$
\end{theorem}

\begin{proof}
Consider the respective $\phi_r$-adic developments
$$
f(x)=\sum_{0\le i}a_i(x)\phi_r(x)^i,\quad
g(x)=\sum_{0\le j}b_j(x)\phi_r(x)^j,
$$and denote $u_i=v_r(a_i\phi_r^i)$, $v_j=v_r(b_j\phi_r^j)$, $N_f=N_r^-(f)$, $N_g=N_r^-(g)$. Take
\begin{equation}\label{product2}
f(x)g(x)=\sum_{0\le k}A_k(x)\phi_r(x)^k,\qquad
A_k(x)=\sum_{i+j=k}a_i(x)b_j(x),
\end{equation}
and denote by $N'$ the principal part of the Newton polygon of order $r$ of $fg$, determined by this $\phi_r$-development.

We shall show that $N'=N_f+N_g$, that this $\phi_r$-development is admissible, and that $R'_{\la_r}(fg)=R_{\la_r}(f)R_{\la_r}(g)$ for all negative $\la_r$. The theorem will be then a
consequence of Lemma \ref{admissible2}.

Let $w_k:=v_r(A_k\phi_r^ k)$ for all $0\le k$. Lemma
\ref{sum} shows that the point $(i,u_i)+(j,v_j)$ lies above $N_f+N_g$
for any $i,j\ge 0$. Since $w_k\ge
\min\{u_i+v_j,i+j=k\}$, the points $(k,w_k)$ lie all above $N_f+N_g$ too.
On the other hand, let $P_k=(k,y_k(N_f+N_g))$ be a vertex of $N_f+N_g$; that is, $P_k$ is the end point
of $S_1+\cdots+S_r+T_1+\cdots+T_s$, for certain sides $S_i$ of $N_f$ and $T_j$ of $N_g$, ordered by increasing slopes among all sides of $N_f$ and $N_g$. By Lemma \ref{sum}, for all
pairs $(i,j)$ such that $i+j=k$, the point $(i,u_i)+(j,v_j)$ lies
strictly above $N_f+N_g$ except for the pair
$i_0=\ell(S_{r-1}+\cdots+S_r)$, $j_0=\ell(T_{r-1}+\cdots+T_s)$ that
satisfies $(i_0,u_{i_0})+(j_0,v_{j_0})=P_k$. Thus,
$(k,w_k)=P_k$. This shows that $N'=N_f+N_g$.

Moreover, for all $(i,j)\ne(i_0,j_0)$ we have 
$$v_r(A_k\phi_r^k)=v_r(a_{i_0}b_{j_0}\phi_r^k)<v_r(a_ib_j\phi_r^k),
$$
so that $v_r(A_k)=v_r(a_{i_0}b_{j_0})<v_r(a_ib_j)$. By Proposition \ref{pseudo}, $\om_r(A_k)=\om_r(a_{i_0}b_{j_0})=\om_r(a_{i_0})+\om_r(b_{j_0})=0$, and the $\phi_r$-development (\ref{product2}) is admissible.

Finally, by (\ref{sumsla}), the $\la_r$-components $S'=S_{\la_r}(N')$, $S_f=S_{\la_r}(N_f)$, $S_g=S_{\la_r}(N_g)$ are related by: $S'=S_f+S_g$. Let $(k,y_k(N'))$ be a point of integer coordinates lying on $S'$ (not necessarily a vertex), and let $T(k)$ be the corresponding side of slope $\la_{r-1}$ given in Lemma \ref{enlarge}, with starting abscissa $\mathfrak{s}_k$. Denote by $I$ the set of the pairs $(i,j)$ such that $(i,u_i)$ lies on $S_f$, $(j,v_j)$ lies on $S_g$, and $i+j=k$. Take $P(x)=\sum_{(i,j)\in I}a_i(x)b_j(x)$. By Lemma \ref{sum}, for all other pairs $(i,j)$ with $i+j=k$, the point $(i,u_i)+(j,v_j)$ lies strictly above $N'$. 
By Lemma \ref{sumRr},  
$$R_{r-1}(A_k,T(k))=R_{r-1}(P,T(k))=\sum_{(i,j)\in I}R_{r-1}(a_ib_j,T(k)).$$ Lemma \ref{enlarge}, (\ref{RvssubSr}) and the Theorem of the product in order $r-1$ show that
\begin{align*}
c'_k(fg)=&\,(z_{r-1})^{\frac{\mathfrak{s}_k-\ell_{r-1} w_k}{e_{r-1}}}R_{r-1}(A_k,T(k))(z_{r-1})\\=&\,(z_{r-1})^{\frac{\mathfrak{s}_k-\ell_{r-1} w_k}{e_{r-1}}}\sum_{(i,j)\in I}R_{r-1}(a_ib_j,T(k))(z_{r-1})\\
=&\sum_{(i,j)\in I}(z_{r-1})^{\frac{s_{r-1}(a_ib_j)-\ell_{r-1} w_k}{e_{r-1}}}R_{r-1}(a_ib_j)(z_{r-1})\\=&\sum_{(i,j)\in I}(z_{r-1})^{t_{r-1}(i,f)+t_{r-1}(j,g)}R_{r-1}(a_i)(z_{r-1})R_{r-1}(b_j)(z_{r-1})\\
=&\sum_{(i,j)\in I}c_i(f)c_j(g).
\end{align*}
This shows that the residual polynomial attached to $S'$ with respect to the $\phi_r$-development (\ref{product2}) is $R_{\la_r}(f)R_{\la_r}(g)$.
\end{proof}

\begin{corollary}\label{vshift}
Let $f(x)\in \zpx$ be a monic polynomial with $\om_r(f)>0$, and let $f_{\ty}(x)$ be the monic factor of $f(x)$ determined by $\ty$ (cf. Definition \ref{ppt}). Then $N_r(f_{\ty})$ is equal to $N_r^-(f)$ up to a vertical shift, and $R_{\la_r}(f)\sim R_{\la_r}(f_\ty)$ for any negative rational number $\la_r$. 
\end{corollary}

\begin{proof}
Let $f(x)=f_{\ty}(x)g(x)$. By (\ref{sameomega}), $\om_r(g)=0$. By the Theorem of the product, $N_r^-(f)=N_r^-(f_{\ty})+N_r^-(g)$ and $R_{\la_r}(f)=R_{\la_r}(f_\ty)R_{\la_r}(g)$. Since  
$N_r^-(g)$ reduces to a point with abscissa $0$ (cf. Lemma \ref{shape}), the polygon $N_r^-(f)$ is a vertical shift of $N_r^-(f_{\ty})$ and $R_{\la_r}(g)$ is a cons\-tant.
\end{proof}

\section{Dissections in order $r$}\label{secOre}
In this section we extend to order $r$ the Theorems of the polygon and of the residual polynomial. We fix throughout a type $\ty$ of order $r-1$ and a representative $\phi_r(x)$ of $\ty$. We proceed by induction and we asume that all results of this section have been proved already in orders $1,\dots,r-1$. The case $r=1$ was considered in section \ref{secNP}. 
 
\subsection{Theorem of the polygon in order $r$}
Let $f(x)\in\zpx$ be a monic polynomial such that $\om_r(f)>0$. The aim of this section is to obtain a factori\-zation of $f_{\ty}(x)$ and certain arithmetic data of the factors. Thanks to Corollary \ref{vshift}, we shall be able to read this information directly on $N_r^-(f)$, and the different residual polynomials $R_{\la_r}(f)(y)$.

\begin{theorem}[Theorem of the polygon in order $r$]\label{thpolygonr}
Let $f(x)\in\zpx$ be a monic polynomial such that $\om_r(f)>0$. Suppose that 
$N_r^-(f)=S_1+\cdots+S_g$ has $g$ sides with pairwise different slopes
$\la_{r,1},\dots,\la_{r,g}$. Then, $f_{\ty}(x)$ admits a factorization
$$f_{\ty}(x)=F_1(x)\cdots F_g(x),$$ as a product of $g$ monic polynomials of $\zpx$ satisfying
the following properties:
\begin{enumerate}
\item $N_r(F_i)$ is equal to $S_i$ up to a translation,
\item If $S_i$ has finite slope, then $R_{\la_{r,i}}(F_i)(y)\sim R_{\la_{r,i}}(f)(y)$
\item   For any root $\t\in\qpb$ of $F_i(x)$, $v(\phi_r(\t))=(v_r(\phi_r)+|\la_{r,i}|)/e_1\cdots e_{r-1}$. 
\end{enumerate}
\end{theorem}

\begin{proof}Let us denote $e=e_1\cdots e_{r-1}$.
We deal first with the case $f_\ty(x)$ irreducible. Note that $\dg f_{\ty}=m_r\om_r(f)>0$, by Lemma \ref{typedegree}, and $N_r(f_\ty)=N_r^-(f_\ty)$ by Corollary \ref{rminus}. Since $f_\ty(x)$ is irreducible, $\rho:=v(\phi_r(\t))$ is constant among all roots $\t\in\qpb$ of $f_\ty(x)$, and $0\le v_r(\phi_r)/e<\rho$, by Proposition \ref{vpt}. We have $\rho=\infty$ if and only if $f_\ty(x)=\phi_r(x)$, and in this case the theorem is clear. Suppose $\rho$ is finite.

Let $P(x)=\sum_{0\le i\le k}b_ix^ i\in\zpx$ be the minimal polynomial of $\phi_r(\t)$, and let $Q(x)=P(\phi_r(x))=\sum_{0\le i\le k}b_i\phi_r(x)^i$. By the Theorem of the polygon in order one, the $x$-polygon of $P$ has only one side and it has slope $-\rho$. 
The end points of $N_r(Q)$ are $(0,ek\rho)$ and $(k,kv_r(\phi_r))$. Now, for all $0\le i\le k$,
$$
\dfrac{v_r(b_i\phi_r^i)-kv_r(\phi_r)}{k-i}=\dfrac{ev(b_i)+iv_r(\phi_r)-kv_r(\phi_r)}{k-i}\ge  
e\rho-v_r(\phi_r).$$
This implies that $N_r(Q)$ has only one side and it has slope $\la_r:=-(e\rho-v_r(\phi_r))$. Since $Q(\t)=0$, $f_\ty(x)$ divides $Q(x)$ and the Theorem of the product shows that $N_r(f_\ty)$ is one-sided, with the same slope. Also, $R_{\la_r}(f_\ty)\sim R_{\la_r}(f)$ by Corollary \ref{vshift}. This ends the proof of the theorem when $f_\ty(x)$ is irreducible.

If $f_\ty(x)$ is not necessarily irreducible, we consider its decomposition $f_{\ty}(x)=\prod_jP_j(x)$ into a product of monic irreducible factors in $\zpx$. By Lemma \ref{factortype}, each $P_j(x)$ has type $\ty$ and by  the proof in the irreducible case, each $P_j(x)$ has a one-sided $N_r(P_j)$. The Theorem of the product shows that the slope of $N_r(P_j)$ is $\la_{r,i}$ for some $1\le i\le s$. If we group these factors according to the slope, we get the desired factorization. By the Theorem of the product, $R_{\la_{r,i}}(F_i)\sim R_{\la_{r,i}}(f_\ty)$, because $R_{\la_{r,i}}(F_j)$ is a constant for all $j\ne i$. Finally, $R_{\la_{r,i}}(f_\ty)\sim R_{\la_{r,i}}(f)$ by Corollary \ref{vshift}.
The statement about $v(\phi_r(\t))$ is obvious because $P_j(\t)=0$ for some $j$, and we have already proved the formula for an irreducible polynomial.
\end{proof}

We recall that the factor corresponding to a side $S_i$ of slope $-\infty$ is necessarily $F_i(x)=\phi_r(x)^{\ord_{\phi_r}(f)}$ (cf. Remark \ref{phiell}). 

Let $\la_r=-h_r/e_r$, with $h_r,e_r$ positive coprime integers, be a negative rational number such that $S:=S_{\la_r}(f)$ has positive length. Let $f_{\ty,\la_r}(x)$ be the factor of $f(x)$, corresponding to the pair $\tilde{\ty},\la_r$ by the Theorem of the polygon. Choose a root $\t\in\qpb$ of $f_{\ty,\la_r}(x)$, and let $L=K(\t)$.
By item 4 of Propositions \ref{vqt} and \ref{vqtr}, in orders $1,\dots,r-1$, there is a well-defined embedding
 $\ff{r}\lra \ff{L}$, determined by 
\begin{equation}\label{embeddingr}
\ff{r}\hookrightarrow \ff{L}, \ z_0\mapsto \tb, \ z_1\mapsto \gb1,\ \dots,\ z_{r-1}\mapsto \gb{r-1}.
\end{equation}
This embedding depends on the choice of $\t$. After this identification of $\ff{r}$ with a subfield of $\ff{L}$ we can think that all residual polynomials of $r$-th order have coefficients in $\ff{L}$. 

\begin{corollary}\label{zerovalue}For the rational functions of Definition \ref{ratfracs}:
\begin{enumerate}
 \item $v(\phi_r(\t))=\sum_{i=1}^re_if_i\cdots e_{r-1}f_{r-1}h_i/(e_1\cdots e_i)$,
\item $v(\pi_r(\t))=1/(e_1\cdots e_{r-1})$,
\item $v(\Phi_r(\t))=h_r/(e_1 \cdots e_r)$,
\item $v(\ga_r(\t))=0$.
\end{enumerate}
\end{corollary}

\begin{proof}
Item 1 is a consequence of the Theorem of the polygon and the formula for $v_r(\phi_r)$ in Proposition \ref{vrphii}. 
Item 2 follows from Proposition \ref{vpt}, because $v_r(\pi_r)=1$, $\om_r(\pi_r)=0$ by Proposition \ref{vrphii}.
Item 3 follows from the Theorem of the polygon and item 2 in order $r-1$. Item 4 follows from items 2,3.
\end{proof}

\begin{corollary}\label{eeee}
With the above notation for $L$, the residual degree $f(L/K)$ is divisible by $f_0\cdots f_{r-1}$, and the ra\-mification index $e(L/K)$ is divisible by $e_1\cdots e_r$. Moreover, the number of irreducible factors of $f_{\ty,\la_r}(x)$ is at most $d(S)$; in particular, if $d(S)=1$ the polynomial $f_{\ty,\la_r}(x)$ is irreducible in $\zpx$, and $f(L/K)=f_0\cdots f_{r-1}$, $e(L/K)=e_1\cdots e_r$.
\end{corollary}

\begin{proof}
The statement on the residual degree is a consequence of the embedding (\ref{embeddingr}). Denote $e_L=e(L/K)$, $e=e_1\cdots e_{r-1}$, $f=f_0\cdots f_{r-1}$. By the same result in order $r-1$ (cf. Corollary \ref{ram} for $r=2$), $e_L$ is divisible by $e$. Now, by the theorem of the polygon,
$v_L(\phi_r(\t))=(e_L/e)v_r(\phi_r)+(e_L/e)(h_r/e_r)$.  Since this is an integer and $h_r,e_r$ are coprime, necessarily $e_r$ divides $e_L/e$. 

The upper bound for the number of irreducible factors is a consequence of the Theorem of the product. Finally, if $d(S)=1$, we have $efe_r=\dg(f_{\ty,\la_r})=f(L/K)e(L/K)$, and necessarily $f(L/K)=f$ and $e(L/K)=ee_r$. 
\end{proof}

We prove now an identity that plays an essential role in what follows.

\begin{lemma}\label{identity}
Let $P(x)=\sum_{0\le i}a_i(x)\phi_r(x)^i$ be the $\phi_r$-adic development of a nonzero polynomial in $\zpx$. Let $\la_r=-h_r/e_r$ be a negative rational number, with $h_r,\,e_r$ coprime positive integers. Let $S=S_{\la_r}(P)$ be the $\la_r$-component of $N_r^-(P)$, let $(s,u)$ be the initial point of $S$ and $(i,u_i)$ any point lying on $S$. Let $(s(a_i),u(a_i))$ be the initial point of the side $S_{r-1}(a_i)$. Then, the following identity holds in $K(x)$:
\begin{equation}\label{id}
\phi_r(x)^i\dfrac{\Phi_{r-1}(x)^{s(a_i)}\pi_{r-1}(x)^{u(a_i)}}{\Phi_r(x)^s\pi_r(x)^u}=\ga_{r-1}(x)^{t_{r-1}(i)}\ga_r(x)^{\frac{i-s}{e_r}}.
\end{equation}
\end{lemma}

\begin{proof}
If we substitute $u=u_i+(i-s)\frac{h_r}{e_r}$ and $\ga_r=\Phi_r^{e_r}/\pi_r^{h_r}$ in (\ref{id}), we see that the identity is equivalent to:
$$
\phi_r(x)^i\dfrac{\Phi_{r-1}(x)^{s(a_i)}\pi_{r-1}(x)^{u(a_i)}}{\pi_r(x)^{u_i}}=\ga_{r-1}(x)^{t_{r-1}(i)}\Phi_r(x)^i.
$$
If we substitute now $\Phi_r$, $\pi_r$ and $\ga_{r-1}$ by its defining values and we use $e_{r-1}t_{r-1}(i)=s(a_i)-\ell_{r-1}u_i$, we get an equation involving only $\pi_{r-1}$, which is equivalent to: 
$$
u(a_i)+\ell'_{r-1}u_i+h_{r-1}t_{r-1}(i)+if_{r-1}v_r(\phi_{r-1})=0.
$$ 
This equality is easy to check by using $e_{r-1}u(a_i)+s(a_i)h_{r-1}=v_r(a_i)=u_i-iv_r(\phi_r)$, $v_r(\phi_r)=e_{r-1}f_{r-1}v_r(\phi_{r-1})$, and $\ell_{r-1}h_{r-1}-\ell'_{r-1}e_{r-1}=1$.
\end{proof}

\begin{proposition}[Computation of $v(P(\t))$ with the polygon]\label{vqtr}
We keep the above notations for $f(x),\,\la_r=-h_r/e_r,\,\t,\,L,$ and the embedding (\ref{embeddingr}).  Let $P(x)\in\zpx$ be a nonzero polynomial, $S=S_{\la_r}(P)$, $L_{\la_r}$ the line of slope $\la_r$ that contains $S$, and $H$ the ordinate at the origin of this line. Denote $e=e_1\cdots e_{r-1}$. Then,  
\begin{enumerate}
 \item $v(P^S(\t))\ge0,\qquad \overline{P^S(\t)}=R_{\la_r}(P)(\gb{r})$,
\item $v(P(\t)-P^0(\t))>H/e$.
\item $v(P(\t))\ge H/e$, and equality holds if and only if $R_{\la_r}(P)(\gb{r})\ne0$,
\item $R_{\la_r}(f)(\gb{r})=0$.
\item  If $R_{\la_r}(f)(y)\sim \psi_r(y)^a$ for some irreducible $\psi_r(y)\in\ff{r}[y]$ then $v(P(\t))=H/e$  if and only if $R_{\la_r}(P)(y)$ is not divisible by $\psi_r(y)$ in $\ff{r}[y]$.
\end{enumerate}
\end{proposition}

\begin{proof}
Let $P(x)=\sum_{0\le i}a_i(x)\phi_r(x)^i$ be the $\phi_r$-adic development of $P(x)$, and denote $u_i=v_r(a_i\phi_r^i)$, $N=N_r^-(P)$.
Recall that 
$$P^S(x)=\Phi_r(x)^{-s}\pi_r(x)^{-u}P^0(x),\quad P^0(x)=\sum_{(i,u_i)\in S}a_i(x)\phi_r(x)^i,$$ where $(s,u)$ are the coordinates of the initial point of $S$. By Corollary \ref{zerovalue}, 
\begin{equation}\label{He}
 v(\Phi_r(\t)^s\pi_r(\t)^u)=\frac 1e\left(s\frac{h_r}{e_r}+u\right)=\frac{H}e.
\end{equation}
On the other hand, by the Theorem of the polygon and Proposition \ref{vpt}, for all $i$:
\begin{equation}\label{eqHe}
v(a_i(\t)\phi_r(\t)^i)=\frac{v_r(a_i)}e+\frac ie\left(v_r(\phi_r)+\frac{h_r}{e_r}\right)=\frac1e\left(u_i+i\frac{h_r}{e_r}\right)\ge\frac He,
\end{equation}
with equality if and only if $(i,u_i)\in S$. This proves item 2. 

\begin{center}
\setlength{\unitlength}{5.mm}
\begin{picture}(10,6)
\put(0,0){\line(1,0){9}}
\put(9,1.5){\line(-2,1){8}}
\put(0.9,5.5){\line(1,0){.2}}
\put(0.3,5.3){\begin{footnotesize}$H$\end{footnotesize}}
\put(1,-1){\line(0,1){7.5}}
\put(4.85,3.35){$\bullet$}\put(2.85,4.35){$\bullet$}\put(6.85,2.35){$\bullet$}
\put(7,2.5){\line(-2,1){4}}\put(7.02,2.5){\line(-2,1){4}}
\put(7,2.5){\line(4,-1){1.4}}\put(7.02,2.5){\line(4,-1){1.4}}
\put(3,4.5){\line(-1,1){1}}\put(3.02,4.5){\line(-1,1){1}}
\put(4,-1.8){\begin{footnotesize}$N_r^-(P)$\end{footnotesize}}
\multiput(3,-.1)(0,.20){24}{\vrule height2pt}
\multiput(5,-.1)(0,.20){19}{\vrule height2pt}
\multiput(0.9,4.5)(.25,0){8}{\hbox to 2pt{\hrulefill }}
\put(5.1,3.8){\begin{footnotesize}$(i,u_i)$\end{footnotesize}}
\put(2.9,-.65){\begin{footnotesize}$s$\end{footnotesize}}
\put(4.9,-.65){\begin{footnotesize}$i$\end{footnotesize}}
\put(0.4,4.35){\begin{footnotesize}$u$\end{footnotesize}}
\put(8.4,.9){\begin{footnotesize}$L_{\la_r}$\end{footnotesize}}
\put(0.6,-.6){\begin{footnotesize}$0$\end{footnotesize}}
\end{picture}
\end{center}\be\be\be

Also, (\ref{He}) and (\ref{eqHe}) show that $v(P^S(\t))\ge 0$, so that $P^S(\t)$ belongs to $\ol$. Denote for simplicity $z_r=\gb{r}$. In order to prove the equality $\overline{P^S(\t)}=R_{\la_r}(P)(z_r)$, we need to show that for every $(i,u_i)\in S$: 
\begin{equation}\label{rdm}
\rdm\left(\dfrac{a_i(\t)\phi_r(\t)^i}{\Phi_r(\t)^s\pi_r(\t)^u}\right)=
(z_{r-1})^{t_{r-1}(i)}R_{r-1}(a_i)(z_{r-1})(z_r)^{(i-s)/e_r}.
\end{equation}
Let $(s(a_i),u(a_i))$ be the initial point of $S_{r-1}(a_i)$. By items 1,2 of the proposition in order $r-1$ (Proposition \ref{vqt} if $r=2$), applied to the polynomial $a_i(x)$,
$$\as{1.4}
\begin{array}{l}
\overline{(a_i)^{S_{r-1}(a_i)}(\t)}=R_{r-1}(a_i)(z_{r-1}),\\
a_i(\t)\equiv \Phi_{r-1}(\t)^{s(a_i)}\pi_{r-1}(\t)^{u(a_i)}(a_i)^{S_{r-1}(a_i)}(\t) \md{\m_L^{(v_r(a_i)e(L/K)/e)+1}}.
\end{array}
$$
Since $v_r(a_i)e(L/K)/e=v_L(a_i(\t))$, it suffices to check the following identity in $L$:
$$
\phi_r(\t)^i\dfrac{\Phi_{r-1}(\t)^{s(a_i)}\pi_{r-1}(\t)^{u(a_i)}}{\Phi_r(\t)^s\pi_r(\t)^u}=\ga_{r-1}(\t)^{t_{r-1}(i)}\ga_r(\t)^{\frac{i-s}{e_r}},
$$which is a consequence of Lemma \ref{identity}. This ends the proof of item 1.

Also, (\ref{eqHe}) shows that $v(P(\t))\ge H/e$, and
$$
v(P(\t))=H/e\sii v(P^0(\t))=H/e\stackrel{(\ref{He})}\sii v(P^S(\t))=0\sii
R_{\la_r}(P)(z_r)\ne0,$$
the last equivalence by item 1. This proves item 3.
The last two items are proved by similar arguments to that of the proof of Proposition \ref{vqt}.
\end{proof}

\subsection{Theorem of the residual polynomial in order $r$}\label{tresidualpolr}
We discuss now how Newton polygons and residual polynomials are affected by an extension of the base field by an unramified extension. We keep the above notations for $f(x),\,\la_r=-h_r/e_r,\,\t,\,L$ and the embedding (\ref{embeddingr}).

\begin{proposition}\label{extensionr}
Let $K'$ be the unramified extension of $K$ of degree $f_0\cdots f_{r-1}$. Let us identify $\ff{r}=\ff{K'}$ through the embedding (\ref{embeddingr}). Let $G(x)\in \oo_{K'}[x]$ be the minimal polynomial of $\t$ over $K'$. Then, there exist a type of order $r-1$ over $K'$, $\ty'=(\phi'_1(x);\la_1,\phi'_2(x);\cdots;\la_{r-1},\psi'_{r-1}(y))$, and a representative $\phi'_r(x)$ of $\ty'$, with the following properties (where the superscript $'$ indicates that the objects are taken with respect to $\ty'$): 
\begin{enumerate}
 \item $f'_0=\cdots =f'_{r-1}=1$, 
\item $G(x)$ is of tytpe $\ty'$,
\item For any nonzero polynomial $P(x)\in\zpx$, 
$$(N')_r^-(P)=N_r^-(P), \quad R'_{\la_r}(P)(y)=\sigma_r^s\tau_r^u R_{\la_r}(P)(\mu_ry),$$
where $(s,u)$ is the initial point of $S_{\la_r}(P)$ and $\sigma_r,\,\tau_r,\,\mu_r\in\ff{K'}^*$ are constants that depend only on $\ty$ and $\t$.
\end{enumerate}
\end{proposition}

\begin{proof}
We proceed by induction on $r$. The case $r=1$ is considered in Lemma \ref{extension}; for the constant $\epsilon$ defined there, we can take $\sigma_1=\epsilon$, $\tau_1=1$, and $\mu_1=\epsilon^{e_1}$. Let $r\ge 2$ and suppose we have already constructed $\ty'_{r-2}$ and a representative $\phi'_{r-1}(x)$ satisfying these properties.
Let $\eta_1,\dots,\eta_{f_{r-1}}\in \ff{K'}$ be the roots of $\psi_{r-1}(y)$, and denote $F(x)=f_{\ty,\la_r}(x)$. We have,
$$\as{1.4}
\begin{array}{l}
R'_{r-1}(\phi_r)(y)\sim R_{r-1}(\phi_r)(\mu_{r-1}y)\sim\psi_{r-1}(\mu_{r-1}y)=\prod_{i=1}^{f_{r-1}}(\mu_{r-1}y-\eta_i),\\
R'_{r-1}(F)(y)\sim R_{r-1}(F)(\mu_{r-1}y)\sim\psi_{r-1}(\mu_{r-1}y)^{a_{r-1}}=\prod_{i=1}^{f_{r-1}}(\mu_{r-1}y-\eta_i)^{a_{r-1}}.
\end{array}
$$
Since $G(x)$ is of type $\ty'_{r-2}$, Lemma \ref{typedegree} shows that $\dg G=m'_{r-1}\om'_{r-1}(G)$.
Since $(N')_{r-1}^-(F)=N_{r-1}^-(F)$, the Theorem of the product shows that $(N')_{r-1}^-(G)$ is one-sided, with slope $\la_{r-1}$ and positive length $\om'_{r-1}(G)$. By the Theorem of the residual polynomial, $R'_{r-1}(G)(y)\sim (\mu_{r-1}y-\eta)^a$, for some root $\eta\in\ff{K'}$ of $\psi_{r-1}(y)$ and some positive integer $a$.  We take
$\psi'_{r-1}(y)=y-\mu_{r-1}^{-1}\eta$, and
$$\ty'=(\phi'_1(x);\la_1,\phi'_2(x);\cdots;\la_{r-2},\phi'_{r-1}(y);\la_{r-1},\psi'_{r-1}(y)).$$ 
Thus, $f'_{r-1}=1$. We have $a=\om'_r(G)$ and $\dg G=m'_{r-1}\om'_{r-1}(G)=m'_{r-1}e_{r-1}a=m'_ra$; thus, $G(x)$ is of type $\ty'$, by Lemma \ref{factortype}. 

The same argument shows that there is a unique irreducible factor $\phi'_r(x)$ of $\phi_r(x)$ in $\oo_{K'}[x]$ such that $R'_{r-1}(\phi'_r)(y)\sim (\mu_{r-1}y-\eta)$. We choose $\phi'_r(x)$ as a representative of $\ty'$. Let $\rho_r(x)=\phi_r(x)/\phi'_r(x)\in\oo_{K'}[x]$. By construction, $\om'_r(\rho_r)=0$, because $R'_{r-1}(\rho_r)(y)\sim \psi_{r-1}(\mu_{r-1}y)/(\mu_{r-1}y-\eta)$.

Let $P(x)\in\zpx$ be a nonzero polynomial. Clearly, 
$$(N')_{r-1}^-(P)=N_{r-1}^-(P)\imp v'_r(P)=v_r(P),$$$$R'_{r-1}(P)(y)\sim R_{r-1}(P)(\mu_{r-1}y)\imp \om'_r(P)=\om_r(P).$$ 
 Consider the $\phi_r$-adic development of $P(x)$:
\begin{multline*}
P(x)=\phi_r(x)^n+a_{n-1}(x)\phi_r(x)^{n-1}+\cdots+a_0(x)=\\=\rho_r(x)^{n}\phi'_r(x)^n+a_{n-1}(x)\rho_r(x)^{n-1}\phi'_r(x)^{n-1}+\cdots+a_0(x).
\end{multline*}
Since $\om'_r(\rho_r)=0$, this $\phi'_r$-adic development of $P(x)$ is admissible. On the other hand, the tautology
$$
v_r(a_i(x)\phi_r(x)^i)=v'_r(a_i(x)\phi_r(x)^i)=v'_r(a_i(x)\rho_r(x)^i\phi'_r(x)^i),
$$shows that $(N')_r^-(P)=N_r^-(P)$. 

In order to prove the relationship between $R'_{\la_r}(P)(y)$ and $R_{\la_r}(P)(y)$, we introduce some elements in $\ff{K'}^*$, constructed in terms of the rational functions of Defi\-nition \ref{ratfracs}.
By Corollary \ref{zerovalue}, $v(\Phi_r(\t))=v(\Phi'_r(\t))$, $v(\pi_r(\t))=(e_1\cdots e_{r-1})^{-1}=v(\pi'_r(\t))$, and $v(\ga_r(\t))=0=v(\ga'_r(\t))$. Also, by the theorem of the polygon,
$$v(\rho_r(\t))=(v_r(\phi_r)-v'_r(\phi'_r))/(e_1\cdots e_{r-1})=v(\pi'_{r-1}(\t))(v_r(\phi_r)-v'_r(\phi'_r))/e_{r-1}.
$$
We introduce the following elements of $\ff{K'}^*$: 
$$\as{1.8}
\begin{array}{ll}
\mu_r:=\overline{\ga_r(\t)/\ga'_r(\t)},\qquad&
\tau_r:=\overline{\pi_r(\t)/\pi'_r(\t)},\\ 
\sigma_r:=\overline{\Phi_r(\t)/\Phi'_r(\t)},\qquad&
\epsilon_r:=\overline{\rho_r(\t)/\pi'_{r-1}(\t)^{(v_r(\phi_r)-v'_r(\phi'_r))/e_{r-1}}}.
\end{array}
$$Since $f_{r-1}v_r(\phi_{r-1})=v_r(\phi_r)/e_{r-1}$, the recursive definition of the functions of De\-finition \ref{ratfracs}, 
yields the following identities:
\begin{equation}\label{taur}
\sigma_r=\epsilon_r/(\tau_{r-1})^{v_r(\phi_r)/e_{r-1}},\qquad
\tau_r=(\sigma_{r-1})^{\ell_{r-1}}/(\tau_{r-1})^{\ell'_{r-1}}.
\end{equation}

We need still another interpretation of $\epsilon_r$. Since $(N')_{r-1}^-(\phi_r)=N_{r-1}^-(\phi_r)$, the Theorem of the product shows that $(N')_{r-1}^-(\rho_r)$ is one-sided with slope $\la_{r-1}$; hence,  the initial point $(s'_{r-1}(\rho_r),u'_{r-1}(\rho_r))$ of $S:=S'_{r-1}(\rho_r)$ is given by  $s'_{r-1}(\rho_r)=0$ and 
\begin{equation}\label{inirhor}
u'_{r-1}(\rho_r)=v'_r(\rho_r)/e_{r-1}=(v'_r(\phi_r)-v'_r(\phi'_r))/e_{r-1}=(v_r(\phi_r)-v'_r(\phi'_r))/e_{r-1}.
\end{equation}
Recall that the virtual factor $\rho_r^S(x)$ is by definition $\rho_r(x)/\pi'_{r-1}(x)^{u'_{r-1}(\rho_r)}$; therefore, item 1 of Proposition \ref{vqtr} shows that, for $r\ge 2$: 
\begin{equation}\label{epsilonr}\epsilon_r=R'_{r-1}(\rho_r)(z'_{r-1}).
\end{equation}

We have seen above that for each integer abscissa $i$, the $i$-th terms of the $\phi_r$ and $\phi'_r$-developments of $P(x)$ determine the same point $(i,u_i)$ of the plane. Let $i=s+je_r$ be an abscissa such that $(i,u_i)$ lies on  $S_{\la_r}(P)=S'_{\la_r}(P)$; the corresponding residual coefficients at this abscissa are respectively 
$$
c_i= (z_{r-1})^{t_{r-1}(i)}R_{r-1}(a_i)(z_{r-1}),\quad
c'_i= (z'_{r-1})^{t'_{r-1}(i)}R'_{r-1}(a_i\rho_r^i)(z'_{r-1}),
$$and $R_{\la_r}(P)(y)=\sum_{0\le j\le d}c_iy^j$, $R'_{\la_r}(P)(y)=\sum_{0\le j\le d}c'_iy^j$. Hence, the last equality of item 3 is equivalent to 
$c'_i=c_i\sigma_r^s\tau_r^u \mu_r^j$, for all such $i$.

Note that $t_{r-1}(i)=(s_{r-1}(a_i)-\ell_{r-1}u_i)/e_{r-1}=t'_{r-1}(i)$, since 
$$s'_{r-1}(a_i\rho_r^i)=s'_{r-1}(a_i)+is'_{r-1}(\rho_r)=s'_{r-1}(a_i)=s_{r-1}(a_i),$$
the last equality because $N_{r-1}^-(a_i)=(N')_{r-1}^-(a_i)$. 
For simplicity we denote by $(s(a_i),u(a_i))$ the initial point of $S_{r-1}(a_i)$. By (\ref{inirhor}), the initial point of $S'_{r-1}(a_i\rho_r^i)$  is $(s(a_i),u(a_i)+i(v_r(\phi_r)-v'_r(\phi'_r))/e_{r-1})$. Now, by induction, the Theorem of the product, and (\ref{epsilonr}), we have
\begin{align*}\as{2}
c'_i=&\,(z'_{r-1})^{t_{r-1}(i)}R'_{r-1}(a_i)(z'_{r-1})\,\epsilon_r^i\\=&\,(z'_{r-1})^{t_{r-1}(i)}(\sigma_{r-1})^{s(a_i)}(\tau_{r-1})^{u(a_i)}R_{r-1}(a_i)(z_{r-1})\,\epsilon_r^i\\
=&\,c_i(\mu_{r-1})^{-t_{r-1}(i)}(\sigma_{r-1})^{s(a_i)}(\tau_{r-1})^{u(a_i)}
\epsilon_r^i\\=&\,c_i\mu_r^j\left(\mu_r^{-j}(\mu_{r-1})^{-t_{r-1}(i)}\right)(\sigma_{r-1})^{s(a_i)}(\tau_{r-1})^{u(a_i)}\epsilon_r^i.
\end{align*}
By Lemma \ref{identity}, 
\begin{align*}
\ga_r(\t)^j\ga_{r-1}(\t)^{t_{r-1}(i)}=&\,\phi_r(\t)^i\Phi_{r-1}(\t)^{s(a_i)}\pi_{r-1}(\t)^{u(a_i)}\Phi_r(\t)^{-s}\pi_r(\t)^{-u}\\
=&\,\phi_r(\t)^i\Phi_{r-1}(\t)^{s(a_i)-\ell_{r-1}u}\pi_{r-1}(\t)^{u(a_i)+\ell'_{r-1}u}\Phi_r(\t)^{-s}.
\end{align*}
We get an analogous expression for $\ga'_r(\t)^j\ga'_{r-1}(\t)^{t_{r-1}(i)}$, just by putting $'$ everywhere and by replacing $u(a_i)$ by $u(a_i\rho_r^i)=u(a_i)+i(v_r(\phi_r)-v'_r(\phi'_r))/e_{r-1}$.
By taking the quotient of both expressions and taking classes modulo $\m_{K'}$ we get
$$
\mu_r^j(\mu_{r-1})^{t_{r-1}(i)}=\epsilon_r^i(\sigma_{r-1})^{s(a_i)-\ell_{r-1}u}(\tau_{r-1})^{u(a_i)+\ell'_{r-1}u}\sigma_r^{-s}.
$$Therefore, $c'_i=c_i\mu_r^j(\sigma_{r-1})^{\ell_{r-1}u}(\tau_{r-1})^{-\ell'_{r-1}u}\sigma_r^s=c_i\mu_r^j\tau_r^u\sigma_r^s$, by (\ref{taur}).
\end{proof}

\begin{theorem}[Theorem of the residual polynomial in order $r$]\label{thresidualr}
Let $f(x)\in\zpx$ be a monic polynomial with $\om_r(f)>0$, and let $S$ be a side of $N_r^-(f)$, of finite slope  $\la_r$. Consider the factorization 
$$
R_{\la_r}(f)(y)\sim\psi_{r,1}(y)^{a_1}\cdots \psi_{r,t}(y)^{a_t},
$$ of the residual polynomial of $f(x)$ into the product of powers of pairwise different monic irreducible polynomials in $\ff{r}[y]$. Then, the factor $f_{\ty,\la_r}(x)$ of $f_{\ty}(x)$, corres\-ponding to $\tilde{\ty},\la_r$ by the Theorem of the polygon, admits a factorization
 in $\zpx$,$$
f_{\ty,\la_r}(x)=G_1(x)\cdots G_t(x),
$$into a product of $t$ monic polynomials, with all $N_r(G_i)$ one-sided of slope $\la_r$, and
 $R_{\la_r}(G_i)(y)\sim \psi_{r,i}(y)^{a_i}$ in $\ff{r}[y]$.
\end{theorem}

\begin{proof}
Let us deal first with the case $F(x):=f_{\ty,\la_r}(x)$ irreducible. We need only to prove that $R_{\la_r}(F)(y)$ is the power of an irreducible polynomial of $\ff{r}[y]$. Let $\t\in\qpb$ be a root of $F(x)$, take $L=K(\t)$, and fix the embedding $\ff{r}\to\ff{L}$ as in (\ref{embeddingr}). Let $K'$ be the unramified extension of $K$ of degree $f_0\cdots f_{r-1}$, and let $G(x)\in\oo_{K'}[x]$ be the minimal polynomial of $\t$ over $K'$, so that $F(x)=\prod_{\sigma\in\op{Gal}(K'/K)}G^{\sigma}(x)$. Under the embedding $\ff{r}\to\ff{L}$, the field $\ff{r}$ is identified to $\ff{K'}$. By Proposition \ref{extensionr}, we can construct a type $\ty'$ of order $r-1$ over $K'$ such that $R'_{\la_r}(F)(y)\sim R_{\la_r}(F)(cy)$, for some nonzero constant $c\in\ff{K'}$. By the construction of $\ty'$, for any $\sigma\ne1$, the polynomial $G^{\sigma}(x)$ is not divisible by $\phi'_1(x)$ modulo $\m_{K'}$; thus, $\om'_r(G^{\sigma})\le \om'_1(G^{\sigma})=0$, and $R'_{\la_r}(G^{\sigma})(y)$ is a constant. Therefore, by the Theorem of the product, $R'_{\la_r}(G)(y)\sim R'_{\la_r}(F)(y)\sim R_{\la_r}(F)(cy)$, so that $R_{\la_r}(F)(y)$ is the power of an irreducible polynomial of $\ff{r}[y]$ if and only if 
$R'_{\la_r}(G)(y)$ has the same property over $\ff{K'}$. In conclusion, by extending the base field, we can suppose that $f_0=\cdots =f_{r-1}=1$.

Let $P(x)=\sum_{j=0}^kb_jx^j\in\zpx$ be the minimal polynomial of $\ga_r(\t)$ over $K$. Let
$$\Pi(x):=\ga_r(x)/\phi_r(x)^{e_r}=\pi_{r-1}(x)^{-e_rf_{r-1}v_r(\phi_{r-1})}\pi_r(x)^{-h_r}.$$
By (\ref{phii}), $\Pi(x)$ admits an expression $\Pi(x)=\pi^{n'_0}\phi_1(x)^{n'_1}\cdots \phi_{r-1}(x)^{n'_{r-1}}$ for some integers $n'_1,\dots,n'_r$. Take $\Phi(x):=\pi^{n_0}\phi_1(x)^{n_1}\cdots \phi_{r-1}(x)^{n_{r-1}}$ with sufficiently large non-negative integers $n_i$ so that $\Pi(x)^k\Phi(x)$ is a polynomial in $\oo[x]$. 
Then, the following rational function is actually a polynomial in $\zpx$:
$$
Q(x):=\Phi(x) P(\ga_r(x))=\sum_{j=0}^kB_{je_r}(x)\phi_r(x)^{je_r},\quad B_{je_r}(x)=\Phi(x)\Pi(x)^jb_j.
$$ Moreover, by item 5 of Proposition \ref{vrphii}, $\om_r(B_{je_r})=0$ for all $j$ such that $B_{je_r}\ne0$, so that this $\phi_r$-development of $g(x)$ is admissible.

Our aim is to show that $N_r(Q)$ is one-sided with slope $\la_r$, and $R_{\la_r}(Q)(y)$ is equal to $P(y)$ modulo $\m$, up to a nonzero multiplicative constant. Since $P(x)$ is irreducible, $R_{\la_r}(Q)(y)$ will be the power of an irreducible polynomial of $\ff{}[y]$. Since $Q(\t)=0$, $F(x)$ is a divisor of $Q(x)$ and the residual polynomial of $F(x)$ will be the  power of an irreducible polynomial too, by the Theorem of the product. This will end the proof of the theorem in the irreducible case.

Let us bound by below all $v_r(B_{je_r}\phi_r^{je_r})$. Denote $u:=v_r(\Phi)$. By Proposition \ref{vrphii} and (\ref{vrphir}), we get: $v_r(\pi_{r-1})=e_{r-1}$, $v_r(\pi_r)=1$, and  $v_r(\Pi)=-e_rv_r(\phi_r)-h_r$. Therefore,
\begin{equation}\label{uj}
u_{je_r}:=v_r(B_{je_r}\phi_r^{je_r})=v_r(b_j)+u-j(e_rv_r(\phi_r)+h_r)+je_rv_r(\phi_r)\ge u-jh_r.
\end{equation}
For $j=0,k$ we have $v(b_0)=0$ (because $v(\ga_r(\t))=0$) and $v(b_k)=0$ (because $b_k=1$). Hence, equality holds in (\ref{uj}) for these two abscissas. This proves that $N_r(g)$ has only one side $T$, with end points $(0,u)$, $(ke_r,u-kh_r)$, and slope $\la_r$.

Let $R_{\la_r}(g)(y)=\sum_{j=0}^kc_{je_r}y^j$. We want to show that $c_{je_r}=c\bar{b}_j$ for certain constant $c\in\ff{}^*$ independent of $j$. If $(je_r,u_{je_r})\not \in T$, then $c_{je_r}=0$, and by (\ref{uj}), this is equivalent to $\bar{b_j}=0$. Suppose now $(je_r,u_{je_r})\in T$; by item 1 of Proposition \ref{vqtr} (cf. (\ref{rdm})) 
$$
\rdm\left(\dfrac{B_{je_r}(\t)\phi_r(\t)^{je_r}}{\pi_r(\t)^u}\right)= c_{je_r}\gb{r}^j.
$$ Hence, we want to check that for all $j$
$$
\rdm\left(\dfrac{B_{je_r}(\t)\phi_r(\t)^{je_r}}{\pi_r(\t)^u\ga_r(\t)^j}\right)=c\bar{b}_j,
$$for some nonzero constant $c$. Now, if we substitute $B_{je_r}(x)$ and $\Pi(x)$ by its defining values, the left hand side is equal to $c\bar{b}_j$, for $c=\rdm(\Phi(\t)/\pi_r(\t)^u)$. This ends the proof of the theorem in the irreducible case.

In the general case, consider the decomposition, $F(x)=\prod_jP_j(x)$, into a product of monic irreducible factors in $\zpx$. By Lemma \ref{factortype}, each $P_j(x)$ has type $\ty$, so that $\om_r(P_j)>0$. By the Theorem of the product, $N_r(P_j)$ is one-sided, of positive length and slope $\la_r$. By the proof in the irreducible case, the residual polynomial $R_{\la_r}(P_j)(y)$ is the positive power of an irreducible polynomial, and by the Theorem of the product it must be $R_{\la_r}(P_j)(y)\sim\psi_{r,i}(y)^{b_j}$ for some $1\le i\le t$. If we group these factors according to the irreducible factor of the residual polynomial, we get the desired factorization.
\end{proof}

\begin{corollary}\label{ramr}
With the above notations, let $\t\in\qpb$ be a root of $G_i(x)$, and $L=K(\t)$. Let $f_r=\dg \psi_{r,i}(y)$. Then, $f(L/K)$ is divisible by $f_0f_1\cdots f_r$. Moreover, the number of irreducible factors of $G_i(x)$ is at most $a_i$; in particular, if $a_i=1$ then $G_i(x)$ is irreducible in $\zpx$ and 
$$f(L/K)=f_0f_1\cdots f_r,\quad e(L/K)=e_1\cdots e_{r-1}e_r.
$$
\end{corollary}

\begin{proof}
The statement about $f(L/K)$ is a consequence of the extension of the embeding (\ref{embeddingr}) to an embedding
\begin{equation}\label{embeddingr+1}
\ff{r}[y]/(\psi_{r,i}(y))\hookrightarrow\ff{L}, \qquad y\mapsto \gb{r},
\end{equation}
which is well-defined by item 4 of Proposition \ref{vqtr}. The other statements follow from the Theorem of the product. The computation of $f(L/K)$ and $e(L/K)$ follows from 
$$f(L/K)e(L/K)=\dg G_i=f_0f_1\cdots f_re_1\cdots e_{r-1}e_r,
$$ and the fact that $f(L/K)$ is divisible by $f_0\cdots f_r$ and $e(L/K)$ is divisible by $e_1\cdots e_r$ (cf. Corollary
\ref{eeee}). 
\end{proof}

\subsection{Types of order $r$ attached to a separable polynomial}
Let $f(x)\in\zpx$ be a monic separable polynomial.
\begin{definition}
Let $\ty$ be a type of order $r-1$. We say that $\ty$ is \emph{$f$-complete}, if $\om_r(f)=1$. In this case, $f_\ty(x)$ is irreducible and the ramification index and residual degree of the extension of $K$ determined by $f_\ty(x)$ can be computed in terms of some data of $\ty$, by applying Corollary \ref{ramr} in order $r-1$ (Corollary \ref{ram} if $r=2$).
\end{definition}

The results of section \ref{secOre} can be interpreted as the addition of \emph{two more dissections}, for each order $2,\dots,r$, to the three classical ones, in the process of factorization of $f(x)$. If  $\ty$ is a type of order $r-1$ and  $\om_r(f)>1$, the factor $f_{\ty}(x)$ experiments further factorizations at two levels: first $f_{\ty}(x)$ factorizes into as many factors as the number of sides of $N_r^-(f)$, and then, the factor corresponding to each finite slope splits into the product of as many factors as the number of pairwise different irreducible factors of the residual polynomial attached to the slope.

We can think that the type $\ty$ has sprouted to produce several types of order $r$,
$\ty'=(\tilde{\ty};\la_r,\psi_r(y))$, each of them distinguished by the choice of a finite slope $\la_r$ of a side of $N_r^-(f)$, and a monic irreducible factor $\psi_r(y)$ of $R_{\la_r}(f)(y)$ in $\ff{r}[y]$.

\begin{definition}
In Sect. \ref{orderone}, we defined two sets $\ty_0(f)$, $\ty_1(f)$. We recursively define $\ty_r(f)$ to be the set of all types  of order $r$ constructed as above, $\ty'=(\tilde{\ty};\la_r,\psi_r(y))$, from those $\ty\in\ty_{r-1}(f)$ that are not $f$-complete.
This set is not an intrinsic invariant of $f(x)$ because it depends on the choices of the representatives $\phi_1(x),\dots,\phi_r(x)$ of the truncations of $\ty$.

We denote by $\ty_s(f)^{\op{compl}}$ the subset of the $f$-complete types of $\ty_s(f)$. We define
$$
\Ty_r(f):=\ty_r(f)\cup\left(\bigcup_{0\le s<r}\ty_s(f)^{\op{compl}}\right).
$$  
\end{definition}

Hensel's lemma and the theorems of the polygon and of the residual polynomial in orders $1,\dots,r$ determine a factorization   
\begin{equation}\label{theend}
f(x)=f_{r,\infty}(x)\prod_{\ty\in\Ty_r(f)}f_{\ty}(x),
\end{equation}
where $f_{r,\infty}(x)$ is the product of the different representatives $\phi_i(x)$ (of the different types in $\Ty_r(f)$) that divide $f(x)$ in $\zpx$. 

The following remark is an immediate consequence of the definitions.
\begin{lemma}\label{tempty}
The following conditions are equivalent:
\begin{enumerate}
 \item $\ty_{r+1}(f)=\emptyset$,
\item $\ty_r(f)^{\op{compl}}=\ty_r(f)$,
\item For all $\ty\in \ty_{r-1}(f)$ and all $\la_r\in\Q^-$, the residual polynomial of $r$-th order, $R_{\la_r}(f)(y)$ is separable.\hfill{$\Box$}
\end{enumerate}
\end{lemma}

If these conditions are satisfied, then (\ref{theend}) is a factorization of $f(x)$ into the product of monic irreducible polynomials in $\zpx$, and we get arithmetic information about each factor by Corollary \ref{ramr}. As long as there is some $\ty\in\ty_r(f)$ which is not $f$-complete, we must apply the results of this section in order $r+1$ to get further factorizations of $f_\ty(x)$, or to detect that it is irreducible.
We need some invariant to control the whole process and ensure that after a finite number of steps we shall have $\ty_r(f)^{\op{compl}}=\ty_r(f)$. This is the aim of the next section.

We end with a remark about $p$-adic approximations to the irreducible factors of $f(x)$, that is an immediate consequence of Lemma \ref{typedegree}, the Theorem of the polygon and Proposition \ref{vrphii}.
\begin{proposition}\label{appfactor}
Let $\ty$ be an $f$-complete type of order $r$, with representative $\phi_{r+1}(x)$. Let $\t\in\qpb$ be a root of $f_{\ty}(x)$, and $L=K(\t)$. Then, $\dg \phi_{r+1}=\dg f_{\ty}$, and $\phi_{r+1}(x)$ is an approximation to $f_{\ty}(x)$ satisfying
$$
v(\phi_{r+1}(\t))=(v_{r+1}(\phi_{r+1})+h_{r+1})/e(L/K)=\sum_{i=1}^{r+1}e_if_i\cdots e_rf_r\,\dfrac{h_i}{e_1\dots e_i},
$$where $-h_{r+1}$ is the slope of the unique side of $N_{r+1}^-(f)$, and $e_{r+1}=1$.\hfill{$\Box$}
\end{proposition}

\section{Indices and resultants of higher order} 
We fix throughout this section a natural number $r\ge 1$.

\subsection{Computation of resultants with Newton polygons}
\begin{definition}
Let $\ty$ be a type of order $r-1$ and let $\phi_r(x)\in\zpx$ be a representative of $\ty$.  
For any pair of monic polynomials $P(x),Q(x)\in\zpx$ we define
$$
\res_\ty(P,Q):=f_0\cdots f_{r-1}\left(\sum_{i,j}\min\{E_iH'_j,E'_jH_i\}\right),
$$
where $E_i=\ell(S_i)$, $H_i=H(S_i)$ are the lengths and heights of the sides $S_i$ of $N_r^-(P)$, and
$E'_j=\ell(S'_j)$, $H'_j=H(S'_j)$ are the lengths and heights of the sides $S'_j$ of $N_r^-(Q)$. 
\end{definition}

We recall that for a side $S$ of slope $-\infty$ we took $H(S)=\infty$ by convention. Thus, the part of $\res_\ty(P,Q)$ that involves sides of slope $-\infty$ is always 
\begin{equation}\label{resinfinity}
f_0\cdots f_{r-1}(\ord_{\phi_r}(P)H(Q)+\ord_{\phi_r}(Q)H(P)),
\end{equation}
where $H(P)$, $H(Q)$ are the total heights respectively of $N_r^-(P)$, $N_r^-(Q)$.

\begin{lemma}\label{resbilineal}
Let  $P(x),P'(x),Q(x)\in\zpx$ be monic polynomials.
\begin{enumerate}
 \item $\res_\ty(P,Q)=0$ if and only if $\om_r(P)\om_r(Q)=0$,
 \item $\res_\ty(P,Q)<\infty$ if and only if $\ord_{\phi_r}(P)\ord_{\phi_r}(Q)=0$,
 \item $\res_\ty(P,Q)=\res_\ty(Q,P)$,
 \item $\res_\ty(PP',Q)=\res_\ty(P,Q)+\res_\ty(P',Q)$.
\end{enumerate}
\end{lemma}

\begin{proof}
The three first items are an immediate consequence of the definition. Item 4 follows from 
$N_r^-(PP')=N_r^-(P)+N_r^-(P')$. 
\end{proof}

In the simplest case when $N_r^-(P)$ and $N_r^-(Q)$ are both one-sided, $\res_\ty(P,Q)$ represents the area of the rectangle joining the two triangles determined by the sides, if they are ordered by increasing slope.
The reader may figure out a similar geometrical interpretation of $\res_\ty(P,Q)$ in the general case, as the area of a union of rectangles below the Newton polygon $N_r^-(PQ)=N_r^-(P)+N_r^-(Q)$.

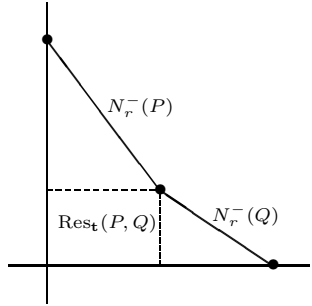
\begin{center}
\setlength{\unitlength}{5.mm}
\begin{picture}(8,7)
\put(2.85,1.85){$\bullet$}\put(5.85,-.15){$\bullet$}\put(-.15,5.85){$\bullet$}\put(-1,0){\line(1,0){8}}
\put(0,-1){\line(0,1){8}}\put(6,0){\line(-3,2){3}}\put(3,2){\line(-3,4){3}}\put(6,0.02){\line(-3,2){3}}
\put(3,2.02){\line(-3,4){3}}
\put(1.6,4.1){\begin{scriptsize}$N_r^-(P)$\end{scriptsize}}
\put(4.4,1.2){\begin{scriptsize}$N_r^-(Q)$\end{scriptsize}}
\put(.3,1){\begin{scriptsize}$\res_\ty(P,Q)$\end{scriptsize}}
\multiput(3,0)(0,.20){10}{\vrule height2pt}
\multiput(0,2)(.25,0){12}{\hbox to 2pt{\hrulefill }}
\end{picture}
\end{center}\be\be

Our aim is to compute $v(\res(P,Q))$ as a sum of several $\res_\ty(P,Q)$ for an adequate choice of types $\ty$. To this end, we want to compare types attached to $P$ and $Q$, and this is uneasy because in the definition of the sets $\ty_r(P)$, $\ty_r(Q)$, we had freedom in the choices of the different representatives $\phi_i(x)$. For commodity in the exposition, we shall assume in this section that these polynomials are universally fixed. \be

\noindent{\bf Convention. }{\it We fix from now on a monic lift $\phi_1(x)\in\zpx$ of every monic irreducible polynomial $\psi_0(y)\in\ff{}[y]$. We proceed now recursively: for any $1\le i<r$ and any type of order $i$
$$
\ty=(\phi_1(x);\la_1,\phi_2(x);\cdots;\la_{i-1},\phi_i(x);\la_i,\psi_i(y)),
$$with $\phi_1(x),\dots,\phi_i(x)$ belonging to the infinite family of previously chosen polynomials, we fix a representative $\phi_{i+1}(x)$ of $\ty$.
Also, we assume from now on that all types are made up only with our chosen polynomials $\phi_i(x)$.}\be 

Once these choices are made, the set $\ty_r(P)$ is uniquely determined by $r$ and $P(x)$. More precisely, $\ty_r(P)$ is the set of all types $\ty$ of order $r$ such that $\om_{r+1}^{\ty}(P)>0$ and the truncation $\ty_{r-1}$ is not $P$-complete; in other words,
$$\ty_r(P)=\{\ty\mbox{ type of order $r$ such that }\om_{r+1}^{\ty}(P)>0,\,\om_r^{\ty}(P)>1\}.
$$However, in view of the computation of resultants, we need a broader concept of ``type attached to a polynomial".

\begin{definition}
For any monic polynomial $P(x)\in\zpx$, we define $$\tt_r(P):=\{\ty\mbox{ type of order $r$ such that }\om_{r+1}^{\ty}(P)>0\}\supseteq \ty_r(P).
$$
\end{definition}

The following observation is a consequence of the fact that $\om_{r+1}^{\ty}$ is a semigroup homomorphism for every type $\ty$ of order $r$. 
\begin{lemma}\label{tunion}
For any two monic polynomials $P(x),Q(x)\in\zpx$, we have $\tt_r(PQ)=\tt_r(P)\cup \tt_r(Q)$.
\hfill{$\Box$}\end{lemma}

Note that the analogous statement for the sets $\ty_r(P)$ is false. For instance, let $P(x),Q(x)$ be two monic polynomials congruent to the same irreducible polynomial $\psi(y)$ modulo $\m$. We have $\ty_0(P)=\ty_0(Q)=\{\psi(y)\}=\ty_0(PQ)$, and the type of order zero $\psi(y)$ is $P$-complete and $Q$-complete; thus, $\ty_1(P)=\emptyset=\ty_1(Q)$. However, $\psi(y)$ is not $PQ$-complete, and $\ty_1(PQ)\ne\emptyset$.  

We could also build the set $\tt_r(P)$ in a constructive way analogous to that used in the last section to construct $\ty_r(P)$. The only difference is that the $P$-complete types of order $r-1$ are expanded to produce types of order $r$ as well. 
Thanks to our above convention 
about fixing a universal family of representatives of the types, these expansions are unique.

\begin{lemma}\label{uniquesprout}
Let $P(x)\in\zpx$ be a monic polynomial. Let $\ty$ be a $P$-complete type of order $r-1$ with representative $\phi_r(x)$, and suppose that $P(x)$ is not divisible by $\phi_r(x)$ in $\zpx$. Then, $\ty$ can be extended to a unique type $\ty'\in\tt_r(P)$ such that $\ty'_{r-1}=\ty$. The type $\ty'$ is $P$-complete too.  
\end{lemma}

\begin{proof}
By Lemma \ref{shape}, $N_r^-(P)$ has length one and finite slope $\la_r\in\Q^-$; hence, $\dg R_{\la_r}(P)(y)=d(N_r^-(P))=1$. Let $\psi_r(y)$ be the monic polynomial of degree one determined by $R_{\la_r}(P)(y)\sim \psi_r(y)$. The type $\ty'=(\tilde{\ty};\la_r,\psi_r(y))$ is $P$-complete, and it is the unique type of order $r$ such that $\ty'_{r-1}=\ty$ and $\om_{r+1}^{\ty'}(P)>0$. In fact, let us check that $\om_{r+1}^{\ty''}(P)=0$ for any $\ty''=(\tilde{\ty};\la'_r,\psi'_r(y))\ne \ty'$.
If $\la'_r\ne\la_r$, then $R_{\la'_r}(P)$ is a constant; if $\la'_r=\la_r$, but $\psi_r(y)\ne \psi'_r(y)$ then $\psi'_r(y)$ cannot divide $R_{\la_r}(P)(y)$.
\end{proof}

\begin{lemma}\label{empty}
Let $P(x)\in\zpx$ be a monic polynomial. Then,
$\tt_r(P)=\emptyset$ if and only if all irreducible factors of $P(x)$ are the representative of some type of order $0,1,\dots,r-1$.
Moreover, if $P(x)$ is irreducible and $\tt_r(P)\ne\emptyset$, then $|\tt_r(P)|=1$.  
\end{lemma}

\begin{proof}
By Lemma \ref{tunion}, we can assume that $P(x)$ is irreducible.
If $P(x)=\phi_s(x)$ is the representative of some type of order $s-1\le r-1$, then
$N_s(\phi_s)$ is one-sided of slope $-\infty$; hence, $R_{\la_s}(\phi_s)$ is a constant for every $\la_s\in\Q^-$, and $\om^{\ty'}_{s+1}(\phi_s)=0$, for every type $\ty'$ of order $\ge s$. Thus, $\tt_r(\phi_s)=\emptyset$, for all $r\ge s$. Otherwise, the theorems of the polygon and of the residual polynomial show that the unique element of  $\tt_0(P)$ can be successively extended to a unique element of $\tt_1(P),\dots,\tt_r(P)$. 
\end{proof}
 
\begin{definition}
For any pair of monic polynomials $P(x),Q(x)\in\zpx$, and any natural number $r\ge 1$, we define $$\res_r(P,Q):=\sum_{\ty\in\tt_{r-1}(P)\cap\tt_{r-1}(Q)}\res_\ty(P,Q).$$
\end{definition}

The following observation is an immediate consequence of Lemma \ref{resbilineal}.
\begin{lemma}\label{resr+1}The following conditions are equivalent:
 \begin{enumerate}
 \item $\res_{r+1}(P,Q)=0$,
\item $\tt_r(P)\cap\tt_r(Q)=\emptyset$,
\item For all $\ty\in \tt_{r-1}(P)\cap\tt_{r-1}(Q)$ and all $\la_r\in\Q^-$, the residual polynomials of $r$-th order, $R_{\la_r}(P)(y)$, $R_{\la_r}(Q)(y)$, have no common factor in
$\ff{r}[y]$.\hfill{$\Box$}
\end{enumerate}
\end{lemma}

The following result is an immediate consequence of Lemmas \ref{resbilineal} and \ref{tunion}.
\begin{lemma}\label{resbilineal2}
For any three monic polynomials $P(x),P'(x),Q(x)\in\zpx$, and any natural number $r\ge 1$, we have $\res_r(PP',Q)=\res_r(P,Q)+\res_r(P',Q)$.\hfill{$\Box$}
\end{lemma}

\begin{theorem}\label{resultant}
Let $P(x),Q(x)\in\zpx$ be two monic polynomials having no common factors, and let $r\ge 1$ be natural number. Then, 
\begin{enumerate}
 \item $v(\res(P,Q))\ge \res_1(P,Q)+\cdots+\res_r(P,Q)$, 
\item Equality holds if and only if $\res_{r+1}(P,Q)=0$.
\end{enumerate}
\end{theorem}

\begin{proof}
Let us deal first with the case where $P(x)$, $Q(x)$ are both irreducible and $\tt_{r-1}(P)=\tt_{r-1}(Q)=\{\ty\}$, for some type $\ty=(\phi_1(x);\cdots;\la_{r-1},\psi_{r-1}(x))$. For $0\le i\le r$, let $E_i,H_i$ be the length and heigth of the unique side of $N_i(P)$, and $E'_i,H'_i$ be the length and heigth of the unique side of $N_i(Q)$. By Lemma \ref{factortype}, $P$ and $Q$ are both of type $\ty$; thus, $H_i/E_i=H'_i/E'_i$, for all $1\le i<r$, with $E_iE'_i>0$, $0<H_iH'_i<\infty$ (because $P$, $Q$ cannot be both equal to $\phi_i(x)$). Suppose that $-\la_r:=H_r/E_r\le H'_r/E'_r=:-\la'_r$; in particular, $H_r<\infty$, and $P(x)\ne\phi_r(x)$.

Let $\phi_r(x)$ be the representative of $\ty$, and let $R_{\la'_r}(Q)(y)$ be the $r$-th order residual polynomial of $Q(x)$ with respect to $(\tilde{\ty};\la'_r)$. By the Theorem of the residual polynomial, $R_{\la'_r}(Q)(y)\sim \psi_r'(y)^{a'}$ for some monic irreducible polynomial $\psi'_r(y)\in\ff{r}[y]$. Consider the type of order $r$, $\ty'=(\tilde{\ty};\la'_r,\psi'_r(y))$. 

It is well-known that
$$
\res(P,Q)=\pm \prod_{Q(\t)=0}P(\t).
$$By applying Proposition \ref{vpt} to the type $\ty'$, we get 
\begin{equation}\label{vr+1}
v(P(\t))\ge v'_{r+1}(P)/e_1\cdots e_{r-1}e'_r=(v_r(P)+H_r)/e_1\cdots e_{r-1},
\end{equation}
the last equality by the definition of $v'_{r+1}$. Also, equality holds in (\ref{vr+1}) if and only if $\om'_{r+1}(P)=0$, where $\om'_{r+1}$ is the pseudo-valuation of order $r+1$ attached to $\ty'$.\be

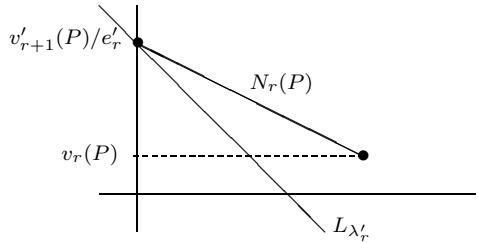
\begin{center}
\setlength{\unitlength}{5.mm}
\begin{picture}(10,5)
\put(-.15,3.85){$\bullet$}\put(5.85,.85){$\bullet$}
\put(-1,0){\line(1,0){10}}
\put(6,1){\line(-2,1){6}}\put(6,1.02){\line(-2,1){6}}
\put(0,-1){\line(0,1){6}}\put(-1,5){\line(1,-1){6}}
\put(5.2,-1){\begin{footnotesize}$L_{\la'_r}$\end{footnotesize}}
\put(3,2.8){\begin{footnotesize}$N_r(P)$\end{footnotesize}}
\put(-3.4,4){\begin{footnotesize}$v'_{r+1}(P)/e'_r$\end{footnotesize}}
\multiput(-.1,1)(.25,0){25}{\hbox to 2pt{\hrulefill }}
\put(-2,.9){\begin{footnotesize}$v_r(P)$\end{footnotesize}}
\end{picture}
\end{center}\be\be
By Lemma  \ref{typedegree},  $\dg Q=m_r\om_r(Q)=f_0e_1f_1\cdots e_{r-1}f_{r-1}E'_r$. If we apply recursively $v_{s+1}(P)=e_s(v_s(P_s)+H_s)$, $E'_{s+1}=(e_sf_s)^{-1}E'_s$, for all $1\le s<r$, and $v_1(P)=0$, we get 
\begin{align*}
v(\res(P,Q))=&\, \dg(Q) v(P(\t))\ge \dg Q\, \dfrac{v_r(P)+H_r}{e_1\cdots e_{r-1}}\\=&\,f_0\cdots f_{r-1} E'_r(v_r(P)+H_r)\\
=&\,\sum_{s=1}^rf_0\cdots f_{s-1}E'_sH_s=\sum_{s=1}^r\res_s(P,Q),
\end{align*}
and equality holds if and only if  $\om'_{r+1}(P)=0$, i.e. if and only if 
$R_{\la'_r}(P)(y)$  is not divisible by $\psi'_r(y)$. This condition is equivalent to (3) of Lemma \ref{resr+1} because  $R_{\la_r}(P)(y)\sim \psi_r(y)^a$ for some monic irreducible polynomial $\psi_r(y)\in\ff{r}[y]$, and $R_{\la''_r}(P)(y)$ is a constant for any negative rational number $\la''_r\ne\la_r$. This ends the proof of the theorem in this case.

Let us still assume that $P(x)$, $Q(x)$ are both irreducible, but now $\tt_{r-1}(P)\cap\tt_{r-1}(Q)=\emptyset$. If $\tt_0(P)\cap \tt_0(Q)=\emptyset$, then $\res_1(P,Q)=\cdots=\res_r(P,Q)=\res_{r+1}(P,Q)=0$, by definition. On the other hand, $v(\res(P,Q))=0$, because $P(x)$ and $Q(x)$ have no common factors modulo $\m$; hence, the theorem is proven in this case. Assume $\tt_0(P)\cap \tt_0(Q)\ne\emptyset$, and let $1\le s< r$ be maximal with the property $\tt_{s-1}(P)\cap\tt_{s-1}(Q)\ne\emptyset$. Clearly, $\res_r(P,Q)=0$ for all $r>s$; thus, we want to show that $v(\res(P,Q))=\res_1(P,Q)+\cdots+\res_s(P,Q)$, and this follows from the proof of the previous case for $r=s$. 

Let now $P(x)=P_1(x)\cdots P_g(x)$, $Q(x)=Q_1(x)\cdots Q_{g'}(x)$ be the  factorizations of $P(x)$, $Q(x)$ into a product of monic irreducible polynomials in $\zpx$. We have proved above that $v(\res(P_i,Q_j))\ge \res_1(P_i,Q_j)+\cdots+\res_r(P_i,Q_j)$ for all $i,j$; hence, item 1 follows from Lemma \ref{resbilineal2} and the bilinearity of resultants. Also, equality in item 1 holds for the pair $P,\,Q$ if and only if it holds for each pair $P_i,\,Q_j$; that is, if and only if $\res_{r+1}(P_i,Q_j)=0$, for all $i,j$. This is equivalent to $\res_{r+1}(P,Q)=0$, again by Lemma \ref{resbilineal2}. 
\end{proof}

We end this section with an example that illustrates the necessity to introduce the sets $\tt_r(P)$. Let $\oo=\Z_p$, $P(x)=x+p$, $Q(x)=x+p+p^{100}$, and let $\ty_0=y\in\ff{}[y]$. Clearly, $\ty_0(P)=\{\ty_0\}=\ty_0(Q)$, and $\ty_0$ is both $P$-complete and $Q$-complete, so that $\ty_1(P)=\emptyset=\ty_1(Q)$. If we take $\phi_1(x)=x$, we get $\res_1(P,Q)=\res_{\ty_0}(P,Q)=1$, whereas $v(\res(P,Q))=100$. Thus, we need to consider the expansions of $\ty_0$ to types of higher order in order to reach the right value of $v(\res(P,Q))$. The number of expansions to consider depends on the choices of the representatives $\phi_i(x)$; for instance, if we take $\ty=(x;-1,y+1)$, with representative $\phi_2(x)=x+p$, we have already $\res_2(P,Q)=99$.   
 
Nevertheless, the sets $\tt_r(P)$ were introduced only as an auxiliary tool to prove Theorem \ref{resultant}. In practice, the factorization algorithm computes only the sets $\ty_r(P)$, as we shall show in the next section.
 
\subsection{Index of a polynomial and index of a polygon}
All types that we consider are still assumed to be made up with  polynomials $\phi_i(x)$ belonging to a universally fixed family, as  indicated in the last section.
  
Let $F(x)\in\zpx$ be a monic irreducible polynomial, $\t\in\qpb$ a root of $F(x)$, and $L=K(\t)$. It is well-known that $(\ol\colon\oo[\t])=q^{\ind(F)}$, for some natural number $\ind(F)$ that will be called the 
 \emph{$v$-index} of $F(x)$. Note that
$$\ind(F)=v(\ol\colon\oo[\t])/[K\colon\Q_p].$$ Recall the well-known relationship,
$v(\dsc(F))=2\ind(F)+v(\dsc(L/K))$, linking $\ind(F)$ with the discriminant of $F(x)$ and the discriminant of $L/K$. 

\begin{definition}
Let  $f(x)\in\zpx$ be a monic separable polynomial and $f(x)=F_1(x)\cdots F_k(x)$ its decomposition into the product of monic irreducible polynomials in $\zpx$. We define the \emph{index} of $f(x)$ by the formula
$$
\ind(f):=\sum_{i=1}^k\ind(F_i)+\sum_{1\le i<j\le k}v(\res(F_i,F_j)).
$$ 
\end{definition}

\begin{definition}
Let $S$ be a side of negative slope, and denote $E=\ell(S)$, $H=H(S)$, $d=d(S)$. We define 
$$\ind(S):=\left\{\begin{array}{ll}
\frac12(EH-E-H+d),&\mbox{ if $S$ has finite slope},\\
0,&\mbox{ otherwise}.
\end{array}
\right.
$$

Let $N=S_1+\cdots+S_g$ be a principal polygon, with sides ordered by increasing slopes $\la_1<\cdots<\la_g$. We define
$$
\ind(N):=\sum_{i=1}^g\ind(S_i)+\sum_{1\le i<j\le g}E_iH_j.
$$
\end{definition}
If $N$ has a side $S_1$ of slope $-\infty$ and length $E_{\infty}:=E_1$, it contributes with $E_{\infty}H_{\op{fin}}(N)$ to $\ind(N)$, where $H_{\op{fin}}(N)$ is the total height of the finite part of $N$. 

\begin{remark}\label{indzero}
Note that $\ind(N)=0$ if and only if either $N$ is reduced to a point, or $N$ is one-sided with slope $-\infty$, or  $N$ is one-sided with $E=1$ or $H=1$.
\end{remark}

\begin{remark}\label{region}
The contribution of the sides of finite slope to $\ind(N)$ is the number of points of integer coordinates that lie below or on the finite part of $N$, strictly above the horizontal line $L$ that passes through the last point of $N$, and strictly beyond the vertical line $L'$ that passes through the initial point of the finite part of $N$. 
\end{remark}

For instance, the polygon below has index $25$, the infinite side contributes with $18$ (the area of the rectangle $3\times 6$) and the finite part has index $7$, corresponding to the marked seven points of integers coordinates, distributed into $\ind(S_1)=2$, $\ind(S_2)=1$, $E_1H_2=4$.

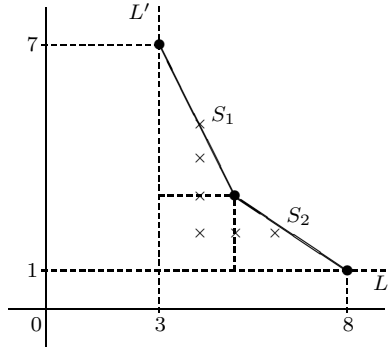
\begin{center}
\setlength{\unitlength}{5.mm}
\begin{picture}(14,9)
\put(4.85,3.85){$\bullet$}\put(2.85,7.85){$\bullet$}\put(7.85,1.85){$\bullet$}
\put(-1,1){\line(1,0){10}}
\put(0,0){\line(0,1){9}}
\put(8,2){\line(-3,2){3}}\put(5,4.05){\line(-1,2){2}}\put(8,2.03){\line(-3,2){3}}
\put(5,4.08){\line(-1,2){2}}
\multiput(5.,2)(0,.25){8}{\vrule height2pt}
\multiput(3,.9)(0,.25){33}{\vrule height2pt}
\multiput(8,.9)(0,.25){4}{\vrule height2pt}
\multiput(-.1,2)(.25,0){37}{\hbox to 2pt{\hrulefill }}
\multiput(-.1,8)(.25,0){13}{\hbox to 2pt{\hrulefill }}
\multiput(3,4)(.25,0){8}{\hbox to 2pt{\hrulefill }}
\put(4.4,6){\begin{footnotesize}$S_1$\end{footnotesize}}
\put(6.4,3.3){\begin{footnotesize}$S_2$\end{footnotesize}}
\put(8.7,1.5){\begin{footnotesize}$L$\end{footnotesize}}
\put(2.2,8.7){\begin{footnotesize}$L'$\end{footnotesize}}
\put(-.4,.4){\begin{footnotesize}$0$\end{footnotesize}}
\put(2.9,.4){\begin{footnotesize}$3$\end{footnotesize}}
\put(7.9,.4){\begin{footnotesize}$8$\end{footnotesize}}
\put(-.5,1.85){\begin{footnotesize}$1$\end{footnotesize}}
\put(-.5,7.85){\begin{footnotesize}$7$\end{footnotesize}}
\put(3.9,2.9){\begin{tiny}$\times$\end{tiny}}
\put(4.85,2.9){\begin{tiny}$\times$\end{tiny}}
\put(5.9,2.9){\begin{tiny}$\times$\end{tiny}}
\put(3.9,3.9){\begin{tiny}$\times$\end{tiny}}
\put(3.9,4.9){\begin{tiny}$\times$\end{tiny}}
\put(3.9,5.8){\begin{tiny}$\times$\end{tiny}}
\end{picture}
\end{center}

Let  $i_1\le i_2$ be the respective abscissas of the starting point and the last point of the finite part $N_{\op{fin}}$ of $N$.  
For any integer abscissa $i_1\le i\le i_2$, let $y_i$ be the distance of the point of $N$ of abscissa $i$ to the line $L$. Clearly, we can count the points of integer coordinates below $N_{\op{fin}}$, above $L$ and beyond $L'$, as the sum of the points with given abscissa: 
\begin{equation}\label{combi}
\ind(N_{\op{fin}})=\lfloor y_{i_1+1}\rfloor+\cdots+\lfloor y_{i_2-1}\rfloor.
\end{equation}
For instance, in the above figure we have $y_4=4$, $y_5=2$, $y_6=1$ and $y_7=0$.

\begin{definition}
Let $P(x)\in\zpx$ be a monic and separable polynomial. 
Let $\ty$ be a type of order $r-1$ and $\phi_r(x)$ a representative of $\ty$. We define
$$\ind_\ty(P):=f_0\cdots f_{r-1}\ind({N_r^-}(P)),$$
where $N_r(P)$ is the Newton polygon of $r$-th order with respect to $\tilde{\ty}$.

For any natural number $r\ge 1$ we define
$$
\ind_r(P):=\sum_{\ty\in\ty_{r-1}(P)}\ind_\ty(P).
$$
\end{definition}

Since the Newton polygon $N_{r}^-(P)$ depends on the choice of $\phi_{r}(x)$, the value of $\ind_\ty(P)$ depends on this choice too, although this is not reflected in the notation. 

\begin{lemma}\label{zeroind}
Let $P(x)\in\zpx$ be a monic and separable polynomial. 
\begin{enumerate}
\item Let $\ty$ be a type of order $r$, and suppose that $\ty\not\in\ty_r(P)$ or $\ty$ is $P$-complete. Then, $\ind_{\ty}(P)=0$.
\item If $\ty_r(P)=\ty_r(P)^{\op{compl}}$ then $\ind_{r+1}(P)=0$.
\item If $\ind_r(P)=0$, then $\ty_r(P)=\ty_r(P)^{\op{compl}}$.
\end{enumerate}
\end{lemma}

\begin{proof}
If $\ty\not\in\ty_{r-1}(P)$, then either $\om_r(P)=0$ or $\om_{r-1}(P)=1$. If $\ty$ is $P$-complete then $\om_r(P)=1$. By Lemmas \ref{typedegree} and \ref{shape}, in all cases $\ell(N_r^-(P))=\om_r(P)\le 1$, and $\ind_{\ty}(P)=0$ by Remark \ref{indzero}. This proves item 1, and item 2 is an immediate consequence.

If $\ind_r(P)=0$, then $\ind_{\ty}(P)=0$ for all $\ty\in\ty_{r-1}(P)$. For any such $\ty$ we have $\om_r(P)>0$, so that $N_r^-(P)$ is not reduced to a point. By Remark \ref{indzero}, $N_r^-(P)$ is one-sided with either slope $-\infty$, or length one, or height one. In the first case $P(x)$ is divisible by the representative $\phi_r(x)$ of $\ty$ and $\om_r(P)=\ell(N_r^-(P))=\ord_{\phi_r}(P)=1$, because $P(x)$ is separable; thus, $\ty$ is $P$-complete and $\ty$ is not extended to any type in $\ty_r(P)$. If $N_r^-(P)$ is one-sided with finite slope $\la_r$ and the side has degree one, then the residual polynomial $R_{\la_r}(P)(y)$ has degree one. Thus, $\ty$ is either $P$-complete or it can be extended in a unique way to a type $\ty'\in \ty_r(P)$; in the latter case, necessarily $\om_{r+1}^{\ty'}(P)=1$ and $\ty'$ is $P$-complete. This proves item 3. 
\end{proof}

\begin{lemma}\label{indres}
Let $P(x),Q(x)\in\zpx$ be two monic and separable polynomials, without common factors. Let $r\ge 1$ be a natural number and $\ty$ a type of order $r-1$. Then,
$$\as{1.6}
\begin{array}{l}
\ind_\ty(PQ)=\ind_\ty(P)+\ind_\ty(Q)+\res_\ty(P,Q),\\
\ind_r(PQ)=\ind_r(P)+\ind_r(Q)+\res_r(P,Q). 
\end{array}
$$
\end{lemma}

\begin{proof}For commodity, in the discussion we omit the weight $f_0\cdots f_{r-1}$ that multiplies all terms in the identities.

All terms involved in the first identity are the sum of a finite part and an infinite part. If $P(x)Q(x)$ is not divisible by $\phi_r(x)$, all infinite parts are zero. If $\phi_r(x)$ divides (say) $P(x)$, then the infinite part of $\ind_\ty(PQ)$ is $\ord_{\phi_r}(P)(H_{\op{fin}}(P)+H_{\op{fin}}(Q))$, the infinite part of $\ind_\ty(P)$ is 
$\ord_{\phi_r}(P)H_{\op{fin}}(P)$, the infinite part of $\ind_\ty(Q)$ is zero, and the infinite part of $\res_\ty(P,Q)$ is $\ord_{\phi_r}(P)H(Q)$, by (\ref{resinfinity}). Thus, the first identity is correct, as far as the infinite parts are concerned. 

The finite part of the first identity follows from $N_r^-(PQ)=N_r^-(P)+N_r^-(Q)$ and Remark \ref{region}. Let $N=N_r^-(PQ)$ and let $\mathcal{R}$ be the region of the plane that lies below $N$, above the line $L$ and beyond the line $L'$, as indicated in Remark \ref{region}. The number $\ind_\ty(PQ)$ counts the total number of points of integer coordinates in $\mathcal{R}$, the number
$\ind_\ty(P)+\ind_\ty(Q)$ counts the number of points of integer cordinates in the regions determined by the right triangles whose hypotenuses are the sides of $N_r^-(P)$ and $N_r^-(Q)$. The region of $\mathcal{R}$ not covered by these triangles is a union of rectangles and $\res_\ty(P,Q)$ is precisely the number of points of integer coordinates of this region.

In order to prove the second identity, we note first that for any monic separable polynomial $R(x)\in\zpx$, 
$$
\ind_r(R)=\sum_{\ty\in\tt_{r-1}(R)}\ind_\ty(R),
$$
by (1) of Lemma \ref{zeroind}. Now, if we apply this to $R=P,Q,PQ$, the identity follows from the first one and Lemma \ref{tunion}, having in mind that $\ind_\ty(Q)=0=\res_\ty(P,Q)$ if $\ty\not\in \tt_{r-1}(Q)$, because $N_r^-(Q)$ reduces to a point. 
\end{proof}

We are ready to state the Theorem of the index, which is a crucial ingredient of the factorization process. It ensures that an algorithm based on the computation of the sets $\ty_r(f)$ and the higher indices $\ind_r(f)$ obtains the factorization of $f(x)$, and relevant aritmetic information on the irreducible factors, after a finite number of steps. Also, this algorithm yields a computation of $\ind(f)$ as a by-product.

\begin{theorem}[Theorem of the index] \label{thindex}
Let $f(x)\in\zpx$ be a monic and separable polynomial, and $r\ge 1$ a natural number. Then,
\begin{enumerate}
 \item $\ind(f)\ge \ind_1(f)+\cdots +\ind_r(f)$, and
\item equality holds if and only if $\ind_{r+1}(f)=0$. 
\end{enumerate}
\end{theorem}

Note that Lemma \ref{zeroind} and this theorem guarantee the equality in (1)
whenever all types of $\ty_r(f)$ are $f$-complete. Also, Theorem \ref{thindex} shows that this latter condition will be reached at some order $r$.

\begin{corollary}
Let $f(x)\in\zpx$ be a monic and separable polynomial. There exists $r\ge 0$ such that all types in $\ty_r(f)$ are $f$-complete, or equivalently, such that $\ty_{r+1}(f)=\emptyset$.
\end{corollary}

\begin{proof}
By the Theorem of the index, there exists $r\ge1$ such that $\ind_r(f)=0$, and by (2) of Lemma \ref{zeroind}
this implies $\ty_r(f)=\ty_r(f)^{\op{compl}}$.
\end{proof}

In the next section we exhibit an example where the factorization is achieved in order three. More examples, and a more accurate discussion of the computational aspects can be found in \cite{gmna}.

\subsection{An example}\label{example}
Take $p=2$, and $f(x)=x^4+ax^2+bx+c\in\Z[x]$, with $v(a)\ge 2$, $v(b)=3$, $v(c)=2$. This polynomial has $v(\op{disc}(f))=12$ for all $a,b,c$ with these restrictions.
Since $f(x)\equiv x^4\md2$, all types we are going to consider will start with $\phi_1(x)=x$. The Newton polygon $N_1(f)$ has slope $\la_1=-1/2$
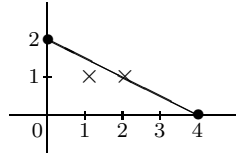
\begin{center}
\setlength{\unitlength}{5.mm}
\begin{picture}(6,4)
\put(-.15,1.85){$\bullet$}\put(3.85,-.15){$\bullet$}
\put(-1,0){\line(1,0){6}}\put(0,-1){\line(0,1){4}}
\put(4,0){\line(-2,1){4}}\put(4,0.02){\line(-2,1){4}}
\put(-.4,-.6){\begin{footnotesize}$0$\end{footnotesize}}
\put(.85,-.6){\begin{footnotesize}$1$\end{footnotesize}}
\put(1.85,-.6){\begin{footnotesize}$2$\end{footnotesize}}
\put(2.85,-.6){\begin{footnotesize}$3$\end{footnotesize}}
\put(3.85,-.6){\begin{footnotesize}$4$\end{footnotesize}}
\put(-.5,.85){\begin{footnotesize}$1$\end{footnotesize}}
\put(-.5,1.85){\begin{footnotesize}$2$\end{footnotesize}}
\put(.85,.85){$\times$}\put(1.8,.85){$\times$}
\put(-.1,1){\line(1,0){.2}}
\put(1,-.1){\line(0,1){.2}}
\put(2,-.1){\line(0,1){.2}}
\put(3,-.1){\line(0,1){.2}}
\end{picture}
\end{center}\be\be

\noindent and the residual polynomial of $f(x)$ with respect to $\la_1$ is $R_1(f)(y)=y^2+1=(y+1)^2\in\ff{}$, where $\ff{}$ is the field ot two elements. Hence, $\ty_1(f)=\{\ty\}$, where $\ty:=(x;-1/2,y+1)$. We have $e_1=2,f_0=f_1=1$ and $\om_2(f)=2$, so that $\ty$ is not $f$-complete. The partial information we get in order one is $\ind_1(f)=2$, and the fact that all irreducible factors of $f(x)$ will generate extensions $L/\Q_2$ with even ramification number, because $e_1=2$.

Take $\phi_2(x)=x^2-2$ as a representative of $\ty$. The $\phi_2$-adic development of $f(x)$ is
$$
f(x)= \phi_2(x)^2+(a+4)\phi_2(x)+(bx+c+2a+4).
$$  
By Proposition \ref{propertiesv} and (\ref{vrphir}), we have
$$
v_2(x)=1,\ v_2(\phi_2)=2,\ v_2(a+4)\ge 4, \ v_2(bx)=7,\ v_2(c+2a+4)\ge 6.
$$
Hence, according to $v(c+2a+4)=3$ or $v(c+2a+4)\ge 4$, the Newton polygon of second order, $N_2(f)$, is:

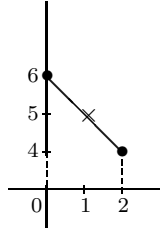
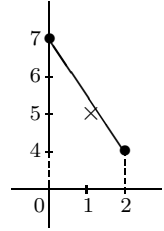
\begin{center}
\setlength{\unitlength}{5.mm}
\begin{picture}(12,6)
\put(-.15,3.85){$\bullet$}\put(1.85,1.85){$\bullet$}
\put(-1,1){\line(1,0){4}}\put(8,1){\line(1,0){4}}
\put(0,0){\line(0,1){1.1}}\put(0,1.9){\line(0,1){4.1}}
\multiput(0,1.1)(0,.25){4}{\vrule height2pt}
\multiput(2,.9)(0,.25){5}{\vrule height2pt}
\put(2,2){\line(-1,1){2}}\put(2,2.02){\line(-1,1){2}}
\put(-.4,.4){\begin{footnotesize}$0$\end{footnotesize}}
\put(.9,.4){\begin{footnotesize}$1$\end{footnotesize}}
\put(1.9,.4){\begin{footnotesize}$2$\end{footnotesize}}
\put(-.5,1.85){\begin{footnotesize}$4$\end{footnotesize}}
\put(-.5,2.85){\begin{footnotesize}$5$\end{footnotesize}}
\put(-.5,3.85){\begin{footnotesize}$6$\end{footnotesize}}
\put(.85,2.8){$\times$}
\put(-.1,3){\line(1,0){.2}}\put(-.1,2){\line(1,0){.2}}
\put(1,.9){\line(0,1){.2}}
\put(-1,-1){\begin{footnotesize}$v(c+2a+4)=3$\end{footnotesize}}
\put(8,-1){\begin{footnotesize}$v(c+2a+4)\ge4$\end{footnotesize}}
\put(8.85,4.85){$\bullet$}\put(10.85,1.85){$\bullet$}
\put(9,0){\line(0,1){1.1}}\put(9,1.9){\line(0,1){4.1}}
\multiput(9,1.1)(0,.25){4}{\vrule height2pt}
\multiput(11,.9)(0,.25){5}{\vrule height2pt}
\put(11,2){\line(-2,3){2}}\put(11,2.02){\line(-2,3){2}}
\put(8.6,.4){\begin{footnotesize}$0$\end{footnotesize}}
\put(9.9,.4){\begin{footnotesize}$1$\end{footnotesize}}
\put(10.9,.4){\begin{footnotesize}$2$\end{footnotesize}}
\put(8.5,1.85){\begin{footnotesize}$4$\end{footnotesize}}
\put(8.5,2.85){\begin{footnotesize}$5$\end{footnotesize}}
\put(8.5,3.85){\begin{footnotesize}$6$\end{footnotesize}}
\put(8.5,4.85){\begin{footnotesize}$7$\end{footnotesize}}
\put(9.85,2.85){$\times$}
\put(10,.9){\line(0,1){.2}}
\put(8.9,2){\line(1,0){.2}}\put(8.9,3){\line(1,0){.2}}\put(8.9,4){\line(1,0){.2}}
\end{picture}
\end{center}\be\be

If $v(c+2a+4)\ge4$, $N_2(f)$ is one-sided with slope $\la_2=-3/2$, and $R_2(f)(y)=y+1$. The type $\ty':=(x;-1/2,x^2-2;-3/2,y+1)$ is $f$-complete and $\ty_2(f)=\{\ty'\}$. We have $e_2=2,f_2=1$. Thus, $f(x)$ is irreducible over $\Z_2[x]$, and it generates an extension $L/\Q_2$ with $e(L/\Q_2)=e_1e_2=4$, $f(L/\Q_2)=f_0f_1f_2=1$. Moreover, $\ind_2(f)=1$, so that $\ind(f)=\ind_1(f)+\ind_2(f)=3$.

If $v(c+2a+4)=3$, $N_2(f)$ is one-sided with slope $\la_2=-1$, and $R_2(f)(y)=y^2+1=(y+1)^2$. The type $\ty':=(x;-1/2,x^2-2;-1,y+1)$ is not $f$-complete, $\ty_2(f)=\{\ty'\}$, and we need to pass to order three. We have $h_2=e_2=f_2=1$ and $\ind_2(f)=1$. Take $\phi_3(x)=x^2-2x-2$ as a representative of $\ty'$.
The $\phi_3$-adic development of $f(x)$ is
$$
f(x)= \phi_3(x)^2+(4x+a+8)\phi_3(x)+(b+2a+16)x+c+2a+12.
$$  
By Proposition \ref{propertiesv} and (\ref{vrphir}), we have
$$
v_3(x)=1,\ v_3(\phi_3)=3,\ v_3(4x)=5, \ v_3(c+2a+12)\ge 8,
$$
$$
v_3(4x+a+8)=\left\{\begin{array}{ll}
4,&\mbox{ if }v(a)=2,\\ 5,&\mbox{ if }v(a)\ge3,
\end{array}
\right.\quad
v_3((b+2a+16)x)=\left\{\begin{array}{ll}
\ge9,&\mbox{ if }v(a)=2,\\ 7,&\mbox{ if }v(a)\ge3.
\end{array}
\right.
$$
We have now three possibilities for the Newton polygon of third order

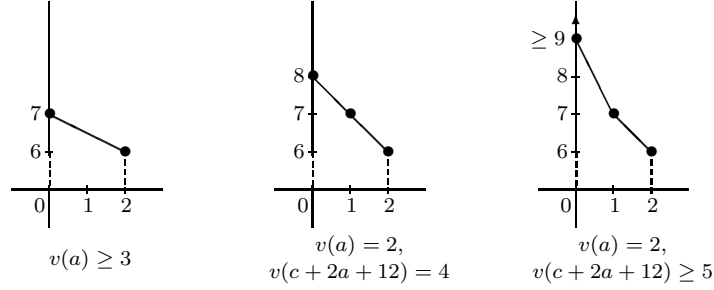
\begin{center}
\setlength{\unitlength}{5.mm}
\begin{picture}(16,6)
\put(-.15,2.85){$\bullet$}\put(1.85,1.85){$\bullet$}
\put(-1,1){\line(1,0){4}}
\put(0,0){\line(0,1){1.1}}\put(0,1.9){\line(0,1){4.1}}
\multiput(0,1.1)(0,.25){4}{\vrule height2pt}
\multiput(2,.9)(0,.25){5}{\vrule height2pt}
\put(2,2){\line(-2,1){2}}\put(2,2.02){\line(-2,1){2}}
\put(-.4,.4){\begin{footnotesize}$0$\end{footnotesize}}
\put(.9,.4){\begin{footnotesize}$1$\end{footnotesize}}
\put(1.9,.4){\begin{footnotesize}$2$\end{footnotesize}}
\put(-.5,1.85){\begin{footnotesize}$6$\end{footnotesize}}
\put(-.5,2.85){\begin{footnotesize}$7$\end{footnotesize}}
\put(-.1,3){\line(1,0){.2}}\put(-.1,2){\line(1,0){.2}}
\put(1,.9){\line(0,1){.2}}
\put(0,-1){\begin{footnotesize}$v(a)\ge3$\end{footnotesize}}

\put(6.85,3.85){$\bullet$}\put(7.85,2.85){$\bullet$}\put(8.85,1.85){$\bullet$}
\put(7,0){\line(0,1){1.1}}\put(7,1.9){\line(0,1){4.1}}
\put(6,1){\line(1,0){4}}
\multiput(7,1.1)(0,.25){4}{\vrule height2pt}
\multiput(9,.9)(0,.25){5}{\vrule height2pt}
\put(9,2){\line(-1,1){2}}\put(9,2.02){\line(-1,1){2}}
\put(6.6,.4){\begin{footnotesize}$0$\end{footnotesize}}
\put(7.9,.4){\begin{footnotesize}$1$\end{footnotesize}}
\put(8.9,.4){\begin{footnotesize}$2$\end{footnotesize}}
\put(6.5,1.85){\begin{footnotesize}$6$\end{footnotesize}}
\put(6.5,2.85){\begin{footnotesize}$7$\end{footnotesize}}
\put(6.5,3.85){\begin{footnotesize}$8$\end{footnotesize}}
\put(8,.9){\line(0,1){.2}}
\put(6.9,2){\line(1,0){.2}}\put(6.9,3){\line(1,0){.2}}\put(6.9,4){\line(1,0){.2}}
\put(5.5,-1){\begin{footnotesize}$\begin{array}{c}
v(a)=2,\\
v(c+2a+12)=4
\end{array}
$\end{footnotesize}}

\put(13.85,4.85){$\bullet$}\put(14.85,2.85){$\bullet$}\put(15.85,1.85){$\bullet$}
\put(14,0){\line(0,1){1.1}}\put(14,1.9){\line(0,1){4.1}}
\put(13,1){\line(1,0){4}}
\multiput(14,1.1)(0,.25){4}{\vrule height2pt}
\multiput(16,.9)(0,.25){5}{\vrule height2pt}
\put(16,2){\line(-1,1){1}}\put(16,2.02){\line(-1,1){1}}
\put(15,3){\line(-1,2){1}}\put(15,3.02){\line(-1,2){1}}
\put(13.6,.4){\begin{footnotesize}$0$\end{footnotesize}}
\put(14.9,.4){\begin{footnotesize}$1$\end{footnotesize}}
\put(15.9,.4){\begin{footnotesize}$2$\end{footnotesize}}
\put(13.5,1.85){\begin{footnotesize}$6$\end{footnotesize}}
\put(13.5,2.85){\begin{footnotesize}$7$\end{footnotesize}}
\put(13.5,3.85){\begin{footnotesize}$8$\end{footnotesize}}
\put(12.8,4.85){\begin{footnotesize}$\ge9$\end{footnotesize}}
\put(14,5){\vector(0,1){.7}}\put(15,.9){\line(0,1){.2}}
\put(13.9,2){\line(1,0){.2}}\put(13.9,3){\line(1,0){.2}}\put(13.9,4){\line(1,0){.2}}
\put(12.5,-1){\begin{footnotesize}$\begin{array}{c}
v(a)=2,\\v(c+2a+12)\ge5\end{array}
$\end{footnotesize}}
\end{picture}
\end{center}\be\be

If $v(a)\ge3$, $N_3(f)$ is one-sided with slope $\la_3=-1/2$, and $R_3(f)(y)=y+1$. The type $\ty'':=(x;-1/2,\phi_2(x);-1,\phi_3(x);-1/2,y+1)$ is $f$-complete and $\ty_3(f)=\{\ty''\}$. We have $e_3=2,f_3=1$. Thus, $f(x)$ is irreducible over $\Z_2[x]$, and it generates an extension $L/\Q_2$ with $e(L/\Q_2)=e_1e_2e_3=4$, $f(L/\Q_2)=f_0f_1f_2f_3=1$. Also, $\ind_3(f)=0$, so that $\ind(f)=\ind_1(f)+\ind_2(f)+\ind_3(f)=3$.

If $v(a)=2$ and $v(c+2a+12)=4$, $N_3(f)$ is one-sided with slope $\la_3=-1$, and $R_3(f)(y)=y^2+y+1$. The type $\ty'':=(x;-1/2,\phi_2(x);-1,\phi_3(x);-1,y^2+y+1)$ is $f$-complete and $\ty_3(f)=\{\ty''\}$. We have $e_3=1,f_3=2$. Thus, $f(x)$ is irreducible over $\Z_2[x]$, and it generates an extension $L/\Q_2$ with $e(L/\Q_2)=e_1e_2e_3=2$, $f(L/\Q_2)=f_0f_1f_2f_3=2$. Also, $\ind_3(f)=1$, so that $\ind(f)=\ind_1(f)+\ind_2(f)+\ind_3(f)=4$.

If $v(a)=2$ and $v(c+2a+12)\ge5$, $N_3(f)$ has two sides with slopes $\la_3\le-2$, $\la'_3=-1$, and $R_{\la_3}(f)(y)=R_{\la'_3}(f)(y)=y+1$. There are two types extending $\ty'$:
$$
\begin{array}{l}
\ty_1'':=(x;-1/2,\phi_2(x);-1,\phi_3(x);\la_3,y+1),\\ 
\ty_2'':=(x;-1/2,\phi_2(x);-1,\phi_3(x);-1,y+1).
\end{array}
$$Both types have $e_3=f_3=1$, they are both $f$-complete and $\ty_3(f)=\{\ty_1'',\ty_2''\}$. Thus, $f(x)$ has two irreducible factors of degree two over $\Z_2[x]$, and both generate extensions $L/\Q_2$ with $e(L/\Q_2)=2$, $f(L/\Q_2)=1$. Finally, $\ind_3(f)=1$, so that $\ind(f)=\ind_1(f)+\ind_2(f)+\ind_3(f)=4$.

In the final design of Montes' algorithm, this polynomial $f(x)$ is factorized already in order two. In the case $v(c+2a+4)=3$ the algorithm considers $\phi_3(x)=x^2-2x-2$ as a different representative of the type $\ty$, in order to avoid the increase of recursivity caused by the work in a higher order. See \cite{gmna} for more details on this optimization.

\subsection{Proof of the Theorem of the index}
Our first aim is to prove Theorem \ref{thindex} for $f(x)\in\zpx$ a monic irreducible polynomial of degree $n$, such that $\tt_r(f)$ is not empty. By Lemma \ref{empty}, $\tt_r(f)=\{\ty\}$ for some $\ty=(\phi_1(x);\cdots,\phi_r(x);\la_r,\psi_r(y))$, and $f(x)\ne\phi_s(x)$ for $s=1,\dots,r$. By Lemma \ref{factortype}, $f(x)$ is of type $\ty$ and $n=m_{r+1}\om_{r+1}(f)$. 

For $1\le s\le r$, let $E_s,H_s,d_s$ be the length, height and degree of the unique side of $N_s(f)$. Note that $E_s>0$, because $f(x)$ is of type $\ty$, and $0<H_s<\infty$, because $f(x)=\phi_s(x)$. By the Theorem of the residual polynomial, $R_{\la_r}(f)\sim \psi_r(y)^{a_r}$, for $a_r=\om_{r+1}(f)>0$.  

Let $\t\in\qpb$ a root of $f(x)$, $L=K(\t)$, and let us fix an embedding $\ff{r}[y]/\psi_r(y)\hookrightarrow \ff{L}$,
as in (\ref{embeddingr+1}). We introduce now some notations:
$$\as{1.8}
\begin{array}{l}
\nu_s:=v(\phi_s(\t))=\sum_{i=1}^se_if_i\cdots e_{s-1}f_{s-1}\,\dfrac{h_i}{e_1\dots e_i}, \mbox{ for all }1\le s\le r,\\
\nu_{\j}:= j_1\nu_1+\cdots+j_r\nu_r\in\Q, \mbox{ for all }
\j=(j_0,\dots,j_r)\in\N^{r+1},\\
\Phi(\j):=\,\dfrac{\t^{j_0}\phi_1(\t)^{j_1}\dots\phi_r(\t)^ {j_r}}{\pi^{\lfloor \nu_{\j}\rfloor}}\,\in\ol, \mbox{ for all }
\j=(j_0,\dots,j_r)\in\N^{r+1},\\
b_0:=f_0;\quad b_s:=e_sf_s,\ \mbox{ for }1\le s<r; \quad b_r:=e_rf_ra_r,\\
J:=\{\j\in\N^{r+1}\tq 0\le j_s<b_s, 0\le s\le r\}.
\end{array}
$$

\begin{lemma}\label{partialindex}
Let $\ol'$ be the sub-$\oo$-module of $\ol$ generated by $\{\Phi(\j)\tq \j\in J\}$. Then,
\begin{enumerate}
 \item  $\ol'$ is a free $\oo$-module of rank $n$, with basis $\{\Phi(\j)\tq \j\in J\}$,
\item $\oo[\t]\subseteq \ol'$, and $\left(\ol'\colon \oo[\t]\right)=q^{\sum_{\j\in J}\lfloor \nu_{\j}\rfloor}$.
\end{enumerate}
\end{lemma}

\begin{proof}
Clearly, $|J|=n$, and the numerators of $\Phi(\j)$, for $\j\in J$, are monic polynomials of degree $0,1,\dots,n-1$. Thus, the family  $\{\Phi(\j)\tq \j\in J\}$ is $\oo$-linearly independent. This proves item 1 and $\oo[\t]\subseteq \oo'_L$. Finally, since the numerators of $\Phi(\j)$, for $\j\in J$, are an $\oo$-basis of $\oo[\t]$: 
$$
\ol'/\oo[\t]\simeq \prod_{\j\in J}\pi^{-\lfloor \nu_{\j}\rfloor}\oo/\oo\simeq \prod_{\j\in J}\oo/\pi^{\lfloor \nu_{\j}\rfloor}\oo,
$$
and since $|\oo/\pi^a\oo|=q^a$, we get $\left(\oo'_L\colon \oo[\t]\right)=q^{\sum_{\j\in J}\lfloor \nu_{\j}\rfloor}$.
\end{proof}

Our next step is to prove that $\ol'$ is actually an order of $\oo_L$. To this end we need a couple of auxiliary results.
\begin{lemma}
Let \ $Q(x)=\sum_{\j=(j_0,\dots,j_{r-1},0)\in J}a_\j\, x^{j_0}\phi_1(x)^{j_1}\dots\phi_{r-1}(x)^ {j_{r-1}}$, for some $a_\j\in\oo$. Then, 
$$
v(Q(\t))=\min_{\j=(j_0,\dots,j_{r-1},0)\in J}\{v(a_\j)+\nu_{\j}\}.
$$
\end{lemma}

\begin{proof}
Since $\dg Q<m_r$, we have $v(Q(\t))=v_r(Q)/e_1\cdots e_{r-1}$ by Lemma \ref{typedegree} and Proposition
\ref{vpt}. Let us prove $v(a_\j)+\nu_{\j}\ge v_r(Q)/e_1\cdots e_{r-1}$ by induction on $r\ge 1$. If $r=1$ this is obvious because $v_1(Q)=\min\{v(a_\j)\}$. Let $r\ge 2$ and suppose the result is true for $r-1$. For each $0\le j_{r-1}<b_{r-1}$, consider the polynomial
$$
Q_{j_{r-1}}(x)=\sum_{(j_0,\dots,j_{r-2},0,0)\in J}a_\j\, x^{j_0}\phi_1(x)^{j_1}\dots\phi_{r-2}(x)^ {j_{r-2}},
$$ where $\j=(j_0,\dots,j_{r-2},j_{r-1},0)$ in each summand. Clearly, $$Q(x)=\sum_{0\le j_{r-1}<b_{r-1}}Q_{j_{r-1}}(x)\phi_{r-1}(x)^{j_{r-1}},
$$ is the $\phi_{r-1}$-adic development of $Q(x)$. By item 3 of Proposition \ref{propertiesv}, the Theorem of the polygon and the induction hypothesis we get
\begin{align*}
v_r(Q)/e_{r-1}=&\,\min_{0\le j_{r-1}<b_{r-1}}\{v_{r-1}(Q_{j_{r-1}})+j_{r-1}(v_{r-1}(\phi_{r-1})+|\la_{r-1}|)\}\\
=&\,\min_{0\le j_{r-1}<b_{r-1}}\{v_{r-1}(Q_{j_{r-1}})+j_{r-1}e_1\cdots e_{r-2}\nu_{r-1}\}\\
\le&\, e_1\cdots e_{r-2}\left(v(a_\j)+j_1\nu_1+\cdots+j_{r-2}\nu_{r-2}+j_{r-1}\nu_{r-1}\right).
\end{align*}
\end{proof}

\begin{lemma}
Let $\j=(j_0,\dots,j_r)\in\N^{r+1}$.
\begin{enumerate}
 \item For all $0\le s<r$,
\begin{align*}
\Phi(j_0,\dots,j_{s-1},j_s+b_s,j_{s+1},\dots,j_r)&\,=\pi^{\delta_{\j,s}}
\Phi(j_0,\dots,j_s,j_{s+1}+1,j_{s+2},\dots,j_r)\\&+\sum_{\j'=(j'_0,\dots,j'_s,0,\dots,0)\in J}
c_{\j,\,\j'}\Phi(\j+\j'),
\end{align*}
for some nonnegative integer $\delta_{\j,s}$ and some $c_{\j,\,\j'}\in\oo$.
\item $\Phi(j_0,\dots,j_{r-1},j_r+b_r)=\sum_{\j'\in J}
c_{\j,\,\j'}\Phi(\j+\j')$, for some $c_{\j,\,\j'}\in\oo$.
\end{enumerate}
\end{lemma}

\begin{proof}
Let $0\le s<r$, and denote $\phi_0(x)=x$, $\nu_0=0$, $e_0=1$. The polynomial $Q(x)=\phi_s(x)^{b_s}-\phi_{s+1}(x)$ has degree less than $m_{s+1}=b_sm_s$; hence, it admits a development $$Q(x)=\sum_{\j'=(j'_0,\dots,j'_s,0,\dots,0)\in J}a_{\j'}\, x^{j'_0}\phi_1(x)^{j'_1}\dots\phi_s(x)^{j'_s},$$ for some $a_{\j'}\in\oo$. 
If we substitute $\phi_s(x)^{b_s}=\phi_{s+1}(x)+Q(x)$ in $\Phi(j_0,\dots,j_{s-1},j_s+b_s,j_{s+1},\dots,j_r)$ we get the identity of item 1,
with $$\delta_{\j,s}=\lfloor \nu_\j+\nu_{s+1}\rfloor-\lfloor \nu_\j+b_s\nu_s\rfloor, \qquad 
c_{\j,\j'}=a_{\j'}\,\pi^{\lfloor \nu_\j+\nu_{\j'}\rfloor-\lfloor \nu_\j+b_s\nu_s\rfloor}.$$
Clearly, 
$$\nu_{s+1}=e_sf_s\nu_s+\dfrac{h_{s+1}}{e_1\cdots e_{s+1}}>b_s\nu_s,$$so that $\delta_{\j,s}\ge 0$. Also, $\nu_{s+1}>b_s\nu_s$ implies that $v(Q(\t))=b_s\nu_s$, and by the above lemma we have 
$v(a_{\j'})+\nu_{\j'}\ge b_s\nu_s$. This shows that $v(c_{\j,\j'})\ge 0$.

Item 2 follows by identical arguments, starting with $Q(x)=\phi_r(x)^{b_r}-f(x)$.
\end{proof}

\begin{proposition}\label{subring}
The $\oo$-module $\ol'$ is a subring of $\ol$. 
\end{proposition}

\begin{proof}
For all $\j,\,\j'\in J$ we have $\Phi(\j)\Phi(\j')=\pi^{\delta}\Phi(\j+\j')$, with $\delta=\lfloor\nu_\j+\nu_{\j'}\rfloor-\lfloor\nu_\j\rfloor-\lfloor\nu_{\j'}\rfloor\in\{0,1\}$. Thus, it is sufficient to check that $\Phi(\j)\in\ol'$, for all $\j\in\N^{r+1}$.

For any $0\le s\le r+1$, let $J_s:=\{\j=(j_0,\dots,j_r)\in\N^{r+1}\tq 0\le j_t<b_t,\ s\le t\le r\}$. Note that $J_0=J$, $J_{r+1}=\N^{r+1}$. Consider the condition
$$(i_s)\qquad \Phi(\j)\in\ol', \mbox{ for all }\j\in J_s.$$
By the definition of $\ol'$, the condition $(i_0)$ holds, and our aim is to show that $(i_{r+1})$ holds. Thus, it is sufficient to show that $(i_s)$ implies $(i_{s+1})$, for all $0\le s\le r$. Let us prove this implication by induction on $j_s$. Take $\j_0=(j_0,\dots,j_r)\in J_{s+1}$. If $0\le j_s<b_s$, condition $(i_{s+1})$ holds for $\j_0$. Let $j_s\ge b_s$ and suppose that $\Phi(j'_0,\dots,j'_{s-1},j,j'_{s+1},\dots,j'_r)\in\ol'$, for all $j'_0,\dots,j'_{s-1}\in\N$, all $0\le j<j_s$, and  all $0\le j'_t<b_t$, for $t>s$. 

By item 2 of the last lemma, applied to $\j=(j_0,\dots,j_{s-1},j_s-b_s,0,\dots,0)$:
\begin{equation}\label{item2}
 \Phi(j_0,\dots,j_{s-1},j_s-b_s,0,\dots,0,b_r)=\sum_{\j'\in J}
c_{\j,\j'}\Phi(\j+\j'),\ \mbox{ if }s<r,
\end{equation}
and $\Phi(j_0,\dots,j_r)=\sum_{\j'\in J}
c_{\j,\j'}\Phi(\j+\j')$, if $s=r$. In both cases, the terms $\Phi(\j+\j')$ belong to $\ol'$, because the $s$-th coordinate of $\j+\j'$ is $j_s-b_s+j'_s<j_s$. In particular, if $s=r$ we are done. If $s<r$ we apply
 item 1 of the last lemma to $\j=(j_0,\dots,j_{s-1},j_s-b_s,j_{s+1},\dots,j_r)$ and we get
$$\Phi(\j_0)=\pi^{\delta_{\j,s}}
\Phi(j_0,\dots,j_s-b_s,j_{s+1}+1,j_{s+2},\dots,j_r)+\sum_{\j'=(j'_0,\dots,j'_s,0,\dots,0)\in J}
c_{\j,\j'}\Phi(\j+\j').$$
The last sum belongs to $\ol'$ by the same argument as above. Thus, we need only to show that the term $\Phi(j_0,\dots,j_s-b_s,j_{s+1}+1,j_{s+2},\dots,j_r)$ belongs to $\ol'$ too. If $j_{s+1}+1<b_{s+1}$, this follows from the induction hypothesis. If $j_{s+1}+1=b_{s+1}$ and $s=r-1$, this is clear by (\ref{item2}). Finally, if $j_{s+1}+1=b_{s+1}$ and $s<r-1$, we can apply item 1 of the last lemma again to see that it is sufficient to check that $\Phi(j_0,\dots,j_s-b_s,0,j_{s+2}+1,\dots,j_r)$ belongs to $\ol'$. In this iterative process we conclude either by (\ref{item2}), or because we find some $j_t+1<b_t$.
\end{proof}

We need still some auxiliary lemmas. The first one is an easy remark about integral parts. 
\begin{lemma}\label{int}
For all $x\in\R$ and $e\in\Z_{>0}$, we have $\sum_{0\le k<e}\Big\lfloor\dfrac{x+k}e\Big\rfloor=\lfloor x\rfloor$.
\end{lemma}

\begin{proof}
The identity is obvious when $x$ is an integer, $0\le x<e$, because
$\Big\lfloor\dfrac{x+k}e\Big\rfloor=1$ for the $x$ values of $k$ such that $e-x\le k<e$, and it is zero otherwise.

Write $x=n+\epsilon$, with $n=\lfloor x\rfloor$ and $0\le\epsilon<1$; clearly, $\lfloor (x+k)/e\rfloor=\lfloor (n+k)/e\rfloor$, because $\epsilon/e<1/e$. Consider the division with remainder,
$n=Qe+r$, with $0\le r<e$. Then,
$$
\sum_{0\le k<e}\Big\lfloor\dfrac{n+k}e\Big\rfloor=\sum_{0\le k<e}\left(Q+\Big\lfloor\dfrac{r+k}e\Big\rfloor\right)=eQ+r=n.
$$
\end{proof}

\begin{lemma}\label{jprime}
Take $e_0=1$, $h_0=0$ by convention. Every $\j\in\N^{r+1}$ can be written in a unique way: $\j=\j'+\j''$, with $\j'$, $\j''$ belonging respectively to the two sets:
$$\as{1.4}
\begin{array}{l}
J':=\{\j'=(j'_0,\dots,j'_r)\in\N^{r+1}\tq 0\le j'_s<e_s,\mbox{ for all }0\le s\le r\}\subseteq J,\\
J'':=\{\j''=(j''_0,\dots,j''_r)\in\N^{r+1}\tq j''_s\equiv 0\md{e_s},\mbox{ for all }0\le s\le r\}.
\end{array}
$$
Then, for any $\j''=(k_0,e_1k_1,\dots,e_rk_r)\in J''$, there is a unique $\j'=(j'_0,\dots,j'_r)\in J'$ such that $v(\Phi(\j'+\j''))=0$. Moreover, $j'_r=0$, and $j'_s$ depends only on $k_{s+1},\dots,k_r$,  for $0\le s<r$.
\end{lemma}

\begin{proof}
 For any $\j\in\N^{r+1}$ denote by $\la_\j$ the positive integer
\begin{align*}
\la_\j:=e_1\cdots e_r\,\nu_\j=&\,\sum_{s=1}^rj_s\sum_{i=1}^se_if_i\cdots e_{s-1}f_{s-1}e_{i+1}\cdots e_rh_i
\\=&\,\sum_{i=1}^r\left(\sum_{t=i}^rj_te_if_i\cdots e_{t-1}f_{t-1}\right)e_{i+1}\cdots e_rh_i.
\end{align*}

Clearly,
\begin{equation}\label{smallvalues}
v(\Phi(\j))=\nu_\j-\lfloor \nu_\j\rfloor=\dfrac{\la_\j}{e_1\cdots e_r}-\Big\lfloor \dfrac{\la_\j}{e_1\cdots e_r}\Big\rfloor.
\end{equation}
Thus, $v(\Phi(\j))=0$ if and only if $\la_\j\equiv 0\md{e_1\cdots e_r}$. Define now, for each $0\le s\le r$,
$$
\la_{\j,s}:=j_sh_se_{s+1}\cdots e_r+\sum_{i=s+1}^r\left(\sum_{t=i}^rj_te_if_i\cdots e_{t-1}f_{t-1}\right)e_{i+1}\cdots e_rh_i.
$$Note that $\la_{\j,s}$ depends only on $j_s,\dots,j_r$, and $\la_{\j,0}=\la_\j$, $\la_{\j,r}=j_rh_r$. Clearly, 
$$
\la_{\j,s}-\la_{\j,s+1}=j_sh_se_{s+1}\cdots e_r+\left(\sum_{t=s+2}^rj_te_{s+1}f_{s+1}\cdots e_{t-1}f_{t-1}\right)e_{s+2}\cdots e_rh_{s+1},
$$for all $0\le s< r$. In  particular, $\la_{\j,s}\equiv\la_{\j,s+1}\md{e_{s+1}\cdots e_r}$, and 
$$
\la_\j\equiv 0\md{e_1\cdots e_r}\sii \la_{\j,s}\equiv 0\md{e_s\cdots e_r},\mbox{ for all }1\le s\le r. 
$$
The condition $\la_{\j,r}\equiv 0\md{e_r}$ is equivalent to $j_r\equiv0\md{e_r}$. On the other hand, for $1\le s<r$, the condition 
$\la_{\j,s}\equiv 0\md{e_s\cdots e_r}$ is equivalent to
$$\as{1.4}
\begin{array}{l}
\la_{\j,s+1}\equiv 0\md{e_{s+1}\cdots e_r},\mbox{ and }\\
j_sh_s+\left(\sum_{t=s+2}^rj_t(f_{s+1}\cdots f_{t-1})(e_{s+2}\cdots e_{t-1})\right)h_{s+1}+\dfrac{\la_{\j,s+1}}{e_{s+1}\cdots e_r}\equiv 0\md{e_s}.
\end{array}
$$Thus, the class of $j_s$ modulo $e_s$ is uniquely determined, and it depends only on $j_{s+1},\dots,j_r$.
\end{proof}

\begin{corollary}\label{gammas}
Let $\kappa=(k_0,\dots,k_r)\in\N^{r+1}$, and let $\j=\j'+(k_0,e_1k_1,\dots,e_rk_r)$, where  $\j'$ is the unique element in $J'$ such that $v(\Phi(\j))=0$. Then,
$$\Phi(\j)=\t^{k_0}\ga_1(\t)^{k_1}\cdots \ga_r(\t)^{k_r}\ga_1(\t)^{i_1}\cdots \ga_{r-1}(\t)^{i_{r-1}},
$$ 
for some integers $i_1,\dots,i_{r-1}$. Moreover, each $i_s$ depends only on $k_{s+1},\dots,k_r$.
\end{corollary}

\begin{proof}
By Lemma \ref{jprime}, $\j=(k_0,j'_1+e_1k_1,\dots,j'_{r-1}+e_{r-1}k_{r-1},e_rk_r)$. By (\ref{phii}), $$\ga_s(\t)^{k_s}=\pi^{n_{s,0}}\phi_1(\t)^{n_{s,1}}\cdots \phi_s(\t)^{e_sk_s},$$
for all $1\le s\le r$, with integers $n_{s,i}$ that depend only on $k_s$. Hence,
$$\Phi(\j)\t^{-k_0}\ga_1(\t)^{-k_1}\cdots\ga_r(\t)^{-k_r}=\pi^{n_0}\phi_1(\t)^{n_1}\cdots \phi_{r-1}(\t)^{n_{r-1}},$$
for integers $n_s$ that depend only on $j'_s$ and $k_{s+1},\dots,k_r$; hence they depend only on $k_{s+1},\dots,k_r$.
By Corollary \ref{zerovalue}, $v(\pi^{n_0}\phi_1(\t)^{n_1}\cdots \phi_{r-1}(\t)^{n_{r-1}})=0$, and by Propositions \ref{vpt} and \ref{vrphii}
we have  $v_r(\pi^{n_0}\phi_1(x)^{n_1}\cdots \phi_{r-1}(x)^{n_{r-1}})=0$. By Lemma \ref{gammas0}, this rational function can be expressed as a product $\ga_1(x)^{i_1}\cdots \ga_{r-1}(x)^{i_{r-1}}$, with integers $i_1,\dots,i_{r-1}$ such that each $i_s$ depends only on $n_s,\dots,n_{r-1}$, that is, on $k_{s+1},\dots,k_r$.
\end{proof}

\begin{corollary}\label{samevalue}
Let $\j_1=\j'_1+\j''$, $\j_2=\j'_2+\j''$, for some $\j'_1,\j'_2\in J'$, $\j''\in J''$. Then, $v(\Phi(\j_1))=v(\Phi(\j_2))$ if and only if $\j_1=\j_2$. In particular, $$\{v(\Phi(\j))\tq\j\in J'\}=\{k/e_1\cdots e_r\tq 0\le k<e_1\cdots e_r\}.$$
\end{corollary}

\begin{proof}
Let $\j_1=(j_{1,0},\dots j_{1,r})$, $\j_2=(j_{2,0},\dots j_{2,r})$.
With the notations of Lemma \ref{jprime}, (\ref{smallvalues}) shows that 
\begin{align*}
v(\Phi(\j_1))=v(\Phi(\j_2))&\sii\la_{\j_1}\equiv \la_{\j_2}\md{e_1\cdots e_r}\\&\sii \la_{\j_1,s}\equiv \la_{\j_2,s}\md{e_s\cdots e_r}, \mbox{ for all }1\le s\le r.
\end{align*}
For $s=r$ this is equivalent to $j_{1,r}=j_{2,r}$. Also, if $j_{1,t}=j_{2,t}$ for all $t>s$, then
$\la_{\j_1,s}-\la_{\j_2,s}=(j_{1,s}-j_{2,s})h_se_{s+1}\cdots e_r$, so that $\la_{\j_1,s}\equiv \la_{\j_2,s}\md{e_s\cdots e_r}$ is equivalent to $j_{1,s}=j_{2,s}$.

Finally, it is clear that $|J'|=e_1\cdots e_r$, and we have just shown that the elements $v(\Phi(\j))$, $\j\in J'$, take $e_1\cdots e_r$ different values, all of them contained in the set $\{k/e_1\cdots e_r\tq 0\le k<e_1\cdots e_r\}$ by (\ref{smallvalues}).
\end{proof}

\begin{proposition}\label{oo=}
If $\tt$ is $f$-complete, then $\oo'_L=\oo_L$. Moreover, the family of all $\Phi(\j)\Phi(\j')$, for $\j\in J_0:=\{\j\in J\tq v(\Phi(\j))=0\}$ and $\j'\in J'$, is an $\oo$-basis of $\oo_L$. Finally, if $L/K$ is ramified, there exists $\j'\in J'$ such that
 $v_L(\Phi(\j'))=1$, so that $\m_L=\Phi(\j')\oo_L$.
\end{proposition}

\begin{proof}
Corollary \ref{ramr} shows that $e(L/K)=e_1\cdots e_r$, $f(L/K)=f_0f_1\cdots f_r$. 
By Corollary \ref{samevalue}, we have $\{v_L(\Phi(\j'))\tq\j'\in J'\}=\{0,1,\dots,e(L/K)-1\}$; in particular, if $e(L/K)>1$, there exists $\j'\in J'$ such that $v_L(\Phi(\j'))=1$. By Lemma \ref{jprime}, $|J_0|=f_0f_1\cdots f_r=\dim_{\ff{K}}\ff{L}$, and each $\j\in J_0$ is parameterized by a sequence $(k_0,\dots,k_r)$, with $0\le k_s<f_s$ for all $0\le s\le r$.
By item 4 of Proposition \ref{vqtr}, $\ff{L}=\ff{K}(\gb0,\dots,\gb{r})$, where $\ga_0(x):=x$. Recall that $z_i=\gb{i}$ for all $0\le i\le r$, under our identification of $\ff{r+1}:=\ff{r}[y]/\psi_r(y)$ with $\ff{L}$.

By Corollary \ref{gammas},
$$\overline{\Phi(\j)}=z_0^{k_0}z_1^{k_1+i_1}\cdots (z_{r-1})^{k_{r-1}+i_{r-1}} z_r^{k_r}=
z_0^{k_0}z_1^{k_1}\g_2(k_2,\dots,k_r)\cdots \g_r(k_r),$$
where $\g_s(k_s,\dots,k_r):=z_s^{k_s}(z_{s-1})^{i_{s-1}}$, for $s\ge 2$. Now,
the family of all $\overline{\Phi(\j)}$ for $\j\in J_0$ is an $\ff{K}$-basis of $\ff{L}$. In fact, the set of all $\g_r(k_r)$ for $0\le k_r<f_r$, is a $\ff{r}$-basis of $\ff{L}=\ff{r+1}$, because they are obtained from the basis $z_r^{k_r}$, just by multiplying every element by the nonzero scalar $z_{r-1}^{i_{r-1}}\in\ff{r}$, which depends only on $k_r$. Then, the set of all $\g_{r-1}(k_{r-1},k_r)\g_r(k_r)$ for $0\le k_{r-1}<f_{r-1}$, $0\le k_r<f_r$, is a $\ff{r-1}$-basis of $\ff{L}$, because they are obtained from the basis $(z_{r-1})^{k_{r-1}}\g_r(k_r)$, just by multiplying every element by the nonzero scalar $z_{r-2}^{i_{r-2}}\in\ff{r-1}$, which depends only on $k_{r-1},k_r$, etc.

Therefore, the $e(L/K)f(L/K)$ elements $\Phi(\j)\Phi(\j')$, $\j\in J_0$, $\j'\in J'$, are a $\oo$-basis of $\ol$. By Proposition \ref{subring}, all these elements are contained in $\ol'$, and we have necessarily $\ol'=\ol$. 
\end{proof}

\begin{proof}[Proof of Theorem \ref{thindex}]
Suppose first that $f(x)\in\zpx$ is a monic irreducible polynomial, such that $\tt_r(f)=\{\ty\}$. In this case we have built an order $\oo[\t]\subseteq \ol'\subseteq \ol$, such that
\begin{equation}\label{end1}
\left(\ol\colon \oo[\t]\right)=q^{\ind(f)},\quad \left(\ol'\colon \oo[\t]\right)=q^{\sum_{\j\in J}\lfloor\nu_\j\rfloor},
\end{equation}
the last equality by Lemma \ref{partialindex}. Therefore, in order to prove item 1 of Theorem \ref{thindex} it is sufficient to show that
\begin{equation}\label{end2}
f_0\sum_{\j=(0,j_1,\dots,j_r)\in J}\lfloor\nu_\j\rfloor=\ind_1(f)+\cdots+\ind_r(f).
\end{equation}

Let us prove this identity by induction on $r\ge 1$. For $r=1$ we have $\ind_1(f)=f_0\ind(N_1(f))$, and $j\nu_1=j|\la_1|=y_j(N_1(f))$; thus, (\ref{end2}) was proved already in (\ref{combi}). From now on, let $r\ge 2$. Both sides of the identity depend only on $a_r$, $f_0$ and the vectors
$\mathbf{e}=(e_1,\dots,e_r)$, $\mathbf{f}=(f_1,\dots,f_{r-1})$, $\mathbf{h}=(h_1,\dots,h_r)$. Recall that
$$
\nu_s=\nu_s(\mathbf{e},\mathbf{f},\mathbf{h}):=\sum_{i=1}^se_if_i\cdots e_{s-1}f_{s-1}\,\dfrac{h_i}{e_1\dots e_i}.
$$If we denote $\mathbf{e'}=(e_2,\dots,e_r)$, $\mathbf{f'}=(f_2,\dots,f_{r-1})$, $\mathbf{h'}=(h_2,\dots,h_r)$, it is easy to check that, for every $2\le s\le r$:
\begin{equation}\label{recursive}
\nu_s(\mathbf{e},\mathbf{f},\mathbf{h})-\frac{m_s}{m_2}f_1h_1=\frac 1{e_1}\nu_{s-1}(\mathbf{e'},\mathbf{f'},\mathbf{h'}).
\end{equation}
Let us show that the identity 
$$f_0\sum_{\j=(0,j_1,\dots,j_r)\in J}\Big\lfloor\sum_{s=1}^rj_s\nu_s(\mathbf{e},\mathbf{f},\mathbf{h})\Big\rfloor=\ind_1(f)+\cdots+\ind_r(f),$$
holds for any choice of $a_r$, $f_0$ and $\mathbf{e},\mathbf{f},\mathbf{h}$, under the assumption that the same statement is true for $r-1$. Write $j_1=je_1+k$, with $0\le j<f_1$, $0\le k<e_1$, and let $0\le s_k<e_1$ be determined by $kh_1\equiv s_k\md{e_1}$. Then, by (\ref{recursive}),
\begin{align*}
\Big\lfloor\sum_{s=1}^r&j_s\nu_s(\mathbf{e},\mathbf{f},\mathbf{h})\Big\rfloor=\Big\lfloor jh_1+k\frac{h_1}{e_1}+\sum_{s=2}^rj_s\nu_s(\mathbf{e},\mathbf{f},\mathbf{h})\Big\rfloor=\\
&=\sum_{s=2}^rj_s\frac{m_s}{m_2}f_1h_1+jh_1+\Big\lfloor k\frac{h_1}{e_1}+\frac1{e_1}\sum_{s=2}^rj_s\nu_{s-1}(\mathbf{e'},\mathbf{f'},\mathbf{h'})\Big\rfloor=\\
&=\sum_{s=2}^rj_s\frac{m_s}{m_2}f_1h_1+jh_1+\Big\lfloor  k\frac{h_1}{e_1}\Big\rfloor+\Big\lfloor \frac{s_k}{e_1}+\frac1{e_1}\sum_{s=1}^{r-1}j_{s+1}\nu_s(\mathbf{e'},\mathbf{f'},\mathbf{h'})\Big\rfloor.
\end{align*}
Therefore, it is sufficient to check the two identities:
$$
f_0\sum_{\begin{array}{c}
\mbox{\begin{tiny}$(0,0,j_2,\dots,j_r)\in J$\end{tiny}}\\
\mbox{\begin{tiny}$0\le j<f_1,0\le k<e_1$\end{tiny}}\end{array}
}\left(\sum_{s=2}^rj_s\frac{m_s}{m_2}f_1h_1+jh_1+\Big\lfloor  k\frac{h_1}{e_1}\Big\rfloor\right)=\ind_1(f),
$$$$
f_0\sum_{
\begin{array}{c}
\mbox{\begin{tiny}$(0,0,j_2,\dots,j_r)\in J$\end{tiny}}\\
\mbox{\begin{tiny}$0\le j<f_1,0\le k<e_1$\end{tiny}}\end{array}
}\Big\lfloor \frac{s_k}{e_1}+\frac1{e_1}\sum_{s=1}^{r-1}j_{s+1}\nu_s(\mathbf{e'},\mathbf{f'},\mathbf{h'})\Big\rfloor=\ind_2(f)+\cdots+\ind_r(f).
$$
The integers $0\le i<(n/f_0)$ are in $1$-$1$ correspondence with the vectors $(0,j_1,\dots,j_r)$ in $J$ via:
$$
i=j_1+j_2(m_2/f_0)+\cdots+j_r(m_r/f_0).
$$Therefore, the left-hand side of the first identity is equal to $f_0\sum_{0\le i<(n/f_0)}\lfloor i\frac{h_1}{e_1}\rfloor$, which is equal to $\ind_1(f)$ by (\ref{combi}). The second identity follows from the induction hypothesis. In fact, the set $\{s_k\tq 0\le k<e_1\}$ coincides with $\{0,1,\dots,e_1-1\}$, and by Lemma \ref{int} the left-hand side of the identity is equal to 
$$f_0f_1\sum_{(0,0,j_2,\dots,j_r)\in J}\Big\lfloor \sum_{s=1}^{r-1}j_{s+1}\nu_s(\mathbf{e'},\mathbf{f'},\mathbf{h'})\Big\rfloor.
$$

Let us prove now the second part of the theorem. Suppose that $\ind(f)=\ind_1(f)+\cdots +\ind_r(f)$. Let $\phi_{r+1}(x)$ be the representative of $\ty$; if $f(x)=\phi_{r+1}(x)$, we have directly $\ind_{r+1}(f)=0$ because $N_{r+1}(f)$ is a side of slope $-\infty$. If $f(x)\ne\phi_{r+1}(x)$, then $\tt_{r+1}(f)\ne\emptyset$ by Lemma \ref{empty}, and $\ind_{r+1}(f)=0$ by item 1 of the theorem in order $r+1$. 

Conversely, suppose that $\ind_{r+1}(f)=0$. Lemma \ref{zeroind} shows that all types in $\ty_{r+1}(f)$ are $f$-complete, and Lemma \ref{uniquesprout} shows that all types in $\tt_{r+1}(f)$ are $f$-complete too. If $\ty$ is $f$-complete, we have $\oo'_L=\oo_L$ by Proposition \ref{oo=}, and we get $\ind(f)=\ind_1(f)+\cdots +\ind_r(f)$, by (\ref{end1}) and (\ref{end2}). If $\ty$ is not $f$-complete, we have in particular $f(x)\ne\phi_{r+1}(x)$, and we can extend $\ty$ in a unique way to a type $\ty'=(\tilde{\ty};\la_{r+1},\psi_{r+1}(y))$ of order $r+1$, which is $f$-complete by our assumption. By Proposition \ref{oo=}, (\ref{end1}) and (\ref{end2}), applied to $\ty'$ in order $r+1$, we get $\ind(f)=\ind_1(f)+\cdots +\ind_r(f)+\ind_{r+1}(f)$ as above. Since $\ind_{r+1}(f)=0$, we have $\ind(f)=\ind_1(f)+\cdots +\ind_r(f)$, as desired.
This ends the proof of the theorem in the particular case we were dealing with. 

Let us prove now the theorem in the other instances where $f(x)$ is irreducible: $f(x)=\phi_s(x)$ for the representative $\phi_s(x)$ of some type of order $s-1\le r-1$  (cf. Lemma \ref{empty}). In this case, $\ind_s(f)=0$ because $N_s(f)$ is a side of slope $-\infty$. Also, if $s<r$ we have $\ind_{s+1}(f)=\cdots=\ind_r(f)=0$ by definition, because $\ty_s(f)=\emptyset$ by Lemma \ref{empty}. Since $f(x)\ne\phi_1(x),\dots,\phi_{s-1}(x)$, we have $\tt_{s-1}(f)\ne\emptyset$ and we can apply the theorem in order $s-1$:
$$
\ind(f)=\ind_1(f)+\cdots+\ind_{s-1}(f)=\ind_1(f)+\cdots+\ind_r(f).
$$This proves  both statements of the theorem and it ends the proof of the theorem when $f(x)$ is irreducible.

In the general case, if $f(x)=F_1(x)\cdots F_k(x)$ is the factorization of $f(x)$ into a product of monic irreducible polynomials,
we have by definition
$$
\ind(f)=\sum_{i=1}^k\ind(F_i)+\sum_{1\le i<j\le k}v(\res(F_i,F_j)).
$$ 
By Lemma \ref{indres}, an analogous relationship holds for every $\ind_s(f)$, $1\le s\le r$. Hence, item 1 of the theorem holds by the theorem applied to each $\ind(F_i)$, and by Theorem \ref{resultant}. Let us prove now item 2. By Lemma \ref{indres},  $\ind_{r+1}(f)=0$ if and only if $\ind_{r+1}(F_i)=0$ and $\res_{r+1}(F_i,F_j)=0$, for all $i$ and all $j\ne i$. By the theorem in the irreducible case and Theorem  \ref{resultant}, this is equivalent to $\ind(f)=\ind_1(f)+\cdots+\ind_r(f)$.
\end{proof}

\end{document}